\documentclass[11pt]{article}
\usepackage{amsmath,amssymb,amsthm,bm,graphicx,color,epsfig,enumerate,caption}
\usepackage{mathtools}
\usepackage{tikz}
\usetikzlibrary{shapes,arrows,matrix,patterns,positioning}
\usepackage[norelsize,boxed,linesnumbered,vlined,algo2e]{algorithm2e} 
\usepackage{algorithm}
\usepackage{algorithmic}
\usepackage{amssymb}
\usepackage{epstopdf}
\usepackage{float}
\usepackage{hyperref}
\usepackage{booktabs}
\usepackage{epsf}
\usepackage{epsfig}
\usepackage{multirow}
\usepackage{array}
\usepackage{bm}
\usepackage{cases}
\usepackage{catchfilebetweentags}
\usepackage{comment}

\newtheorem{theorem}{Theorem}[section]

\newtheorem{definition}[theorem]{Definition}

\newcommand{\T}{\mathcal{T}}

\newcommand{\sgn}{\text{sgn}}

\newcommand{\cA}{\mathcal{A}}
\newcommand{\cD}{\mathcal{D}}

\newcommand{\WD}{\mathcal{WD}}
\newcommand{\WS}{\mathcal{WS}}
\newcommand{\defeq}{\vcentcolon=}
\newcommand{\E}{\mathrm{E}}

\newcommand{\Var}{\mathrm{Var}}

\usepackage{mathtools}

\usepackage{catchfilebetweentags}

\makeatletter
\def\CatchFBT@Fin@l#1[#2]{%
   \begingroup
      \makeatletter #2%
      \scantokens\expandafter{%
         \expandafter\CatchFBT@tok\expandafter{\the\CatchFBT@tok}}%
      \CatchFBT@IsAToken{#1}
         {\global#1\expandafter{\the\CatchFBT@tok}}
         {\xdef#1{\the\CatchFBT@tok}}%
      \ifx\CatchFBT@tok#1\else\global\CatchFBT@tok{}\fi
   \endgroup
}% \CatchFBT@Final
\makeatother

\begin{document}

\title{A Fast Algorithm for Multiresolution Mode Decomposition}
\author{Gao Tang$^\dagger$, Haizhao Yang$^*$\\
  \vspace{0.1in}\\
  $^\dagger$ Department of Mechanical Engineering and Materials Science, Duke University, US\\
  $^*$Department of Mathematics, National University of Singapore, Singapore\\
}

\date{November 2017}
\maketitle

\begin{abstract}
\emph{Multiresolution mode decomposition} (MMD) is an adaptive tool to analyze a time series $f(t)=\sum_{k=1}^K f_k(t)$, where $f_k(t)$ is a \emph{multiresolution intrinsic mode function} (MIMF) of the form 
\begin{eqnarray*}
f_k(t)&=&\sum_{n=-N/2}^{N/2-1} a_{n,k}\cos(2\pi n\phi_k(t))s_{cn,k}(2\pi N_k\phi_k(t))\\&&+\sum_{n=-N/2}^{N/2-1}b_{n,k} \sin(2\pi n\phi_k(t))s_{sn,k}(2\pi N_k\phi_k(t))
\end{eqnarray*}
with time-dependent amplitudes, frequencies, and waveforms. The multiresolution expansion coefficients $\{a_{n,k}\}$, $\{b_{n,k}\}$, and the shape function series $\{s_{cn,k}(t)\}$ and $\{s_{sn,k}(t)\}$ provide innovative features for adaptive time series analysis. The MMD aims at identifying these MIMF's (including their multiresolution expansion coefficients and shape functions series) from their superposition. However, due to the lack of efficient algorithms to solve the MMD problem, the application of MMD for large-scale data science is prohibitive, especially for real-time data analysis. This paper proposes a fast algorithm for solving the MMD problem based on recursive diffeomorphism-based spectral analysis (RDSA). RDSA admits highly efficient numerical implementation via the nonuniform fast Fourier transform (NUFFT); its convergence and accuracy can be guaranteed theoretically. Numerical examples from synthetic data and natural phenomena are given to demonstrate the efficiency of the proposed method.

\end{abstract}



\section{Introduction}
\label{sec:intro}

Oscillatory data analysis is important for a considerate number of real world applications such as medical electrocardiography (ECG) reading \cite{HauBio2,Pinheiro2012175,7042968}, atomic crystal images in physics \cite{Crystal,LuWirthYang:2016}, mechanical engineering \cite{Eng2,ME}, art investigation \cite{Canvas,Canvas2}, geology \cite{GeoReview,SSCT,977903}, imaging \cite{4685952}, etc. One single record of the data might contain several principal components with different oscillation patterns. The goal is to extract these components and analyze them individually. A typical model in \emph{mode decomposition} is to assume that a signal $f(t)$ defined on $[0,1]$ consists of several oscillatory modes like
\begin{equation}
\label{P1}
f(t)=\sum_{k=1}^K \alpha_k(t) e^{2\pi i N_k \phi_k(t)}+r(t),
\end{equation}
where $\alpha_k(t)$ is a smooth, positive, and non-oscillatory instantaneous amplitude, $N_k \phi_k(t)$ is a smooth and strictly increasing instantaneous phase, $N_k\phi_k'(t)$ is the instantaneous frequency, and $r(t)$ is the residual signal. Methods for the mode decomposition problem in \eqref{P1} include the empirical mode decomposition approach \cite{Huang1998,doi:10.1142/S1793536909000047}, synchrosqueezed transforms \cite{Daubechies2011,BEHERA2016}, time-frequency reassignment methods \cite{Auger1995,Chassande-Mottin2003}, adaptive optimization \cite{VMD,Hou2012}, iterative filters \cite{YangWang,Cicone2016384}, etc.

In complicated applications, sinusoidal oscillatory patterns may lose important physical information \cite{PhysicalAnal,Hau-Tieng2013,1DSSWPT,Hou2016,HZYregression}, which motivates the introduction of shape functions $\{s_k(t)\}_{1\leq k\leq K}$ and the \emph{generalized mode decomposition} as follows
\begin{equation}
\label{P2}
f(t)= \sum_{k=1}^K \alpha_k(t) s_k(2 \pi N_k \phi_k(t))+r(t),
\end{equation}
where $\{s_k(t)\}_{1\leq k\leq K}$ are $2\pi$-periodic and zero-mean shape functions with a unit norm in $L^2([0,2\pi])$, $\alpha_k(t)$ and $\phi_k(t)$ are the same functions as in \eqref{P1}. One of such examples is the photoplethysmogram (PPG) signal (see Figure \ref{fig:ppg0}) in medical study. Shape functions reflect complicated evolution patterns of the signal $f(t)$ and contain valuable information for monitoring the health condition of patients \cite{doi:10.1097/ALN.0b013e31816c89e1,LI201589,1615746,7529221,cite-key}.

\begin{figure}
  \vspace{-0.5cm}
  \begin{center}
  \begin{tabular}{ccc}
 &  \includegraphics[height=1.2in]{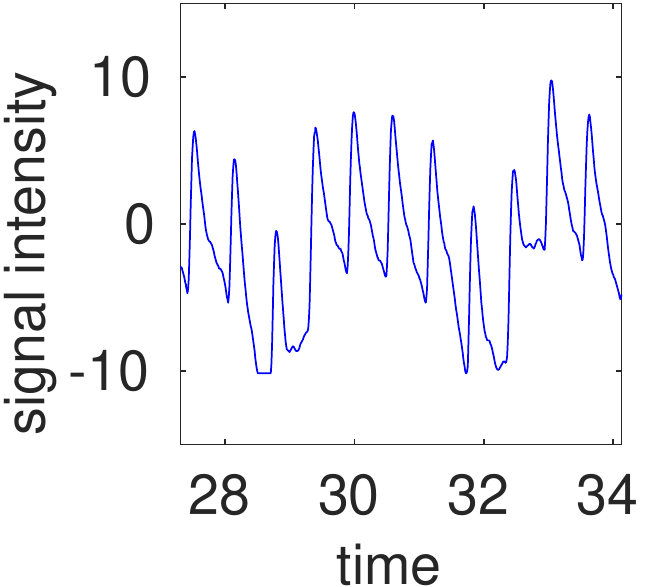}&\\
         \includegraphics[height=1.2in]{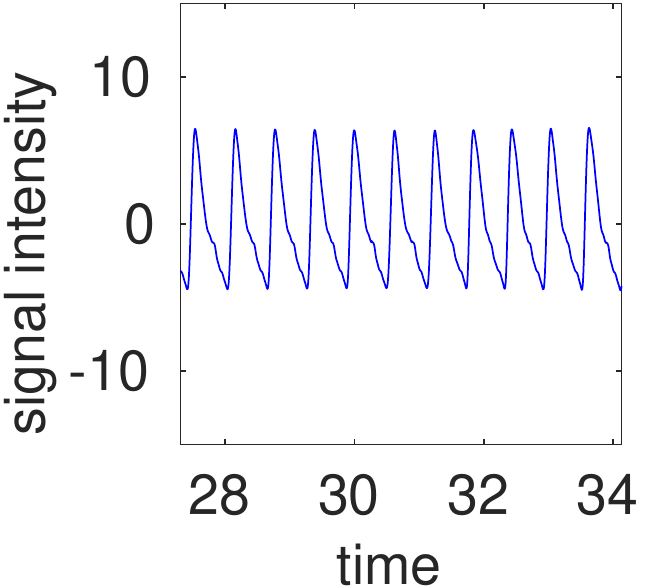}  &
      \includegraphics[height=1.2in]{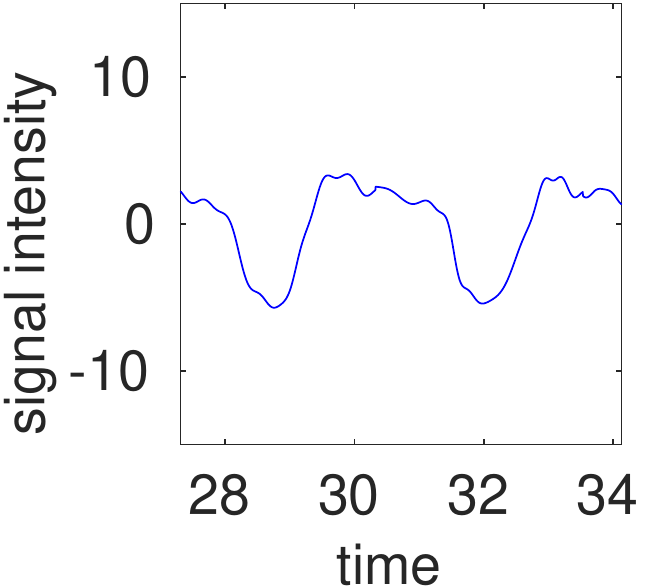}&
      \includegraphics[height=1.2in]{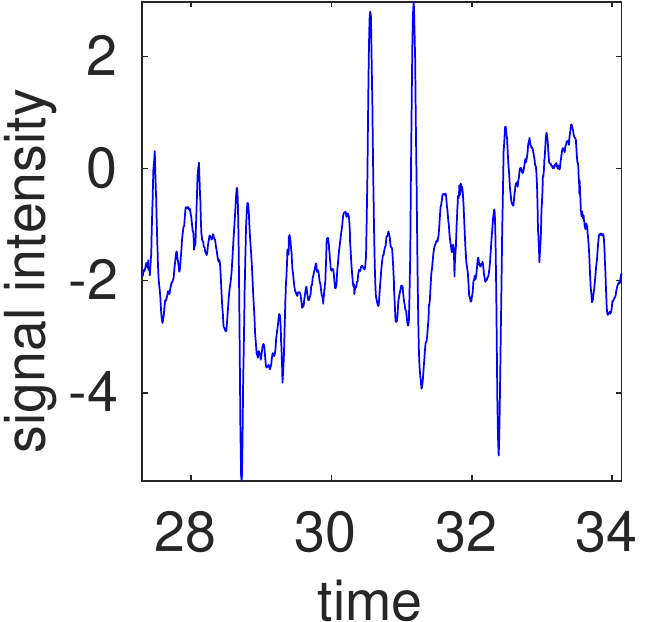} \\
      \includegraphics[height=1.2in]{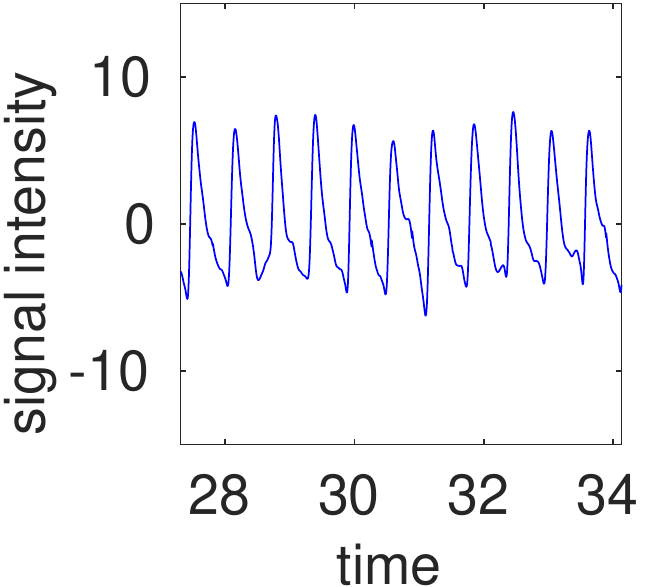}  &
      \includegraphics[height=1.2in]{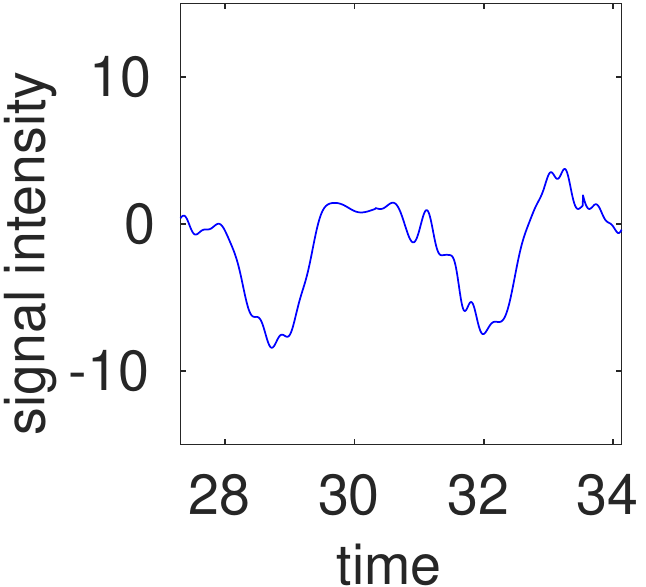}&
      \includegraphics[height=1.2in]{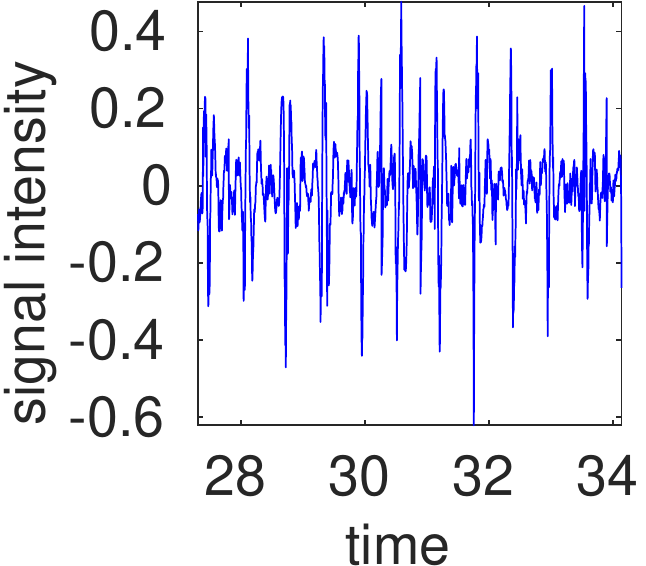}  \\
      \end{tabular}
  \end{center}
  \vspace{-0.5cm}
    \caption{Decomposition of a {\bf real} PPG signal (top). Middle panel is the decomposition by the generalized mode decomposition in \eqref{P2} and the bottom panel is by the MMD \eqref{eqn:m_mmd}. The decomposed results contain one cardiac mode (left), one respiratory mode (middle), and the residual signal (right). The residual signal by MMD is significantly weaker and behaves more like white noise than that by Model \eqref{P2}.  Figure \ref{fig:ppg01} below visualizes the whiteness of the residual signals.}
  \label{fig:ppg0}
  \vspace{-0.5cm}
\end{figure}

\begin{figure}
  \vspace{-0.5cm}
  \begin{center}
\begin{tabular}{ccc}
   \includegraphics[width=1.2in]{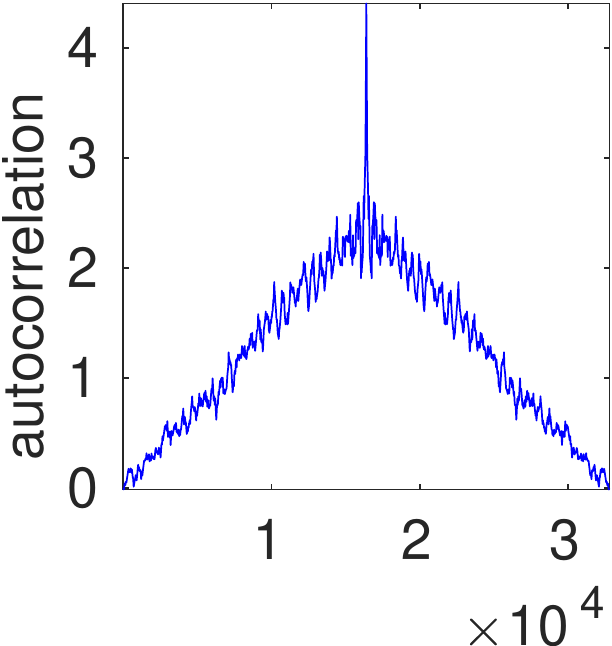} & \includegraphics[width=1.25in]{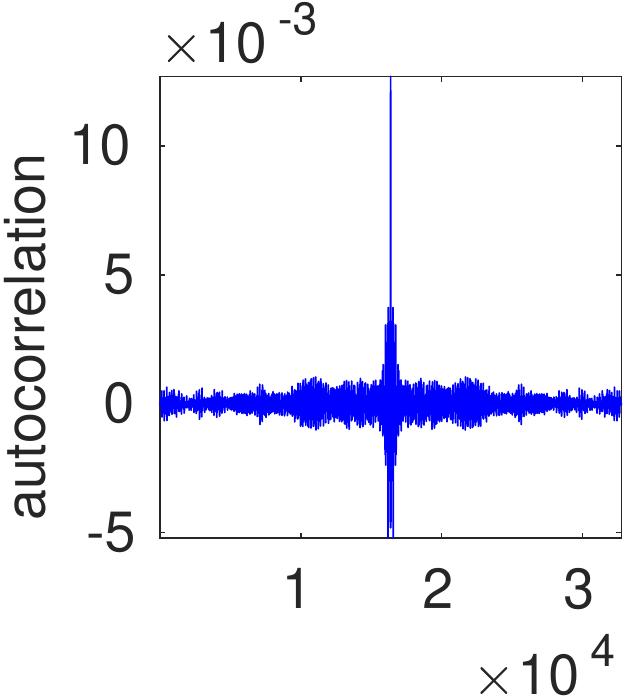}&  \includegraphics[width=1.3in]{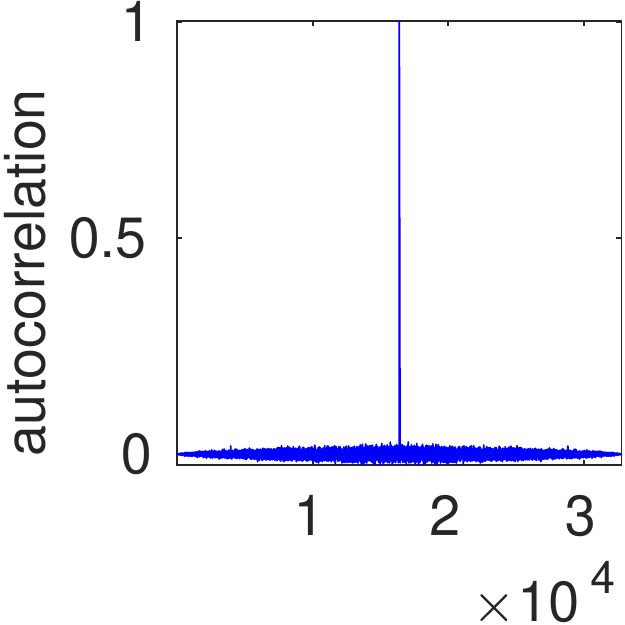} 
   \end{tabular}
  \end{center}
  \vspace{-0.5cm}
    \caption{Comparison of the whiteness of the residual signal generated by \eqref{P2} and MMD \eqref{eqn:m_mmd} for the PPG signal in Figure \ref{fig:ppg0}. The autocorrelation of the residual signal by \eqref{P2}, the residual signal by MMD, and a vector of Gaussian random noise is plotted in the left, middle, and right figures, respectively. Theoretically, the autocorrelation of white noise is an impulse at lag 0. Hence, the results here show that the residual signal by MMD is close to white noise, while that by \eqref{P2} still contains correlated oscillation.}
  \label{fig:ppg01}
\end{figure}

To better analyze time series with time-dependent amplitudes, phases, and shapes, the \emph{multiresolution mode decomposition} (MMD) is proposed in \cite{MMD} of the form
\begin{equation}
\label{eqn:m_mmd}
f(t) = \sum_{k=1}^{K} f_k(t),
\end{equation}
where each
\begin{eqnarray}\label{eqn:m_IMF}
f_k(t)&=&\sum_{n=-N/2}^{N/2-1} a_{n,k}\cos(2\pi n\phi_k(t))s_{cn,k}(2\pi N_k\phi_k(t))\nonumber\\
&&+\sum_{n=-N/2}^{N/2-1}b_{n,k} \sin(2\pi n\phi_k(t))s_{sn,k}(2\pi N_k\phi_k(t))
\end{eqnarray}
is a \emph{multiresolution intrinsic mode function} (MIMF) with shape functions and phase functions satisfying the same conditions as in \eqref{P2}. 
MIMF is a generalization of the model $\alpha_k(t)s_k(2\pi N_k\phi_k(t))$ in Equation \eqref{P2} for more accurate data analysis (see the comparison of the model \eqref{P2} and \eqref{eqn:m_mmd} in Figure \ref{fig:ppg0} for the improvement). When $s_{cn,k}(t)$ and $s_{sn,k}(t)$ in Equation \eqref{eqn:m_IMF} are equal to the same shape function $s_k(t)$, the model in Equation \eqref{eqn:m_IMF} is reduced to $\alpha_k(t)s_k(2\pi N_k\phi_k(t))$ once the amplitude function $\alpha_k(t)$ is written in the form of its Fourier series expansion. When $s_{cn,k}(t)$ and $s_{sn,k}(t)$ are different shape functions, the two summations in Equation \eqref{eqn:m_IMF} lead to time-dependent shape functions to describe the nonlinear and non-stationary time series adaption. 

A recent paper \cite{ceptrum} also tried to address the limitation of Model \eqref{P2} by replacing $\widehat{s_k}(n)\alpha_k(t)$ with a time-varying function, denoted as $B_{k,n}(t)$, i.e., introducing more variance to amplitude functions. % and hoping to absorb the time-variance of shape functions into amplitude function.  
Our model in \eqref{eqn:m_IMF} emphasizes both the time variance of amplitude and shape functions by introducing multiresolution expansion coefficients and shape function series. 

It was shown in \cite{MMD} that the MIMF model can capture the evolution variance, which is more important than the average evolution patterns of oscillatory data, for detecting diseases and measuring health risk. Let $\mathcal{M}_\ell$ be the operator for computing the \emph{$\ell$-banded multiresolution approximation} to a MIMF $f(t)$, i.e.,
\begin{eqnarray}
\label{eqn:mm_IMF}
\mathcal{M}_\ell(f)(t) = \sum_{n=-\ell}^{\ell} a_n\cos(2\pi n\phi(t))s_{cn}(2\pi N\phi(t))+\sum_{n=-\ell}^{\ell}b_n \sin(2\pi n\phi(t))s_{sn}(2\pi N\phi(t)),
\end{eqnarray}  
and $\mathcal{R}_\ell$ be the operator for the computing of the residual sum
\begin{equation}
\label{eqn:rr_IMF}
\mathcal{R}_\ell(f)(t) = f(t)-\mathcal{M}_\ell(f)(t).
\end{equation}
 Then the $0$-banded multiresolution approximation $\mathcal{M}_0(f)(t)=a_0 s_{c0}(2\pi N\phi(t))$ describes the average evolution pattern of the signal, while the rest describe the evolution variance. Figure \ref{fig:ECG1} shows that, if $f(t)$ is an ECG signal\footnote{From the PhysiNet \url{https://physionet.org/}. }, $\mathcal{R}_0(f)(t)$ visualizes the change of the evolution pattern better than $f(t)$, e.g., the change of the height of R peaks, the width of QRS and S waves.

\begin{figure}
  \begin{center}
   \includegraphics[width=5in]{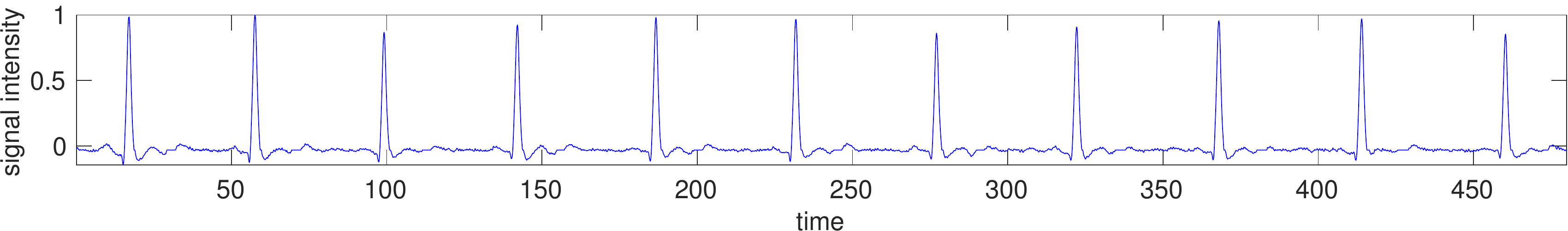} \\
      \includegraphics[width=5in]{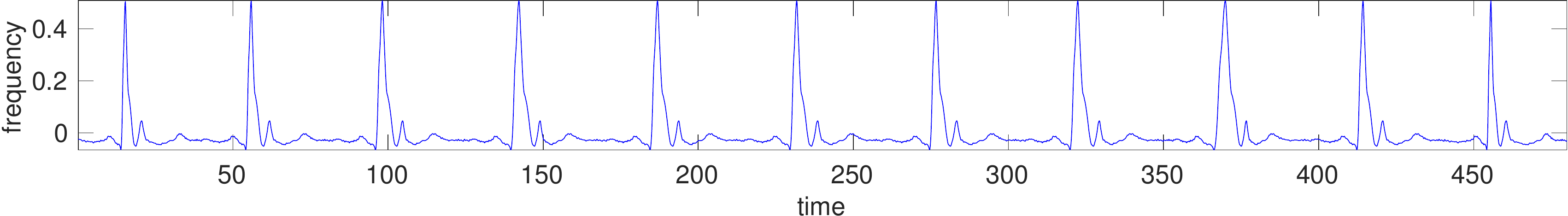}\\
      \includegraphics[width=5in]{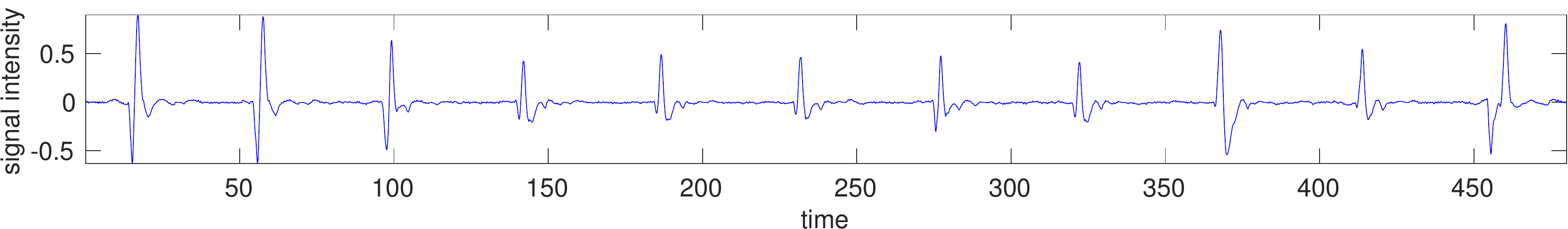}
  \end{center}
  \vspace{-0.5cm}
    \caption{Top: a motion artifact contaminated ECG signal $f(t)$ modeled by Equation \eqref{eqn:m_IMF}. Middle: the $0$-band multiresolution approximation $\mathcal{M}_0(f)(t)=a_0 s_{c0}(2\pi N\phi(t))$ of $f(t)$. Bottom: $f(t) -\mathcal{M}_0(f)(t)$, the variance of the evolution pattern of $f(t)$.}
  \label{fig:ECG1}
\end{figure}

The MMD problem aims at extracting each MIMF $f_k(t)$, estimating its corresponding multiresolution expansion coefficients $\{a_{n,k}\}$, $\{b_{n,k}\}$, and the shape function series $\{s_{cn,k}(t)\}$ and $\{s_{sn,k}(t)\}$, {\bf assuming} that the phase functions $\{\phi_k\}$ are known. Estimating phase functions have been an active research field in mode decomposition and method well-established approaches \cite{Huang1998,doi:10.1142/S1793536909000047,Daubechies2011,BEHERA2016,Auger1995,Chassande-Mottin2003,VMD,Hou2012,YangWang,Cicone2016384,ceptrum} can be applied to estimate phase functions. Hence, we only focus on the estimation of other quantities in MMD. 

Applying the idea of recursive diffeomorphism-based regression (RDBR) \cite{HZYregression}, \cite{MMD} has proposed a recursive scheme for decomposing $f(t)$ into several MIMF's, $\{f_k(t)\}$. Due to the repeated application of the expensive diffeomorphism-based regression, the method in \cite{MMD} is not suitable for analyzing large data sets, especially when real-time analysis is required. Analyzing a single record of a high-resolution ECG or PPG signal with a few minutes of duration could take a whole day. FFT-based shape function analysis in \cite{1DSSWPT,Hou2016} is efficient but they can be only applied to Model \eqref{P2} and sometimes even for the case when $K=1$ without any proof of convergence.

This dilemma motivates the design of the recursive diffeomorphism-based spectral analysis (RDSA) in this paper. From the {\bf computational point of veiw}, RDSA takes only $O(m L\log L)$ operations to solve the MMD problem by taking advantage of the NUFFT, where $L$ is the length of the signal and $m$ is the number of iterations. As we shall see later, the speedup of RDSA over RDBR in \cite{MMD} can be as large as $1000$. From the {\bf theoretical point of view}, RDSA builds the bridge between FFT-based analysis and the RDBR, leading to a complete convergence analysis and filling the gap of theoretical analysis of FFT-based approaches in \cite{1DSSWPT,Hou2016}. 

A recent paper \cite{NOP} proposed a complete framework for estimating all instantaneous quantities together in a two-step alternative fitting scheme: 1) fitting shape functions when amplitude and phase estimations are given; 2) fitting amplitude and phase functions when shape estimations are given. RDSA can be applied in this alternative fitting scheme to speed up the convergence. For many challenging numerical examples concerning crossover instantaneous frequencies, close frequencies, the elimination of numerical errors in both steps via alternative fitting, the reader is referred to \cite{NOP} for more examples. 

We will first introduce RDSA in Section \ref{sec:GMD}. The convergence of RDSA is introduced and asymptotically analyzed\footnote{Notations in the asymptotic analysis: we shall use the $O(\epsilon)$ notation, as well as the related notations  $\lesssim$ and $\gtrsim$; in particular, we write $F=O(\epsilon)G$ if there exists a constant $C$ (which we will not specify further) such that $|F|\leq C\epsilon |G|$; here $C$ may depend on some general parameters as detailed just before Theorem \ref{thm:conv3}.} in Section \ref{sec:cov}. In Section \ref{sec:NumEx}, we present some numerical examples to demonstrate the efficiency of RDSA. Finally, we conclude this paper in Section \ref{sec:con}.

\section{RDSA}
\label{sec:GMD}

\subsection{Diffeomorphism-based spectral analysis (DSA) for a single MIMF}
\label{sec:sm}
We first introduce the DSA for a single MIMF defined as follows.

\begin{definition} Generalized shape functions:
\label{def:GSF}
The generalized shape function class ${\cal S}_M$ consists of $2\pi$-periodic functions $s(t)$ in the Wiener Algebra with a unit $L^2([0,2\pi])$-norm and a $L^\infty$-norm bounded by $M$ satisfying the following spectral conditions:
\begin{enumerate}
\item The Fourier series of $s(t)$ is uniformly convergent;
\item $\sum_{n=-\infty}^{\infty}|\widehat{s}(n)|\leq M$ and $\widehat{s}(0)=0$;
\item Let $\Lambda$ be the set of integers $\{|n|: \widehat{s}(n)\neq 0\}$. The greatest common divisor $\gcd(s)$ of all the elements in $\Lambda$ is $1$.
\end{enumerate}
\end{definition}

\begin{definition}
  \label{def:m_IMF}
  A function 
  \begin{equation}
\label{eqn:m_IMF1}
f(t) = \sum_{n=-N/2}^{N/2-1} a_n\cos(2\pi n\phi(t))s_{cn}(2\pi N\phi(t))+\sum_{n=-N/2}^{N/2-1}b_n \sin(2\pi n\phi(t))s_{sn}(2\pi N\phi(t))
\end{equation}
is a MIMF of type $(M_0,M,N,\epsilon)$ defined on $[0,1]$, if the conditions below are satisfied:
\begin{itemize}
\item the shape function series $\{s_{cn}(t)\}$ and $\{s_{sn}(t)\}$ are in ${\cal S}_M$;
\item 
the multiresolution expansion coefficients $\{a_{n}\}$ and $\{b_{n}\}$ satisfy
\begin{align*}
    \sum_{n=-N/2}^{N/2-1}|a_n|\leq M, \quad  \quad  \sum_{n=-N/2}^{N/2-1}|a_n|-\sum_{n=-M_0}^{M_0-1}|a_n|\leq \epsilon, \\
  \sum_{n=-N/2}^{N/2-1}|b_n|\leq M,  \quad  \quad \sum_{n=-N/2}^{N/2-1}|b_n|-\sum_{n=-M_0}^{M_0-1}|b_n|\leq \epsilon;
\end{align*}
\item $\phi(t)$ satisfies
\begin{align*}
    \phi(t)\in C^\infty,  \quad  1/M \leq | \phi'|\leq M.
\end{align*}
\end{itemize}
\end{definition}

As usual, we assume that the phase function $N\phi(t)$ and $N$ are available; these quantities can be estimated using time-frequency concentration methods \cite{Daubechies2011,BEHERA2016,Auger1995,Chassande-Mottin2003}, or  adaptive optimization \cite{VMD,Hou2012,7990179}. With the abuse of notations, we will use the same notation for continuous and discrete functions or transforms for simplicity. Without loss of generality, we assume that the signal $f(t)$ is uniformly sampled on $[0,1]$ with $L$ grid points 
\begin{equation}
\label{eqn:sample}
\left\{t_\ell:=\frac{\ell}{L}:0\leq \ell\leq L-1,\ell\in\mathbb{Z}\right\};
\end{equation}
the discrete Fourier transform of $f$ denoted as $\widehat{f}(\xi)$ (or $\mathcal{F}(f)$) is defined on \[\left\{\xi\in\mathbb{Z}:-\frac{L}{2}\leq \xi\leq \frac{L}{2}-1 \right\}.\] 
We further assume that $f(t)$ satisfies a periodic boundary condition only in theoretical analysis for simplicity; in the case of non-periodic boundary condition, the proposed method still works but it is tedious to guarantee the estimation near the boundary theoretically. In all of our numerical examples, signals are non-periodic.

When $f(t)$ is a MIMF, the smooth function $\phi(t)$ serves as a diffeomorphism mapping $\cos(2\pi m\phi(t))f(t)$ to
  \begin{eqnarray}\label{eqn:hdef}
h(t)&=&(\cos(2\pi m\phi)f)\circ\phi^{-1}(t) \\
&=& \sum_{n=-N/2}^{N/2-1} \frac{a_n}{2}\left(\cos(2\pi (m+n)t) + \cos(2\pi (m-n)t) \right)s_{cn}(2\pi Nt)\nonumber\\
& &+\sum_{n=-N/2}^{N/2-1}\frac{b_n}{2}\left( \sin(2\pi (n+m)t)+\sin(2\pi (n-m)t)\right)s_{sn}(2\pi Nt).\nonumber
\end{eqnarray}
%where there is only one term with a non-trigonometric amplitude, $\frac{a_m}{2}s_{cm}(2\pi N\phi(t))$.

Let us define a scaling operator $T_{N}$ mapping a function $g(\xi)$ to a function $T_{N}(g)$ defined as follows:
\[
T_{N}(g)(\xi):=g(N\xi).
\]
In the discrete case, this is equivalent to subsampling the function $g(\xi)$ at the grid points $\{N\xi\}_{\xi\in\mathbb{Z}}$. After Fourier transform and subsampling, we have
\begin{eqnarray}
\label{eqn:hhat}
T_{N}(\widehat{h}(\xi) )&=& 2^{1-|\sgn(m)|} \frac{a_m}{2}\sum_\ell \widehat{s}_{cm}(\ell)\delta(\xi-\ell),
\end{eqnarray}
where $\delta(\cdot)$ denotes the Dirac delta function. In practice, $\widehat{h}(\xi)$ can be evaluated via the NUFFT of $\cos(2\pi m\phi(t))f(t)$ on non-uniform grids \[\left\{\psi_\ell:=\phi(\frac{\ell}{L}):0\leq \ell\leq L-1,\ell\in\mathbb{Z}\right\}.\]

Equation \eqref{eqn:hdef} and \eqref{eqn:hhat} result in
\begin{align}\label{eqn:s1}
a_ms_{cm}(2\pi t)=  2^{|\sgn(m)|}\mathcal{F}^{-1}\left(T_{N}\left(\mathcal{F}((\cos(2\pi m\phi)f)\circ\phi^{-1})\right)\right)(t), 
\end{align}
and similarly we have
\begin{align}\label{eqn:s2}
b_ms_{sm}(2\pi t)=  2^{|\sgn(m)|}\mathcal{F}^{-1}\left(T_{N}\left(\mathcal{F}((\sin(2\pi m\phi)f)\circ\phi^{-1})\right)\right)(t)      
\end{align}
for $t\in[0,1]$. Since all shape functions have a unit $L^2([0,2\pi])$-norm\footnote{In numerical implementation, we have a band-width parameter $L_s$ for shape functions, i.e., only consider the Fourier series coefficient vector $\vec{\widehat{s}}\in\mathbb{C}^{L_s}$  with entries $\widehat{s}(m)$ for $-\frac{L_s}{2}\leq m\leq \frac{L_s}{2}-1$ in the reconstruction of a shape function vector $\vec{s}\in\mathbb{R}^{L_s}$ with entries $s(2\pi t)$ sampled on the grid points $\left\{t=\frac{k}{L_s}:0\leq k\leq L_s-1,k\in\mathbb{Z}\right\}$. The discrete analog of the $L^2$-norm $\|s\|_{L^2([0,2\pi])}$ of a function $s(t)$ is defined as $\sqrt{\frac{2\pi}{L_s}}\|\vec{s}\|_{\ell^2}$.}, we have
\begin{eqnarray}\label{eqn:coefest1}
a_m=\sqrt{2\pi}\|2^{|\sgn(m)|}\mathcal{F}^{-1}\left(T_{N}\left(\mathcal{F}((\cos(2\pi m\phi)f)\circ\phi^{-1})\right)\right) \|_{L^2([0,1])},
\end{eqnarray}
and
\begin{eqnarray}\label{eqn:coefest2}
b_m=\sqrt{2\pi} \|2^{|\sgn(m)|}\mathcal{F}^{-1}\left(T_{N}\left(\mathcal{F}((\sin(2\pi m\phi)f)\circ\phi^{-1})\right)\right)    \|_{L^2([0,1])},   
\end{eqnarray}
where the prefactor $\sqrt{2\pi}$ comes from changing the integral domain from $[0,2\pi]$ to $[0,1]$. Hence,
\begin{eqnarray}\label{eqn:shapeest1}
s_{cm}(2\pi t)=& \begin{cases}
\frac{1}{a_m}2^{|\sgn(m)|}\mathcal{F}^{-1}\left(T_{N}\left(\mathcal{F}((\cos(2\pi m\phi)f)\circ\phi^{-1})\right)\right)(t), & \quad a_m\neq 0, \\
0,  & \quad a_m=0,
   \end{cases}   
\end{eqnarray}
and
\begin{eqnarray}\label{eqn:shapeest2}
s_{sm}(2\pi t)=& \begin{cases}
\frac{1}{b_m}2^{|\sgn(m)|}\mathcal{F}^{-1}\left(T_{N}\left(\mathcal{F}((\sin(2\pi m\phi)f)\circ\phi^{-1})\right)\right)(t), & \quad b_m\neq 0, \\
0,  & \quad b_m=0.
   \end{cases}   
\end{eqnarray}

The above discussion can be summarized in Algorithm \ref{alg:DSA} for estimating shape functions and expansion coefficients from a single MIMF in \eqref{eqn:m_IMF1}.

\vspace{0.2in}

\begin{algorithm2e}[H]
\label{alg:DSA}
\caption{DSA for shape functions and expansion coefficients with $O(mL\log L)$ operation complexity. The operation complexity comes from the fact that all routines are dominated by a NUFFT of size at most $L$ and point-wise summation of vectors of size at most $L$.}
Input: A single MIMF $f(t)$ in \eqref{eqn:m_IMF1} and the phase function $p(t)=N\phi(t)$ sampled over $t_\ell$, $\ell = 0,\dots, L-1$, frequency parameters $N$, a band-width parameter $L_s$, and a set of scale indices $\mathfrak{S}=\{n_1,\cdots,n_m\}$.

Output: The shape functions $s_{cn}$ and $s_{sn}$, the expansion coefficients $a_n$ and $b_n$, and the partial summation $f_{c}(t)$ and $f_{s}(t)$.

Compute the expansion coefficients $a_n$ and $b_n$ for $n\in \mathfrak{S}$ according to \eqref{eqn:coefest1} and \eqref{eqn:coefest2}\footnote{Note that in Definition \ref{def:GSF} all shape functions have zero-mean. Hence, in the numerical implementation of $T_{N}(\widehat{h})(\xi)$ in \eqref{eqn:hhat}, we will manually make $T_{N}(\widehat{h})(0)=0$, if $T_{N}(\widehat{h})(0)\neq 0$ due to noise perturbation in the signal $f(t)$. }.

Evaluate the shape functions $s_{cn}(2\pi t)$ and $s_{sn}(2\pi t)$ for $n\in \mathfrak{S}$ according to \eqref{eqn:shapeest1} and \eqref{eqn:shapeest2} on uniform grids $\left\{t=\frac{k}{L_s}:0\leq k\leq L_s-1,k\in\mathbb{Z}\right\}$.

Evaluate $a_n cos(2\pi n\phi(t))s_{cn}(2\pi N\phi(t))$ and $b_n sin(2\pi n\phi(t))s_{sn}(2\pi N\phi(t))$ for $n\in \mathfrak{S}$ based on interpolating the shape functions from uniform grids to non-uniform grids \[\left\{t=\phi(\frac{k}{L}):0\leq k\leq L-1,k\in\mathbb{Z}\right\}.\]

Compute the partial summation 
\[
f_c(t)=\sum_{n\in\mathfrak{S}} a_n cos(2\pi n\phi(t))s_{cn}(2\pi N\phi(t))
\]
and
\[
f_s(t)=\sum_{n\in\mathfrak{S}} b_n sin(2\pi n\phi(t))s_{sn}(2\pi N\phi(t)).
\]

\end{algorithm2e}

\vspace{0.2in}

\subsection{RDSA for multiple MIMF's}
\label{sec:mm}
Next, in the case of a superposition of several MIMF's, 
\begin{equation}
\label{eqn:m_mmd2}
f(t) = \sum_{k=1}^{K} f_k(t),
\end{equation}
where each
\begin{eqnarray}\label{eqn:m_IMF2}
f_k(t)&=&\sum_{n=-N/2}^{N/2-1} a_{n,k}\cos(2\pi n\phi_k(t))s_{cn,k}(2\pi N_k\phi_k(t))\nonumber\\
&&+\sum_{n=-N/2}^{N/2-1}b_{n,k} \sin(2\pi n\phi_k(t))s_{sn,k}(2\pi N_k\phi_k(t)),
\end{eqnarray}
we propose the RDSA to extract each MIMF, estimate its corresponding multiresolution expansion coefficients, and the shape function series from the superposition.

Due to the interference between different MIMF's, directly applying Algorithm \ref{alg:DSA} with an input signal in \eqref{eqn:m_mmd2} and a phase function $N_k\phi_k(t)$ would not lead to accurate estimation of the multiresolution expansion coefficients, denoted as $\dot{a}_{n,k}$ and $\dot{b}_{n,k}$, and shape function series of $f_k(t)$, denoted as $\dot{s}_{cn,k}$ and $\dot{s}_{sn,k}$. This motivates the application of Algorithm \ref{alg:DSA} combined with the recursive scheme proposed in \cite{MMD}. The intuition of the recursive scheme can be summarized as follows. Though the accuracy of Algorithm \ref{alg:DSA} might not be good, we can still get a rough estimation of $f_k(t)$, denoted as 
\begin{eqnarray*}
\dot{f}_k(t)&=&\sum_{n=-N/2}^{N/2-1} \dot{a}_{n,k}\cos(2\pi n\phi_k(t))\dot{s}_{cn,k}(2\pi N_k\phi_k(t))\nonumber\\
&&+\sum_{n=-N/2}^{N/2-1}\dot{b}_{n,k} \sin(2\pi n\phi_k(t))\dot{s}_{sn,k}(2\pi N_k\phi_k(t)).
\end{eqnarray*}
Hence, the residual signal $r(t):=f(t)-\dot{f}(t)$ is again a new superposition of MIMF's. The recursive scheme applies Algorithm \ref{alg:DSA} again to $r(t)$ to estimate new multiresolution expansion coefficients and shape function series. We hope that the new estimations can correct the estimation error in the previous step; if this correction idea is applied repeatedly, we hope that the residual signal will decay and the estimation error will approach to zero. In more particular, RDSA can be summarized in Algorithm \ref{alg:MMD1}. In the pseudo-code in Algorithm \ref{alg:MMD1}, the input and output of Algorithm \ref{alg:DSA} is denoted as 
\[
[\{s_{cn}\}_{n\in\mathfrak{S}},\{s_{sn}\}_{n\in\mathfrak{S}},\{a_{n}\}_{n\in\mathfrak{S}},\{b_{n}\}_{n\in\mathfrak{S}},f_s,f_c]=DSA(f,p,N,L_s,\mathfrak{S}).
\]
When the input $M_1$ of Algorithm \ref{alg:MMD1} is set to be empty, then Algorithm \ref{alg:MMD1} returns the $M_0$-banded multiresolution approximation to each MIMF $f_k(t)$, its corresponding multiresolution expansion coefficients, and shape function series. 

In fact, we have two for-loops to apply Algorithm \ref{alg:DSA} repeatedly to correct the estimation error: 1) one for-loop for the scale index $n$ in Algorithm \ref{alg:DSA}; 2) another one for the MIMF component index $k$ in Algorithm \ref{alg:MMD1}. Note that in the case of a superposition of several MIMF's, the estimation provided by Line $12$ in Algorithm \ref{alg:MMD1} is not accurate: the estimation error of a larger $|n|$ is much larger than that of a smaller $|n|$ because $|a_{n,k}|$ and $|b_{n,k}|$ usually decay quickly in $|n|$. As the iteration goes on, the multiresolution expansion coefficients with a small scale index ${n}$ in the residual signal will decay since previous estimation steps try to eliminate them in the residual signal; only after a sufficiently large number of iterations in $j$, Line $12$ in Algorithm \ref{alg:MMD1} can give accurate estimations for multiresolution expansion coefficients with a large $|n|$. Hence, to make Algorithm \ref{alg:MMD1} converge, a large number of iteration number $J_1$ might be required. 

\vspace{0.2in}

\begin{algorithm2e}[H]
\label{alg:MMD1}
\caption{The first RDSA for MMD. The operation complexity of this algorithm is $O(\max\{M_0,M_1\}J_1 K L\log L)$ since the essential cost is the application of the DSA in Algorithm \ref{alg:DSA} for $J_1K$ times.}
Input: $L$ points of measurement $\{f(t_\ell)\}_{\ell=0,\dots,L-1}$ with $t_\ell \in [0,1]$, estimated instantaneous phases $\{p_k\}_{k=1,\dots,K}$, an accuracy parameter $\epsilon$, the maximum iteration number $J_1$, and band-width parameters $L_s$, $M_0$ and $M_1$.

Output: A scale index set $\mathfrak{S}$, $f^{est}_k(t)$ at the sampling grid points $\{t_\ell\}_{0\leq \ell\leq L-1}$, its multiresolution expansion coefficients $\{a_{n,k}\}_{n\in\mathfrak{S}}$ and $\{b_{n,k}\}_{n\in\mathfrak{S}}$, and its shape function series $\{s_{cn,k}\}_{n\in\mathfrak{S}}$ and $\{s_{sn,k}\}_{n\in\mathfrak{S}}$ for $1\leq k\leq K$. 

\If {$M_1$ has not been specified}{
Define the scale index set $\mathfrak{S}=\{-M_0,-M_0+1,\cdots,M_0\}$.
}
\Else{
Define $\mathfrak{S}=\{-M_1+1,-M_1+2,\cdots,-M_0,\}\cup\{M_0,M_0+1,\cdots,M_1-1,\}$.
}

Initialize: let $a_{n,k}=0$, $b_{n,k}=0$, $s_{cn,k}=0$, $s_{sn,k}=0$, $f^{est}_k(t)=0$ for all $k$ and $n\in\mathfrak{S}$; let $c=\|f\|_{L^2}$; let $e=1$; let $r^{(0)}=f$.

Compute $N_k$ as the integer nearest to the average of $p'_k(t)$ for $k=1,\dots,K$.

Sort $\{N_k\}_{1\leq k\leq K}$ in an ascending order and reorder the phase functions accordingly.

\For{$j=1,2,\dots,J_1,$}{

\For{$k=1,\dots,K$}{

$[\{\bar{s}_{cn}\}_{n\in\mathfrak{S}},\{\bar{s}_{sn}\}_{n\in\mathfrak{S}},\{\bar{a}_{n}\}_{n\in\mathfrak{S}},\{\bar{b}_{n}\}_{n\in\mathfrak{S}},\bar{f}_{s},\bar{f}_c]=DSA(r^{(j-1)},p_k,N_k,L_s,\mathfrak{S})$.

$s_{cn,k}\leftarrow s_{cn,k}+\bar{s}_{cn}$ and $s_{sn,k}\leftarrow s_{sn,k}+\bar{s}_{sn}$ for $n\in \mathfrak{S}$.

Update $f^{est}_k(t)\leftarrow f^{est}_k(t) + \bar{f}_c+\bar{f}_s$.

\If {$k<K$}{
Update $r^{(j-1)}\leftarrow r^{(j-1)}- \bar{f}_c-\bar{f}_s$.
}
\Else{
Compute $r^{(j)}= r^{(j-1)}- \bar{f}_c-\bar{f}_s$.
}
}

If $\|r^{(j)}\|_{L^2}/c\leq\epsilon$, then break the for-loop.

\If{$\|r^{(j)}\|_{L^2}/c\geq e-\epsilon$}{
Break the for loop.
}
\Else{
$e=\|r^{(j)}\|_{L^2}/c$.
}
}

Let $a_{n,k}=\|s_{cn,k}\|_{L^2}$ and $s_{cn,k}=s_{cn,k}/a_{n,k}$ for all $k$ and $n\in \mathfrak{S}$.

Let $b_{n,k}=\|s_{sn,k}\|_{L^2}$ and $s_{sn,k}=s_{sn,k}/b_{n,k}$ for all $k$ and $n\in \mathfrak{S}$.
\end{algorithm2e}

\vspace{0.2in}

To reduce the number of iterations $J_1$ in Algorithm \ref{alg:MMD1}, it might be better to put the $k$-for-loop inside the $n$-for-loop as in Algorithm \ref{alg:MMD2}, i.e., eliminating the multiresolution expansion coefficients with a small $|n|$ in the residual signal first before estimating those coefficients with a large $|n|$. It is still unclear which algorithm is faster since it relies on the decay rate of multiresolution expansion coefficients in $|n|$. Hence, a block size parameter $\mathfrak{b}$ is used to make a balance: when $J_2=1$ and $\mathfrak{b}=M_0+1$ in Algorithm \ref{alg:MMD2}, Algorithm \ref{alg:MMD2} essentially becomes Algorithm \ref{alg:MMD1}; when $\mathfrak{b}=1$ in Algorithm \ref{alg:MMD2}, Algorithm \ref{alg:MMD2} only computes the multiresolution expansion coefficients and shape functions for two scale indices per iteration in $\ell$ in Line $5$ of Algorithm \ref{alg:MMD2}. In the pseudo-code in Algorithm \ref{alg:MMD2}, the input and output of Algorithm \ref{alg:MMD1} is denoted as 
\begin{eqnarray*}
[\mathfrak{S},\{f^{est}_k(t)\}_{1\leq k\leq K},\{s_{cn,k}\}_{1\leq k\leq K,n\in\mathfrak{S}},\{s_{sn,k}\}_{1\leq k\leq K,n\in\mathfrak{S}},\{a_{n,k}\}_{1\leq k\leq K,n\in\mathfrak{S}},\{b_{n,k}\}_{1\leq k\leq K,n\in\mathfrak{S}}]\\
=RDSA_1(f,\{p_k\}_{1\leq k\leq K},\epsilon,J_1,L_s,M_0,M_1).
\end{eqnarray*}

\vspace{0.2in}

\begin{algorithm2e}[H]
\label{alg:MMD2}
\caption{The second RDSA for MMD. For the purpose of simplicity, we assume that $M_0+1$ is a multiple of $\mathfrak{b}$. The operation complexity is bounded by $O(\max\{M_0,M_1\}J_1 K L\log L)$ since this is a faster algorithm than Algorithm \ref{alg:MMD2}.}
Input: $L$ points of measurement $\{f(t_\ell)\}_{\ell=0,\dots,L-1}$ with $t_\ell \in [0,1]$, estimated instantaneous phases $\{p_k\}_{k=1,\dots,K}$, accuracy parameters $\epsilon_1$ and $\epsilon_2$, the maximum iteration numbers $J_1$ and $J_2$, band-width parameters $M_0$ and $L_s$, and a block-size parameter $\mathfrak{b}$.

Output: $\mathcal{M}_{M_0}(f_k)(t)$ at the sampling grid points $\{t_\ell\}_{0\leq \ell\leq L-1}$, its multiresolution expansion coefficients $\{a_{n,k}\}_{n=-M_0,\dots,M_0}$ and $\{b_{n,k}\}_{n=-M_0,\dots,M_0}$, and its shape function series $\{s_{cn,k}\}_{n=-M_0,\dots,M_0}$ and $\{s_{sn,k}\}_{n=-M_0,\dots,M_0}$ for $1\leq k\leq K$. 

Initialize: let $a_{n,k}=0$, $b_{n,k}=0$, $s_{cn,k}=0$, $s_{sn,k}=0$, $\mathcal{M}_{M_0}(f_k)=0$ for all $k$ and $n$; let $c=\|f\|_{L^2}$; let $e=1$; let $r^{(0)}=f$.

\For{$j=1,2,\dots,J_2,$}{

\For{$\ell=1,\dots,(M_0+1)/\mathfrak{b}$}{
Compute $m=(\ell-1)\mathfrak{b}$ and apply
\begin{eqnarray*}
[\mathfrak{S},\{f^{est}_k(t)\}_{1\leq k\leq K},\{\bar{s}_{cn,k}\}_{1\leq k\leq K,n\in\mathfrak{S}},\{\bar{s}_{sn,k}\}_{1\leq k\leq K,n\in\mathfrak{S}},\{a_{n,k}\}_{1\leq k\leq K,n\in\mathfrak{S}},\\
\{b_{n,k}\}_{1\leq k\leq K,n\in\mathfrak{S}}]=RDSA_1(r^{(j-1)},\{p_k\}_{1\leq k\leq K},\epsilon_2,J_1,L_s,m,m+\mathfrak{b}).
\end{eqnarray*}

\For{$k=1,\dots,K$}{
$s_{cn,k}\leftarrow s_{cn,k}+\bar{s}_{cn,k}$ and $s_{sn,k}\leftarrow s_{sn,k}+\bar{s}_{sn,k}$.

Update $\mathcal{M}_{M_0}(f_k)(t)\leftarrow \mathcal{M}_{M_0}(f_k)(t) + \bar{f}^{est}_k$.

Compute $r^{(j-1)}\leftarrow r^{(j-1)}-\bar{f}^{est}_k$.
}

}

$r^{(j)}=r^{(j-1)}$.

If $\|r^{(j)}\|_{L^2}/c\leq\epsilon_1$, then break the for loop.

\If{$\|r^{(j)}\|_{L^2}/c\geq e-\epsilon_1$}{
Break the for loop.
}
\Else{
$e=\|r^{(j)}\|_{L^2}/c$.
}
}

Let $a_{n,k}=\|s_{cn,k}\|_{L^2}$ and $s_{cn,k}=s_{cn,k}/a_{n,k}$ for all $k=1,\dots,K$ and $n=-M_0,\dots,M_0$.

Let $b_{n,k}=\|s_{sn,k}\|_{L^2}$ and $s_{sn,k}=s_{sn,k}/b_{n,k}$ for all $k=1,\dots,K$ and $n=-M_0,\dots,M_0$.
\end{algorithm2e}

\section{Convergence analysis}
\label{sec:cov}

Although the RDSA in Algorithm \ref{alg:MMD1} and \ref{alg:MMD2} is mainly based on Fourier analysis, it can be proved that they are equivalent to the RDBR in \cite{MMD}, which leads to the theory of the convergence of RDSA. Since Algorithm \ref{alg:MMD1} is a special case of Algorithm \ref{alg:MMD2}, we will only focus on the convergence analysis of Algorithm \ref{alg:MMD2}. Without loss of generality, we assume $\mathfrak{b}=1$ in the analysis. 

\subsection{Preliminaries}

Before presenting the theory for RDSA, let us revisit RDBR for MMD in \cite{MMD}. In RDBR, if $f(t)=a_ns_{cn}(2\pi N\phi(t))$, we define the inverse-warping data by
$h(v) = f\circ p^{-1}(v)=a_n s_{cn}(2\pi v)$, 
where $v = p(t)=N\phi(t)$. As a consequence, we have a set of measurements of $h(v)$ sampled on  $\{h(v_\ell)\}_{\ell=0,\dots,L-1}$ with $v_\ell=p(t_\ell)$. Note that $h(v)$ is a periodic function with period $1$. Hence, if we define a folding map $\tau$ that folds the two-dimensional point set $\{(v_\ell,h(v_\ell))\}_{\ell=0,\dots,L-1}$  together
\begin{eqnarray}\label{eqn:fold}
\tau: \ \      \left(v_\ell, h(v_\ell) \right)   \mapsto    \left(\text{mod}(v_\ell,1),  h(v_\ell) \right),
\end{eqnarray}
then the point set $\{\tau(v_\ell,a_n s_{cn}(2\pi v_\ell))\}_{\ell=0,\dots,L-1}\subset \mathbb{R}^2$ is a two-dimensional point set located at the curve $(v,a_n s_{cn}(2\pi v))\subset \mathbb{R}^2$ given by the shape function $a_n s_{cn}(2\pi v)$ with $v\in [0,1)$. Using the notations in non-parametric regression, let $X$ be an independent random variable in $[0,1)$, $Y$ be the response random variable in $\mathbb{R}$, and consider $(x_\ell,y_\ell)=\tau(v_\ell,a_n s_{cn}(2\pi v_\ell))$ as $L$ samples of the random vector $(X,Y)$, then a simple regression results in the shape function
\begin{equation}
\label{eqn:re}
a_n s_{cn} =  s^R\defeq \underset{s:\mathbb{R}\rightarrow \mathbb{R}}{\arg\min}\quad \E\{\left| s(2\pi X)-Y\right|^2\},
\end{equation}
where the superscript $^R$ means the ground truth regression function. 

RDBR applies the partition-based regression method (or partitioning estimate) in Chapter 4 of \cite{regressionBook} to solve the above regression problem. Given a small step size $h\ll 1$, the time domain $[0,1]$ is uniformly partitioned into $N^h = \frac{1}{h}$ (assumed to be an integer) parts $\{[t^h_k,t^h_{k+1})\}_{k=0,\dots,N^h-1}$, where $t^h_k=kh$.  Let $s^P$ denote the estimated regression function by the partition-based regression method with $L$ samples. Following the definition in Chapter 4 of \cite{regressionBook}, we have a piecewise function 
\begin{eqnarray}\label{eqn:pbr}
s^P(2\pi x) \defeq\frac{\sum_{\ell=0}^{L-1}\mathcal{X}_{[t^h_k,t^h_{k+1})} (x_\ell) y_\ell }{\sum_{\ell=0}^{L-1}\mathcal{X}_{[t^h_k,t^h_{k+1})}(x_\ell) }&=&\frac{\sum_{\ell=0}^{L-1}\mathcal{X}_{[t^h_k,t^h_{k+1})} (\text{mod}(v_\ell,1)) a_n s_{cn}(2\pi v_\ell) }{\sum_{\ell=0}^{L-1}\mathcal{X}_{[t^h_k,t^h_{k+1})}(\text{mod}(v_\ell,1)) },\\
&=&\frac{\sum_{\ell=0}^{L-1}\mathcal{X}_{[t^h_k,t^h_{k+1})} (\text{mod}(N\phi(t_\ell),1)) f(t_\ell) }{\sum_{\ell=0}^{L-1}\mathcal{X}_{[t^h_k,t^h_{k+1})}(\text{mod}(N\phi(t_\ell),1)) }.\nonumber
\end{eqnarray}
when $x\in [t^h_k,t^h_{k+1})$, where $\mathcal{X}_{[t^h_k,t^h_{k+1})}(x)$ is the indicator function supported on $[t^h_k,t^h_{k+1})$. When $L$ is sufficiently large \begin{equation}\label{eqn:pbr2}
a_n s_{cn}(2\pi x)=s^R(2\pi x)\approx s^P(2\pi x)=\frac{\sum_{\ell=0}^{L-1}\mathcal{X}_{[t^h_k,t^h_{k+1})} (\text{mod}(N\phi(t_\ell),1)) f(t_\ell) }{\sum_{\ell=0}^{L-1}\mathcal{X}_{[t^h_k,t^h_{k+1})}(\text{mod}(N\phi(t_\ell),1)) }
\end{equation}
when $x\in [t^h_k,t^h_{k+1})$, and the approximation is robust against noise perturbation \cite{regressionBook} (Chapter $4$). 

More rigorously, the following theorem given in Chapter 4 in \cite{regressionBook} estimates the $L_2$ risk of the approximation $s^P\approx s^R$ as follows.

\begin{theorem}
\label{thm:reg}
For the uniform partition with a step size $h$ in $[0,1)$ as defined just above, assume that 
\[
\Var(Y|X=x)\leq \sigma^2,\quad x\in\mathbb{R},
\]
\[
|s^R(x)-s^R(z)|\leq C|x-z|,\quad x,z\in\mathbb{R},
\]
$X$ has a compact support $[0,1)$, and there are $L$ i.i.d. samples of $(X,Y)$. Then the partition-based regression method provides an estimated regression function $s^P$ to approximate the ground truth regression function $s^R$, where
\begin{equation*}
s^R = \underset{s:\mathbb{R}\rightarrow \mathbb{R}}{\arg\min}\quad \E\{\left| s(2\pi X)-Y\right|^2\},
\end{equation*}
with an $L^2$ risk bounded by
\begin{equation}\label{bnd}
\E\|s^P-s^R\|^2\leq c_0\frac{\sigma^2+\|s^R\|^2_{L^\infty}}{Lh}+C^2h^2,
\end{equation}
where $c_0$ is a constant independent of the number of samples $L$, the regression function $s^R$,  the step size $h$, and the Lipschitz continuity constant $C$.
\end{theorem}

If $f(t)$ is a MIMF, i.e.,
\begin{equation}\label{eqn:fs}
f(t)= \sum_{n=-N/2}^{N/2-1} a_n\cos(2\pi n\phi(t))s_{cn}(2\pi N\phi(t))+\sum_{n=-N/2}^{N/2-1}b_n \sin(2\pi n\phi(t))s_{sn}(2\pi N\phi(t)),
\end{equation}
under the same condition as in \eqref{eqn:pbr2}, it was shown in \cite{MMD} that 
\begin{equation}\label{eqn:pbr3}
a_n s_{cn}(2\pi x)\approx \frac{2^{|\sgn(n)|}\sum_{\ell=0}^{L-1}\mathcal{X}_{[t^h_k,t^h_{k+1})} (\text{mod}(N\phi(t_\ell),1))\cos(2\pi n\phi(t_\ell)) f(t_\ell) }{\sum_{\ell=0}^{L-1}\mathcal{X}_{[t^h_k,t^h_{k+1})}(\text{mod}(N\phi(t_\ell),1)) }
\end{equation}
and
\begin{equation}\label{eqn:pbr4}
b_n s_{sn}(2\pi x)\approx \frac{2^{|\sgn(n)|}\sum_{\ell=0}^{L-1}\mathcal{X}_{[t^h_k,t^h_{k+1})} (\text{mod}(N\phi(t_\ell),1))\sin(2\pi n\phi(t_\ell)) f(t_\ell) }{\sum_{\ell=0}^{L-1}\mathcal{X}_{[t^h_k,t^h_{k+1})}(\text{mod}(N\phi(t_\ell),1)) }
\end{equation}
when $x\in [t^h_k,t^h_{k+1})$, by a similar argument as in \eqref{eqn:pbr2} and the fact that the oscillation in amplitude functions $\cos(2\pi n\phi(t))$ and $\sin(2\pi n\phi(t))$ removes the influence of other terms in \eqref{eqn:fs} on the estimation of $a_n s_{cn}(2\pi x)$ and $b_n s_{sn}(2\pi x)$, respectively.

In practice, in the case of a superposition of MIMF's, RDBR uses the same recursive algorithm as in Algorithm \ref{alg:MMD2} (when $\mathfrak{b}=1$) to solve the MMD problem. Unlike RDSA that uses the DSA in Algorithm \ref{alg:DSA}\footnote{The DSA is called in Line $12$ in Algorithm \ref{alg:MMD1}, which is called in Line $6$ in Algorithm \ref{alg:MMD2}.} to estimate shape functions, RDBR applies \eqref{eqn:pbr3} and \eqref{eqn:pbr4}. Even though in each iteration \eqref{eqn:pbr3} and \eqref{eqn:pbr4} cannot give exact estimation, \cite{MMD} proves that the estimation error can be corrected recursively as long as the MIMF's are well-differentiated. The well-differentiation of MIMF's relies on the well-differentiation of phase functions. Denote the set of sampling grid points $\{t_\ell\}_{\ell=0,\dots,L-1}$ in \eqref{eqn:sample} as $\T$. $\T$ is divided into several subsets as follows. For $i,j=1,\dots,K$, $i\neq j$, $m,n=0,\dots,N^h-1$, let 
\[
\T^{ij}_h(m,n)=\left\{ t\in \T: \mod(p_i(t),1)\in[t^h_m,t^h_m+h),\mod(p_j(t),1)\in[t^h_n,t^h_n+h)\right\},
\]
and
\[
\T^{i}_h(m)=\left\{ t\in \T: \mod(p_i(t),1)\in[t^h_m,t^h_m+h)\right\},
\]
then $\T=\cup_{m=0}^{N^h-1} \T^{i}_h(m)=\cup_{m=0}^{N^h-1}\cup_{n=0}^{N^h-1} \T^{ij}_h(m,n)$. 
Let 
\begin{equation}
\label{eqn:D}
D^{ij}_h(m,n)\quad \text{ and }\quad D^{i}_h(m)
\end{equation}
 denote the number of points in  $\T^{ij}_h(m,n)$ and $\T^{i}_h(m)$, respectively.

\begin{definition}
  \label{def:wd}
  Suppose phase functions $p_k(t)= N_k \phi_k(t)$ for $t\in[0,1]$, and $k=1,\dots,K$, where $\phi_k(t)$ satisfies
  \begin{align*}
    \phi_k(t)\in C^\infty,  \quad  1/M\leq | \phi_k'|\leq M.
   \end{align*}
Then the collection of phase functions $\{p_k(t)\}_{1\leq k\leq K}$ is said to be $(M,N,K,h,\beta,\gamma)$-well-differentiated and denoted as $\{p_k(t)\}_{1\leq k\leq K}\subset\WD(M,N,K,h,\beta,\gamma)$, if the following conditions are satisfied:
\begin{enumerate}
\item $N_k\geq N$ for $k=1,\dots,K$;
\item $  \gamma\defeq  \underset{m,n,i\neq j }{\min}D^{ij}_h(m,n)$ satisfies $\gamma>0$ %for all $m,n=0,\dots,N^h-1$, $i,j=1,$ $2$, and $i\neq j$
, where $D^{ij}_h(m,n)$ (and $D^{i}_h(m)$  below) is defined in \eqref{eqn:D};
\item Let \[\beta_{i,j} \defeq \left(\sum_{m=0}^{N^h-1} \frac{1}{D^{i}_h(m) }\left( \sum_{n=0}^{N^h-1} ( D^{ij}_h(m,n)-\gamma )^2 \right)\right)^{1/2}\] for all $i\neq j$, then $\beta\defeq \max\{\beta_{i,j}:i\neq j\}$ satisfies $M^2(K-1)\beta<1$.
\end{enumerate}
\end{definition}

\begin{definition}
  \label{def:wd}
  Suppose 
  \begin{eqnarray*}
f_k(t)=\sum_{n=-N_k/2}^{N_k/2-1} a_{n,k}\cos(2\pi n\phi_k(t))s_{cn,k}(2\pi N_k\phi_k(t))+\sum_{n=-N_k/2}^{N_k/2-1}b_{n,k} \sin(2\pi n\phi_k(t))s_{sn,k}(2\pi N_k\phi_k(t)).
\end{eqnarray*}
is a MIMF of type $(M_0,M,N_k,\epsilon)$ for $t\in[0,1]$, $k=1,\dots,K$, and \[\{p_k(t)=N_k\phi_k(t)\}_{1\leq k\leq K}\subset\WD(M,N,K,h,\beta,\gamma),\] then  $f(t)=\sum_{k=1}^K f_k(t)$ is said to be a well-differentiated superposition of MIMF's of type $(M_0,M,N,K,h,\beta,\gamma,\epsilon)$. Denote the set of all these functions $f(t)$ as $\WS(M_0,M,N,K,h,\beta,\gamma,\epsilon)$.
\end{definition}

We recall again that in the case of $ \mathfrak{b}=1$, RDBR replaces DSA in Line $12$ of Algorithm \ref{alg:MMD1}, which is used in Line $6$ in Algorithm \ref{alg:MMD2}, with \eqref{eqn:pbr3} and \eqref{eqn:pbr4} to estimate shape functions. Under the well-differentiation condition introduced just above, \cite{MMD} proves that the estimation error of  \eqref{eqn:pbr3} and \eqref{eqn:pbr4} (denoted as $s^{E,(j)}_{cn,k}$ and $s^{E,(j)}_{sn,k}$\footnote{The estimation error of a shape function at step $j$, $s^{E,(j)}_{cn,k}$, is defined as the difference of the ground truth regression function of the regression problem at step $j$ and the target shape function at step $j$, $a_{n,k}^{(j-1)}s^{(j-1)}_{cn,k}$,  i.e.,
\[
s^{E,(j)}_{cn,k}= \underset{s:\mathbb{R}\rightarrow \mathbb{R}}{\arg\min}\quad \E\{\left| s(2\pi X_{cn,k}^{(j-1)})-Y_{cn,k}^{(j-1)}\right|^2\}-a_{n,k}^{(j-1)}s^{(j-1)}_{cn,k},
\]
where $(X_{cn,k}^{(j-1)},Y_{cn,k}^{(j-1)})$ has samples $(\text{mod}(N_k\phi_k(t_\ell),1),2^{|\sgn(n)|}\cos(2\pi n\phi_k(t_\ell)) r^{(j-1)}(t_\ell) )$ for $t_\ell$ from \eqref{eqn:sample}, where $r^{(j-1)}$ is the MIMF at step $j$ as used in Line $6$ in Algorithm \ref{alg:MMD2}. Similarly, we define the estimation error $s^{E,(j)}_{sn,k}$ as
\[
s^{E,(j)}_{sn,k}= \underset{s:\mathbb{R}\rightarrow \mathbb{R}}{\arg\min}\quad \E\{\left| s(2\pi X_{sn,k}^{(j-1)})-Y_{sn,k}^{(j-1)}\right|^2\}-b_{n,k}^{(j-1)}s^{(j-1)}_{sn,k},
\]
where $(X_{sn,k}^{(j-1)},Y_{sn,k}^{(j-1)})$ has samples $(\text{mod}(N_k\phi_k(t_\ell),1),2^{|\sgn(n)|}\sin(2\pi n\phi_k(t_\ell)) r^{(j-1)}(t_\ell) )$.
}, respectively) in each iteration of the for-loop for $J_2$ in Algorithm \ref{alg:MMD2} can be corrected recursively: the estimation errors of shape functions in the $j$-th step becomes the target shape function to be estimated in the $(j+1)$-th step; to show the convergence of RDBR, it is sufficient to show that $s^{E,(j)}_{cn,k}$ and $s^{E,(j)}_{sn,k}$ decays as $j\rightarrow \infty$. Theorem \ref{thm:conv3} below (see the proof of Theorem $3.3$ in \cite{MMD}) shows that the estimation error decays to $O(\epsilon)$ as the iteration number goes to infinity. 

Recall that, when we write $O(\cdot)$, $\lesssim$, or $\gtrsim$, the implicit constants may depend on $M_0$, $M$, $K$, $C$, and no other parameters.

\begin{theorem}
\label{thm:conv3} (Convergence of RDBR for MMD) Suppose all shape functions are in the space of Lipschitz continuous functions with a constant $C$ and $\epsilon$ is an accuracy parameter. Assume that $J_1=1$, \eqref{eqn:pbr3} and \eqref{eqn:pbr4} are used to estimate shape functions instead of DSA  in Algorithm \ref{alg:MMD2}. For fixed $\epsilon$, $M_0$, $M$, $K$, and $C$, there exists $h_0(\epsilon,C)$ such that $\forall h<h_0(\epsilon,C)$, there exist $L_0(\epsilon,M_0,M,K,C,h)$ and $N_0(\epsilon,M,K,C,h)$ such that, when $L>L_0$, $N>N_0$, and $f(t)\in\WS(M_0,M,N,K,h,\beta,\gamma,\epsilon)$, we have  
 \[
  \|s^{E,(j)}_{cn,k}\|_{L^2}\leq O(c_0 \epsilon+( \beta  (2M_0+1)(K-1))^j ),
 \]
 and 
  \[
  \|s^{E,(j)}_{sn,k}\|_{L^2}\leq O(c_0 \epsilon+(\beta  (2M_0+1)(K-1))^j )
 \]
for all $j\geq 0$ and $1\leq k\leq K$, where $c_0=\frac{1}{1-\beta  (2M_0+1)(K-1)}$ is a constant number, $s^{E,(j)}_{cn,k}$ and $s^{E,(j)}_{sn,k}$ are defined just before this theorem.
\end{theorem}

As shown in \cite{MMD}, in each regression step, the variation of noise perturbation, which comes from the interference of other components, is bounded by a constant depending only on $M_0$, $M$ and $K$. For the fixed $\epsilon$ and $C$, there exists $h_0(\epsilon,C)$ such that $C^2h^2<\epsilon^2$ if $0<h<h_0$. For the fixed $\epsilon$, $M_0$, $M$, $K$, $C$, and $h$, there exists $L_0(\epsilon,M_0,M,K,C,h)$ such that, if $L>L_0$, then the $L^2$ error of the partition-based regression is bounded by $\epsilon^2$ according to Theorem  \ref{thm:reg}. Under these conditions, one can prove Theorem \ref{thm:conv3} using classical inequalities like the triangle inequality, H{\"o}lder's inequality, and Taylor expansion, following the steps in \cite{MMD} (Theorem $3.3$) and the ideas in \cite{HZYregression} (Theorem $3.5$).

\subsection{Theory of RDSA}

With the theory of RDBR introduced in the previous section, we are ready to prove the convergence of RDSA in this section. The main idea is to prove that RDSA is a special kind of RDBR by the downsampling theorem (aliasing theorem).

\begin{definition}
Suppose $L$, $N$, and $\frac{L}{N}$ are integers. A downsampling operator, denoted as $\cD_{N,L}$ with a factor $N$ is a map from $x\in\mathbb{C}^{L}$ to $y=\cD_{N,L}(x)\in\mathbb{C}^{\frac{L}{N}}$ such that
\[
y[n]=x[nN]
\]
for $n=0,\cdots,\frac{L}{N}-1$.
\end{definition} 

\begin{definition}
Suppose $L$, $N$, and $\frac{L}{N}$ are integers. An aliasing operator, denoted as $\cA_{N,L}$ with a factor $N$ is a map from $x\in\mathbb{C}^{L}$ to $y=\cA_{N,L}(x)\in\mathbb{C}^{\frac{L}{N}}$ such that
\[
y[n]=\sum_{j=0}^{N-1}x[n+j\frac{L}{N}]
\]
for $n=0,\cdots,\frac{L}{N}-1$.
\end{definition}

\begin{theorem}
\label{thm:al}
(Downsampling Theorem) Suppose $L$, $N$, and $\frac{L}{N}$ are integers. For all $x\in\mathbb{C}^L$, it holds that \[\mathcal{F}(\cD_{N,L}(x))=\frac{1}{N}\cA_{N,L}(\mathcal{F}(x)),\] where $\mathcal{F}$ denotes the discrete Fourier transform.
\end{theorem}

The reader is referred to  \cite{MDFT07} for  the proof of Theorem \ref{thm:al}. An immediate result of Theorem \ref{thm:al} is the following convergence theorem for RDSA.

\begin{theorem}
\label{thm:conv4} (Convergence of RDSA for MMD) Suppose all shape functions are in the space of Lipschitz continuous functions with a constant $C$ and $\epsilon$ is an accuracy parameter. Assume that $J_1=1$ and Algorithm \ref{alg:DSA} is used to estimate shape functions in Algorithm \ref{alg:MMD2}. For fixed $\epsilon$, $M_0$, $M$, $K$, and $C$, there exists $h_0(\epsilon,C)$ such that $\forall h<h_0(\epsilon,C)$, there exist $L_0(\epsilon,M_0,M,K,C,h)$ and $\bar{N}_0(\epsilon,M_0,M,K,C,h)$ such that, when $L>L_0$, $N>\bar{N}_0$, $\frac{N}{L}<h_0$, and $f(t)\in\WS(M_0,M,N,K,\frac{N}{L},\beta,\gamma,\epsilon)$, we have  
 \[
  \|s^{E,(j)}_{cn,k}\|_{L^2}\leq O(c_0 \epsilon+( \beta  (2M_0+1)(K-1))^j ),
 \]
 and 
  \[
  \|s^{E,(j)}_{sn,k}\|_{L^2}\leq O(c_0 \epsilon+(\beta  (2M_0+1)(K-1))^j )
 \]
for all $j\geq 0$ and $1\leq k\leq K$, where $c_0=\frac{1}{1-\beta  (2M_0+1)(K-1)}$ is a constant number, $s^{E,(j)}_{cn,k}$ and $s^{E,(j)}_{sn,k}$ are defined just before this theorem.
\end{theorem}

\begin{proof}
In the first part of the proof, we show that \eqref{eqn:s1} and \eqref{eqn:s2} are equivalent to partition-based regression with a step size $\frac{N}{L}$ up to an approximation error due to the NUFFT. Since the approximation error of the NUFFT can be controled within arbitrary accuracy \cite{Dutt}, we assume that this approximation error is $O(\epsilon)$. 

For a MIMF $f(t)$ as defined in \eqref{eqn:m_IMF1}, let us define a uniform grid
\[
\left\{\psi_\ell=\phi(0)+\frac{(\phi(1)-\phi(0))\ell}{L}:0\leq \ell\leq L-1,\ell\in\mathbb{Z}\right\}.
\]
Define a vector $\overrightarrow{cf}\in\mathbb{R}^L$ associated with the function $f(t)$ such that the $\ell$-th entry is $\overrightarrow{cf}[\ell]=\cos(2\pi n\psi_\ell)(f\circ\phi^{-1}(\psi_\ell))$. Define a vector $\vec{s}_{cn}\in\mathbb{R}^L$ associated with the function $a_ns_{cn}(2\pi t)$ such that the $\ell$-th entry is $\vec{s}_{cn}[\ell]=a_ns_{cn}(2\pi t_\ell)$, where $t_\ell$ is from the uniform grid
\begin{equation*}
\left\{t_\ell=\frac{\ell}{L}:0\leq \ell\leq L-1,\ell\in\mathbb{Z}\right\}.
\end{equation*}

By Theorem \ref{thm:al}, we know \[\mathcal{F}^{-1}(\cD_{N,L}(\mathcal{F}(\overrightarrow{cf})))=\frac{1}{N}\cA_{N,L}(\overrightarrow{cf}).\]  By the definition of $\cD_{N,L}$ and $T_N$, and the fact that the right $\mathcal{F}$ in  \eqref{eqn:s1} is carried out via the NUFFT, we see  that \eqref{eqn:s1} is equivalent to
\[
\vec{s}_{cn}= O(\epsilon) +  \frac{2^{|\sgn(n)|}}{N}\cA_{N,L}(\overrightarrow{cf}),
\]
i.e.,
\begin{eqnarray*}
\vec{s}_{cn}[k]&=&  O(\epsilon)+ \frac{2^{|\sgn(n)|}}{N} \sum_{j=0}^{N-1}\overrightarrow{cf}[k+j\frac{L}{N}]\\
&=&  O(\epsilon)+ \frac{2^{|\sgn(n)|}}{N} \sum_{j=0}^{N-1} \cos(2\pi n\psi_{k+j\frac{L}{N}})\left(f\circ\phi^{-1}(\psi_{k+j\frac{L}{N}})\right).
\end{eqnarray*}
If we write the above equation in the terminology of partition-based regression, then the above equation (and hence \eqref{eqn:s1}) is equivalent to
\begin{equation}\label{eqn:pbr5}
a_n s_{cn}(2\pi x)=O(\epsilon)+ \frac{2^{|\sgn(n)|}\sum_{\ell=0}^{L-1}\mathcal{X}_{[t^{N/L}_k,t^{N/L}_{k+1})} (\frac{\text{mod}(N(\psi_\ell-\psi_0),\phi(1)-\phi(0))}{\phi(1)-\phi(0)}) \cos(2\pi n\psi_\ell)\left( f\circ\phi^{-1}(\psi_\ell) \right)}{\sum_{\ell=0}^{L-1}\mathcal{X}_{[t^{N/L}_k,t^{N/L}_{k+1})}(\frac{\text{mod}(N(\psi_\ell-\psi_0),\phi(1)-\phi(0))}{\phi(1)-\phi(0)}) }
\end{equation}
when $x\in [t^{N/L}_k,t^{N/L}_{k+1})$, where $t^{N/L}_k=k \frac{N}{L}$, for $k=0,\cdots,\frac{L}{N}-1$. In this special case of partition-based regression, the samples are \[\{\frac{\text{mod}(N(\psi_\ell-\psi_0),\phi(1)-\phi(0))}{\phi(1)-\phi(0)},\cos(2\pi n\psi_\ell)\left( f\circ\phi^{-1}(\psi_\ell)\right)\}_{\ell=0,\cdots,L-1},\]
where $\frac{\text{mod}(N(\psi_\ell-\psi_0),\phi(1)-\phi(0))}{\phi(1)-\phi(0)}  = t^{N/L}_{mod(\ell,L/N)}\in  [0,1]$ is always on the partition grid points (with a step size $\frac{N}{L}$) of the partition-based regression. In fact, these samples are uniformly distributed on the partition grid points and each grid point has $N$ samples.

Similarly, we see that \eqref{eqn:s2} is equivalent to 
\begin{equation}\label{eqn:pbr6}
b_n s_{sn}(2\pi x)=O(\epsilon)+ \frac{2^{|\sgn(n)|}\sum_{\ell=0}^{L-1}\mathcal{X}_{[t^{N/L}_k,t^{N/L}_{k+1})} (\frac{\text{mod}(N(\psi_\ell-\psi_0),\phi(1)-\phi(0))}{\phi(1)-\phi(0)}) \sin(2\pi n\psi_\ell) \left(f\circ\phi^{-1}(\psi_\ell) \right)}{\sum_{\ell=0}^{L-1}\mathcal{X}_{[t^{N/L}_k,t^{N/L}_{k+1})}(\frac{\text{mod}(N(\psi_\ell-\psi_0),\phi(1)-\phi(0))}{\phi(1)-\phi(0)}) },
\end{equation}
when $x\in [t^{N/L}_k,t^{N/L}_{k+1})$. \eqref{eqn:pbr6} is again a special case of partition-based regression with samples \[\{\frac{\text{mod}(N(\psi_\ell-\psi_0),\phi(1)-\phi(0))}{\phi(1)-\phi(0)},\sin(2\pi n\psi_\ell) \left(f\circ\phi^{-1}(\psi_\ell)\right)\}_{\ell=0,\cdots,L-1}.\] The step size of the sampling domain $[0,1]$ is $\frac{N}{L}$.

Recall that RDBR uses formulas \eqref{eqn:pbr3} and \eqref{eqn:pbr4} to estimate shape functions, and these formulas come from partition-based regression with sampling points \[\{\text{mod}(N\phi(t_\ell),1),\cos(2\pi n\phi(t_\ell)) f(t_\ell)\}_{0\leq \ell\leq L-1}\] and \[\{\text{mod}(N\phi(t_\ell),1),\sin(2\pi n\phi(t_\ell)) f(t_\ell)\}_{0\leq \ell\leq L-1},\] respectively. The step size of the sampling domain $[0,1]$ is a fixed parameter $h$. 

By Theorem \ref{thm:conv3}, we see that, if RDBR was used to estimate shape functions (i.e.,  formulas \eqref{eqn:pbr3} and \eqref{eqn:pbr4} were used), then for fixed $\epsilon$, $M_0$, $M$, $K$, and $C$, there exists $h_0(\epsilon,C)$ such that $\forall h<h_0(\epsilon,C)$, there exist $L_0(\epsilon,M_0,M,K,C,h)$ and $N_0(\epsilon,M_0,M,K,C,h)$ such that, when $L>L_0(\epsilon,M_0,M,K,C,h)$, $N>N_0(\epsilon,M_0,M,K,C,h)$, and $f(t)\in\WS(M_0,M,N,K,h,\beta,\gamma,\epsilon)$, we have  
 \[
  \|s^{E,(j)}_{cn,k}\|_{L^2}\leq O(c_0 \epsilon+( \beta  (2M_0+1)(K-1))^j ),
 \]
 and 
  \[
  \|s^{E,(j)}_{sn,k}\|_{L^2}\leq O(c_0 \epsilon+(\beta  (2M_0+1)(K-1))^j )
 \]
for all $j\geq 0$ and $1\leq k\leq K$, where $c_0=\frac{1}{1-\beta  (2M_0+1)(K-1)}$ is a constant number, $s^{E,(j)}_{cn,k}$ and $s^{E,(j)}_{sn,k}$ are defined just before this theorem.

Hence, in the second part of the proof of Theorem \ref{thm:conv4} for RDSA, we only need to clarify the conditions under which the estimations by \eqref{eqn:pbr3} and \eqref{eqn:pbr4} are almost the same as those by \eqref{eqn:pbr5} and \eqref{eqn:pbr6}, respectively, up to a small difference $O(\epsilon)$.

Under the conditions of Theorem \ref{thm:conv3}, as mentioned right after Theorem \ref{thm:conv3} in this paper, in each step of regression in \eqref{eqn:pbr3} and \eqref{eqn:pbr4}, the estimated regression function only differs to the ground truth regression function up to an $L^2$ error bounded by $\epsilon^2$. In more particular, Theorem \ref{thm:reg} gives the error bound as follows
\[
O\left(\frac{\sigma^2+\|s^R\|^2_{L^\infty}}{L\cdot h}+C^2 h^2\right),
\]
where $\sigma^2$ is the variation of noise perturbation (coming from the interference between different components) and $\sigma^2$ is bounded by a constant depending only on $M_0$, $M$ and $K$; $s^R$ denotes the ground truth regression function for the regression problem and it has an $L^\infty$-norm depending on $M_0$, $M$ and $K$ as well; $C$ is the Lipschitz continuity constant; $L$ is the number of samples; and $h$ is the step size of the partition-based regression. Hence,  there exists $h_0(\epsilon,C)$ such that $\forall h<h_0(\epsilon,C)$, there exists $L_0(\epsilon,M_0,M,K,C,h)$ such that, when $L>L_0(\epsilon,M_0,M,K,C,h)$ we have 
\[
O\left(\frac{\sigma^2+\|s^R\|^2_{L^\infty}}{L\cdot h}+C^2 h^2\right)\lesssim O(\epsilon^2).
\]

Similarly by Theorem \ref{thm:reg}, we see that the estimated regression function by \eqref{eqn:pbr5} and \eqref{eqn:pbr6} has an $L^2$ error bounded by
\[
O\left(\frac{\sigma^2+\|s^R\|^2_{L^\infty}}{L\cdot \frac{N}{L}}+C^2\frac{N^2}{L^2}\right)=O(\frac{1}{N}+\frac{N^2}{L^2}).
\]
Hence, there exists $\dot{N}_0(\epsilon,M_0,M,K,C)$ such that, when $L>L_0(\epsilon,M_0,M,K,C,h)$, $N>\dot{N}_0(\epsilon,M_0,M,K,C)$, $\frac{N}{L}<h_0(\epsilon,C)$, we have
\[
O\left(\frac{\sigma^2+\|s^R\|^2_{L^\infty}}{L\cdot \frac{N}{L}}+C^2\frac{N^2}{L^2}\right)\lesssim O(\epsilon^2).
\]

Hence, following the proof of Theorem $3.3$ in \cite{MMD}, we can prove that for fixed $\epsilon$, $M_0$, $M$, $K$, and $C$, there exists $h_0(\epsilon,C)$ such that $\forall h<h_0(\epsilon,C)$, there exist $L_0(\epsilon,M_0,M,K,C,h)$ and \[\bar{N}_0(\epsilon,M_0,M,K,C,h):=\max\{N_0(\epsilon,M_0,M,K,C,h),\dot{N}_0(\epsilon,M_0,M,K,C)\}\] such that, when $L>L_0(\epsilon,M_0,M,K,C,h)$, $N>\bar{N}_0(\epsilon,M_0,M,K,C,h)$, $\frac{N}{L}<h_0(\epsilon,C)$, and $f(t)\in\WS(M_0,M,N,K,\frac{N}{L},\beta,\gamma,\epsilon)$, we have  
 \[
  \|s^{E,(j)}_{cn,k}\|_{L^2}\leq O(c_0 \epsilon+( \beta  (2M_0+1)(K-1))^j ),
 \]
 and 
  \[
  \|s^{E,(j)}_{sn,k}\|_{L^2}\leq O(c_0 \epsilon+(\beta  (2M_0+1)(K-1))^j )
 \]
for all $j\geq 0$ and $1\leq k\leq K$, where $c_0=\frac{1}{1-\beta  (2M_0+1)(K-1)}$ is a constant number, $s^{E,(j)}_{cn,k}$ and $s^{E,(j)}_{sn,k}$ are defined just before this theorem.

The reason for requiring \[f(t)\in\WS(M_0,M,N,K,\frac{N}{L},\beta,\gamma,\epsilon),\] instead of \[f(t)\in\WS(M_0,M,N,K,h,\beta,\gamma,\epsilon),\] is that the partition-based regression in \eqref{eqn:pbr5} and \eqref{eqn:pbr6} has a step size $\frac{N}{L}$ instead of $h$.

\end{proof}

\section{Numerical Examples\label{sec:NumEx}}

In this section, some numerical examples of synthetic and real data are provided to demonstrate the proposed properties of RDSA. In all synthetic examples, we assume the instantaneous phases and amplitudes are known and only focus on verifying the RDSA in Section \ref{sec:GMD} and its convergence theory in Section \ref{sec:cov}. In real examples, we apply the one-dimensional highly redundant synchrosqueezed wave packet transform (SSWPT) \cite{1DSSWPT,Robustness} to estimate instantaneous phases and amplitudes as inputs of RDSA.  The implementation of SSWPT is publicly available in SynLab\footnote{Available at https://github.com/HaizhaoYang/SynLab.}. Some more packages for estimating instantaneous frequencies can be found in \cite{8081384}. The code for the RDSA is available online as well in a MATLAB package named DeCom\footnote{Available at https://github.com/HaizhaoYang/DeCom.}.

Let us summarize the main parameters in the above packages and in Algorithm \ref{alg:MMD2}. In SynLab, main parameters are
\begin{itemize}
\item $s$: a geometric scaling parameter;
\item $rad$: the support size of the mother wave packet in the Fourier domain;
\item $red$: a redundancy parameter, the number of frames in the wave packet transform;
\item $\epsilon_{sst}$: a threshold for the wave packet coefficients.
\end{itemize}
In Algorithm \ref{alg:MMD2}, main parameters are
\begin{itemize}
\item $J_1$: the maximum number of iterations allowed in Algorithm \ref{alg:MMD1};
\item $J_2$: the maximum number of iterations allowed in Algorithm \ref{alg:MMD2};
\item $M_0$ and $L_s$: bandwidth parameters;
\item $\epsilon_1=\epsilon_2=\epsilon$: the accuracy parameter.
\end{itemize}
For the purpose of convenience, the synthetic data is defined on $[0,1]$ and sampled on a uniform grid. All these parameters in different examples are summarized in Table \ref{tab:1}.

\begin{table}[htp]
\centering
\begin{tabular}{rcccccccccc}
\toprule
  figure  & $s$
                             & $rad$ & $red$ & $\epsilon_{sst}$ & $J_1$ & $J_2$ & $M_0$ & $\epsilon$ & $L_s$ & $L$ \\
\toprule
 3   & -- & -- & -- & -- & 10 & 200 & 20 & 1e-13 & 2000 & -- \\
4, 5, 6   & 0.5 & 1.5 & 8 & 1e-3 & 10 & 200 & -- & 1e-6 & 5000 & $2^{16}$ \\
7, 8, 9    & 0.5 & 1 & 8 & 1e-3 & 10 & 200 & -- & 1e-6 & 5000 & $4000$ \\
10, 11   & -- & -- & -- & -- & 10 & 200 & 10 & 1e-6 & 2000 & $2^{15}$ \\
12, 13, 14   & 0.5 & 1.5 & 8 & 1e-3 & 10 & 200 & 40 & 1e-6 & 1000 & $2^{16}$ \\
\bottomrule
\end{tabular}
\caption{Parameters in SynLab and Algorithm \ref{alg:MMD2}. The notation ``--" means the corresponding parameter is not used or will be specified later in the example.}
\label{tab:1}
\end{table}

\subsection{Convergence of RDSA}

In this section, we provide numerical examples to verify the convergence theory of RDSA in Section \ref{sec:cov}. For a fixed accuracy parameter $\epsilon$, Theorem \ref{thm:conv4} shows that as long as instantaneous frequencies are sufficiently high and the number of samples is large enough, RDSA is able to estimate shape functions from a class of superpositions of MIMF's. The residual error in the iterative scheme linearly converges to a quantity of order $\epsilon$. Since it is difficult to specify the relation of the rate of convergence and other parameters explicitly in the analysis, we provide numerical examples to study this rate quantitatively.

\begin{figure}[ht!]
  \begin{center}
    \begin{tabular}{cccc}
    \includegraphics[height=1.3in]{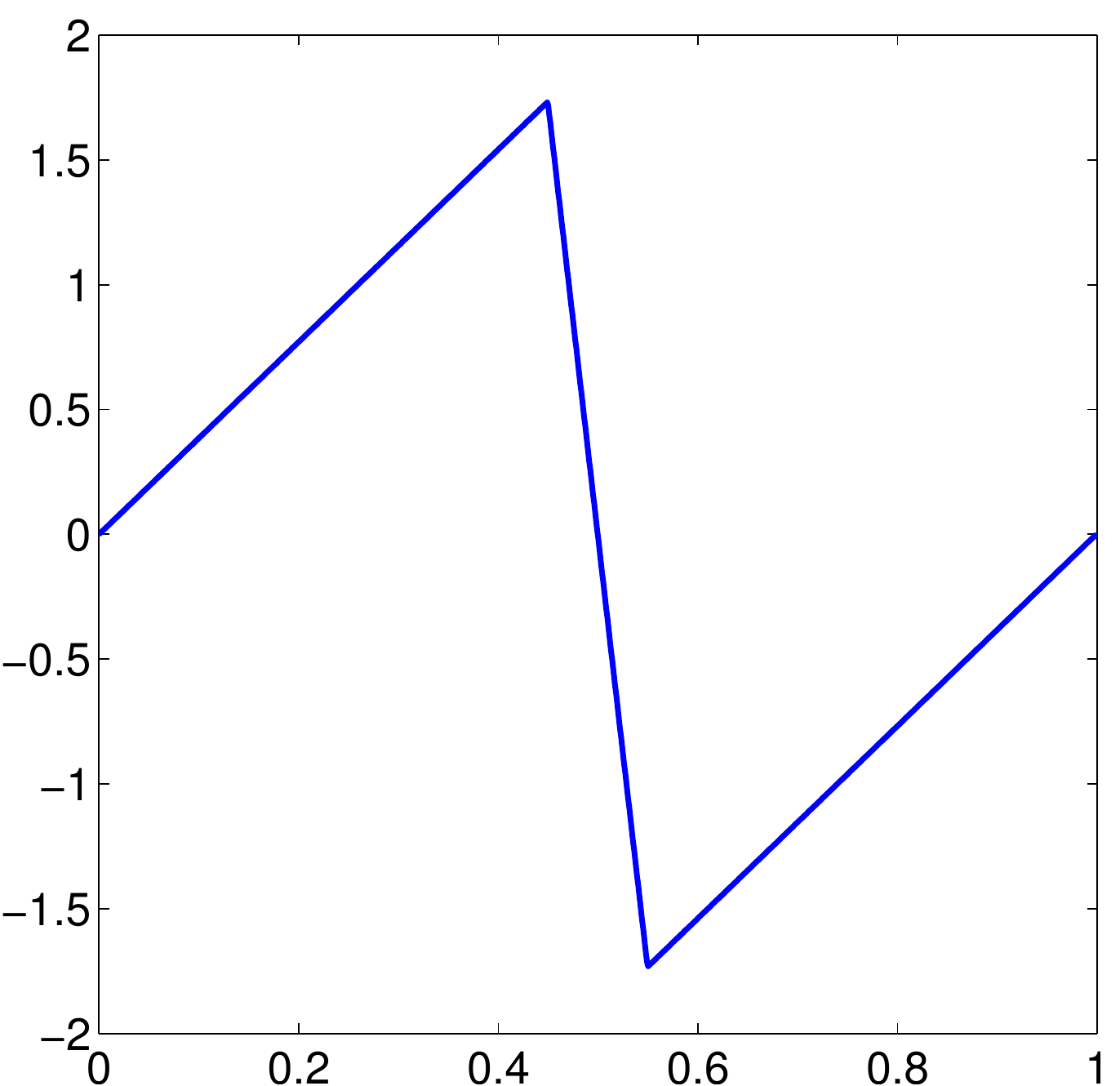}  & \includegraphics[height=1.3in]{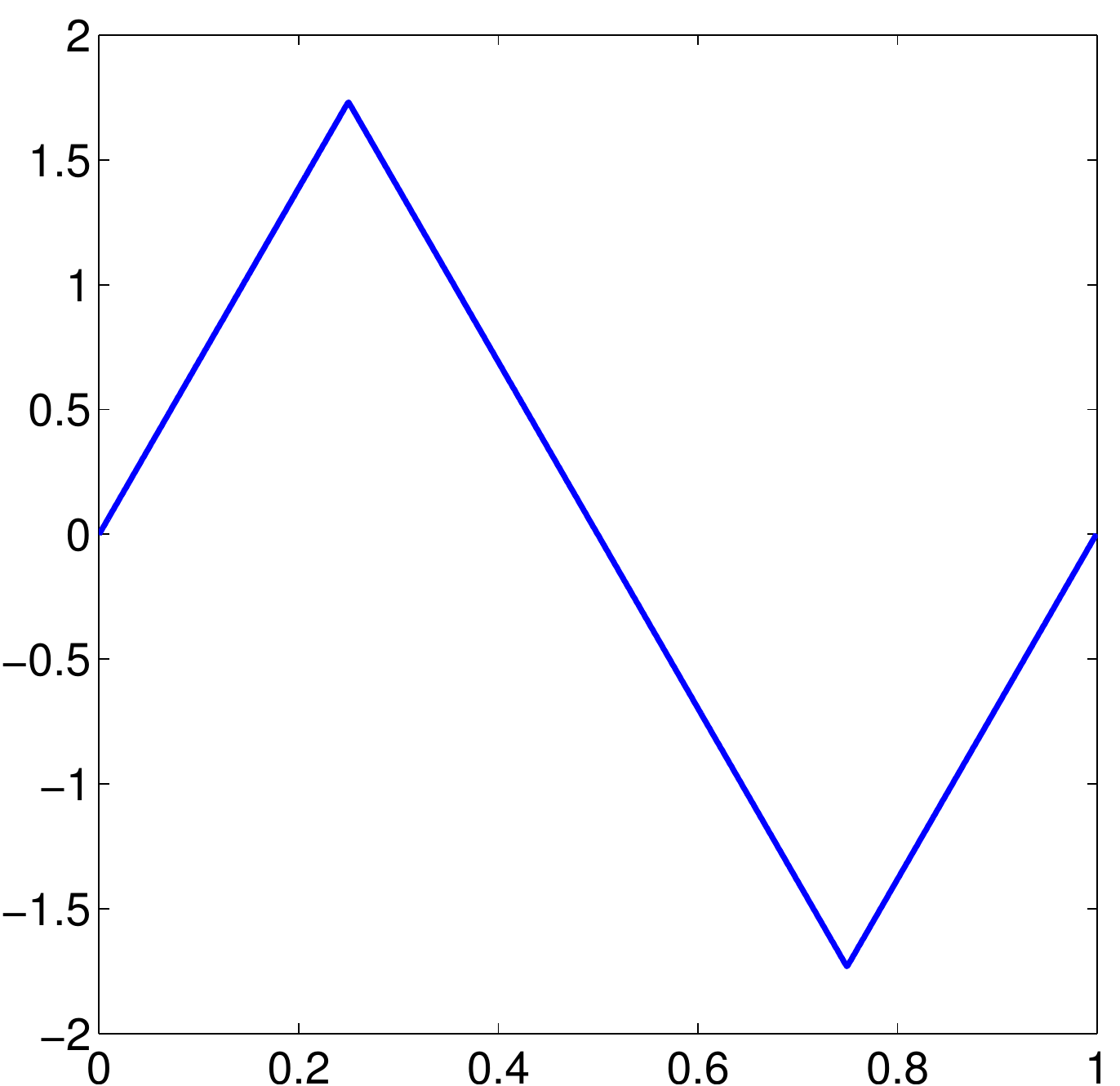}  &   \includegraphics[height=1.3in]{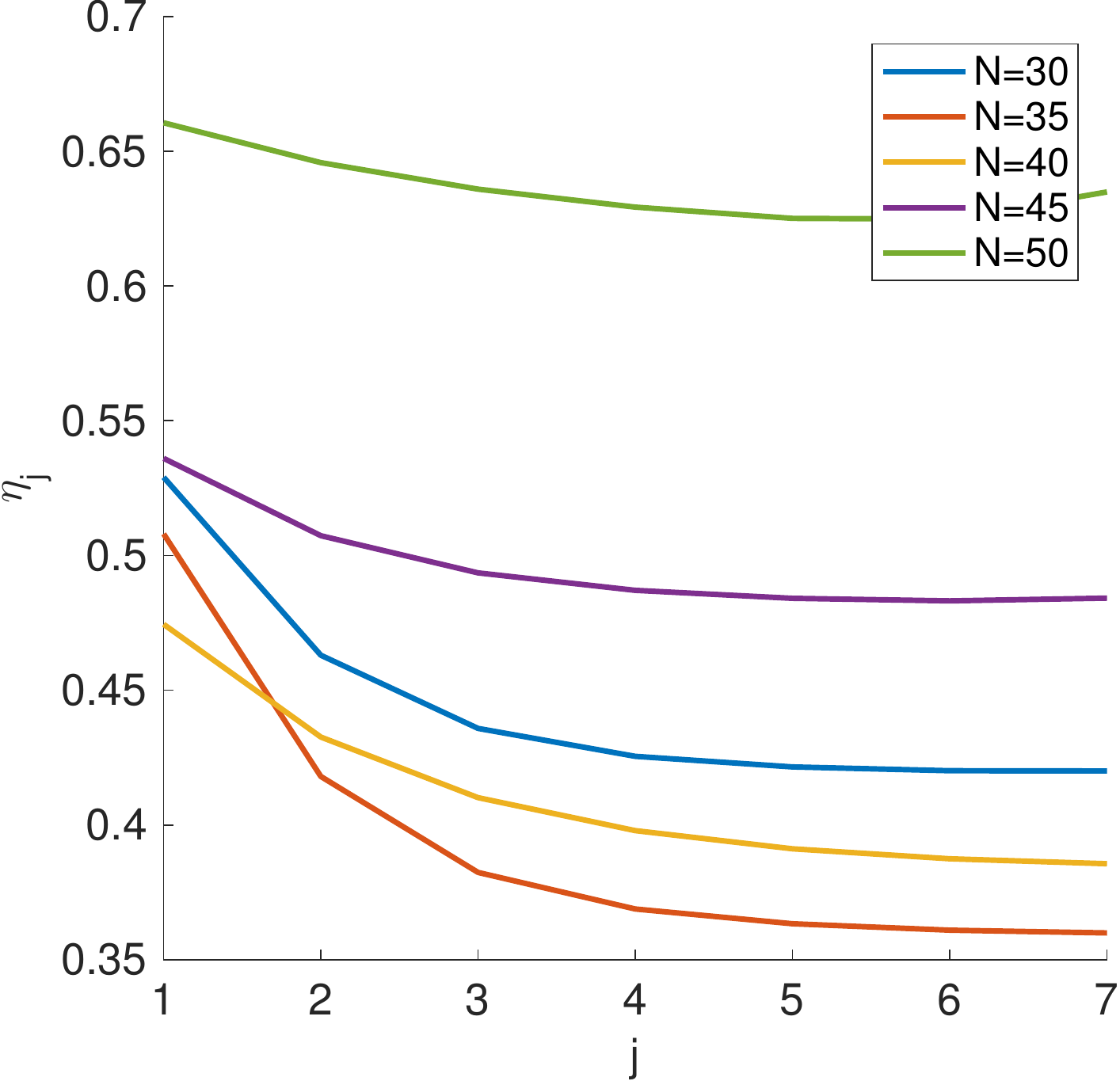}  &
   \includegraphics[height=1.3in]{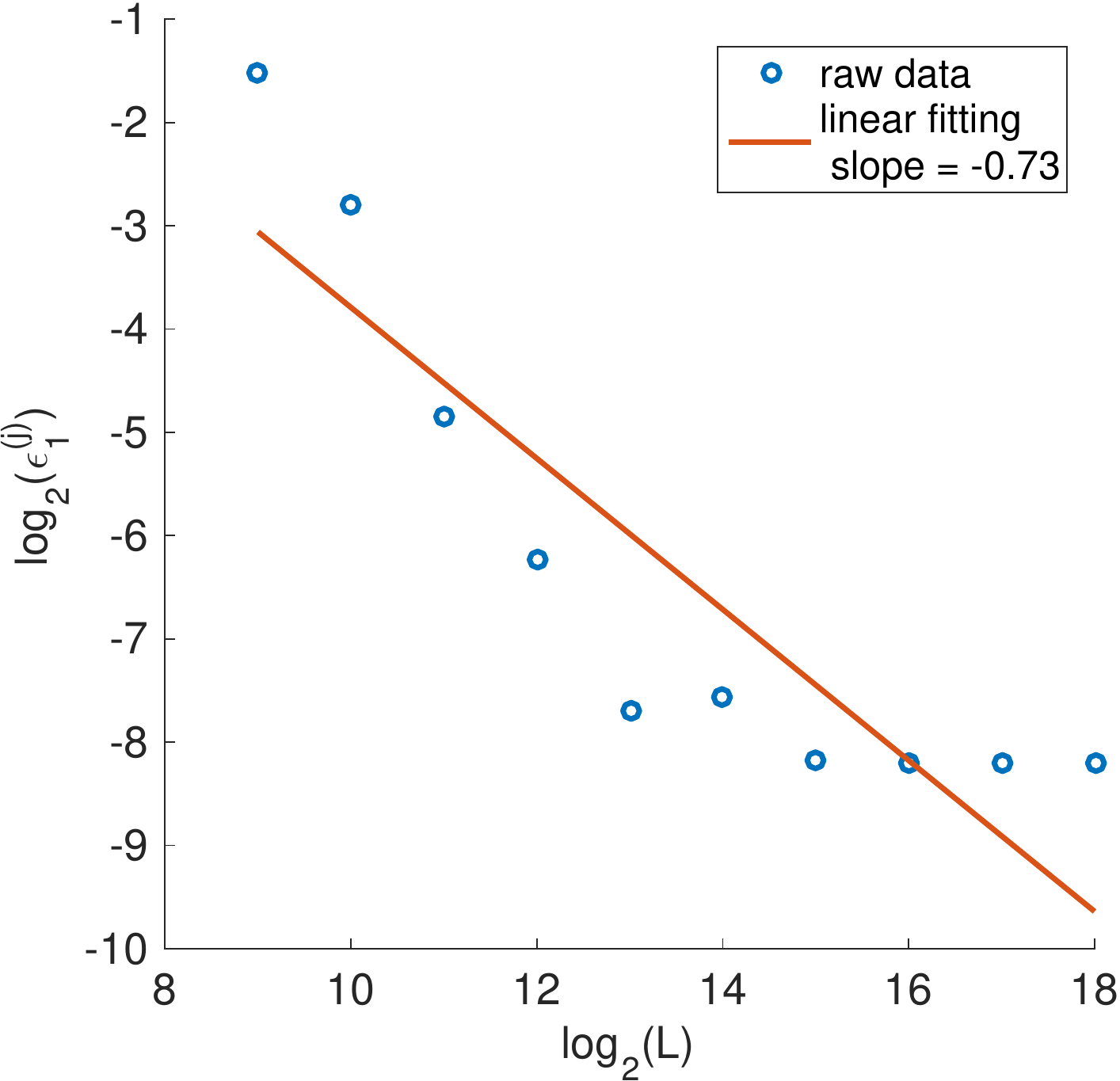}  \\
   (a) & (b) & (c) & (d)
    \end{tabular}
  \end{center}
  \caption{(a) Shape function $s_1$ in \eqref{eqn:ex2}. (b) Shape function $s_2$ in \eqref{eqn:ex2}. (c) Estimated convergence rates $\beta$ in different iteration steps when different values are assigned to $N$ in \eqref{eqn:ex2}. (d) The relation of the final residual norm $\epsilon^{(j)}$ (after the RDSA has been terminated) and the number of samples $L$.} 
\label{fig:3}
\end{figure}

In all examples in this section, we consider a simple case when the signal has two components with piecewise linear and continuous shapes. This makes it easier to verify the convergence analysis. For example, we consider a signal of the form
\begin{equation}
\label{eqn:ex2}
f(t) = f_1(t)+f_2(t),
\end{equation}
where
\[
f_1(t) = s_1(2\pi N\phi_1(t))= s_1\left(2\pi N(t+0.006\sin(2\pi t))\right)
\] 
and
\[
f_2(t) =s_2(2\pi N\phi_2(t))= s_2\left(2\pi N(t+0.006\cos(2\pi t))\right),
\] 
$s_1(2\pi t)$ and $s_2(2\pi t)$ are shape functions defined on $[0,1]$ as shown in Figure \ref{fig:3} (left). Here $f_1$ and $f_2$ can be considered as two IMFs as well as two MIMFs by definition.

First, we fix the number of samples $L=2^{19}$, vary the parameter $N$ in \eqref{eqn:ex2}, and estimate the convergence rate numerically. By Theorem \ref{thm:conv4} (adapted to the example in this section), the residual norm $\epsilon_1$ in Algorithm \ref{alg:MMD2} converges to $O(\epsilon)$ as follows
\[
\epsilon_1^{(j)}=O(\epsilon)+\beta^jO(1).
\]
Hence, if we define a sequence $\{\mu_j\}$ by
\[
\mu_{j} =\log( |\epsilon_1^{(j-1)}-\epsilon_1^{(j)}|).
\]
and a sequence $\{\eta_j\}$ by
\[
\eta_j = \mu_j-\mu_{j+1},
\]
then $\eta_j$ approximately quantifies the convergence in the $j$th iteration, and should be nearly a constant close to $-\log(\beta)$. Figure \ref{fig:3} (c) visualizes the sequences $\{\eta_j\}$ generated from different signals with various $N$'s. It shows that when $N$ is sufficiently large, $\{\eta_j\}$ are approximately a constant for all $j$ and hence the convergence is linear; when $N$ is small, RDSA converges sublinearly since $\eta_j>0$ for all $j$ and $\{\eta_j\}$ decays as $j$ becomes large. After a few iterations, the residual error has been small enough. Hence, we do not show the results when the iteration number is larger than $7$.

Second, we fixed $N=100$, vary the number of samples $L=2^m$ with $m=9,10,\dots,18$, and show the accuracy of RDSA after it converges. To obtain results with an accuracy as high as possible, we let $J_2=200$ and $\epsilon=1e-13$. Figure \ref{fig:3} (d) shows that the final residual norm $\epsilon_1$ after RDSA converged essentially decays in $L$.

\subsection{The speedup of RDSA against RDBR}

In this section, we compare the computational efficiency of RDSA proposed in this paper and RDBR in \cite{MMD}. In this comparison, we still adopt the simple example in \eqref{eqn:ex2} and only compare the computational time of one iteration in RDSA and RDBR, i.e., the time (denoted as $t_{RDBR}$) for performing Algorithm $2$ in \cite{MMD} with $J=1$ and the time (denoted as $t_{RDSA}$) for performing Algorithm \ref{alg:DSA} in this paper for computing only one shape function. The speedup of RDSA against RDBR, i.e., $t_{RDBR}/t_{RDSA}$, is shown in Table \ref{tab:2} for various $L$'s and $N$'s. Since the main computational cost for RDSA is the NUFFT, which has a computational complexity $O(L\log L)$ for a problem of size $L$, the RDSA is highly efficient. The speedup of RDSA against RDBR is much more prominent as $L$ increases. Hence, RDSA is a more practical algorithm for MMD than RDBR when the problem size is large. 

\begin{table}[htp]
\centering
\begin{tabular}{rcccccccccc}
  \toprule
N$\setminus$L  & $2^{10}$  & $2^{11}$  & $2^{12}$  & $2^{13}$  & $2^{14}$  & $2^{15}$  & $2^{16}$  & $2^{17}$  & $2^{18}$  & $2^{19}$   \\
\toprule
50 & 59 & 154 & 287 & 713 & 217 & 900 & 946 & 1246 & 1396 & 1278  \\ 
70 & 525 & 682 & 723 & 1364 & 771 & 916 & 1136 & 1293 & 1645 & 1650  \\ 
90 & 482 & 436 & 593 & 645 & 1184 & 611 & 945 & 1451 & 1501 & 2086  \\ 
110 & 382 & 505 & 685 & 700 & 1056 & 1026 & 966 & 1253 & 1416 & 1298  \\ 
 
\bottomrule
\end{tabular}
\caption{The speedup of RDSA against RDBR, i.e., $t_{RDBR}/t_{RDSA}$, for estimating shape functions in one iteration of the recursive scheme. The above table shows the speedup for various numbers of samples $L$ and essential frequencies $N$.}
\label{tab:2}
\end{table}

%
%\begin{figure}[ht!]
%  \begin{center}
%    \begin{tabular}{c}
%      \includegraphics[width=2in]{Pictures/RDSA_fig4-eps-converted-to.pdf}  \\
%    \end{tabular}
%  \end{center}
%  \caption{The scaling of RDSA and RDBR }
%\label{fig:4}
%\end{figure}

\subsection{Analysis of MIMF's in real data}

In this section, we apply RDSA to analyze MIMF's in real applications. We adopt the same numerical examples in \cite{MMD} to compare the performance of RDSA and RDBR. To save space, only the results of RDSA will be provided. The reader is referred to \cite{MMD} for the results of RDBR as comparison. The first example is an ECG record from a normal subject and the second example is a motion-contaminated ECG record. More details about the ECG data can be found in \url{https://www.physionet.org/physiobank/database/}. We compute the band-limited multiresolution approximations of the first example and visualize them in Figure \ref{fig:10_1}, \ref{fig:10_2}, and \ref{fig:10_3}; the band-limited multiresolution approximations of the second example are plotted in Figure \ref{fig:11_1}, \ref{fig:11_2}, and \ref{fig:11_3}. Note that when the bandwidth of the multiresolution approximation increases, the approximation error decreases, and finer variation of the time series can be captured. Figure \ref{fig:10_3} and \ref{fig:11_3} show the first five shape functions of these two examples, respectively; all shape functions vary a lot at different level of resolution. The actual time-varying shape of an ECG signal we see in the raw data is not exactly any single shape function in the shape function series; they are actually the results of all shape functions in the shape function series. The results by RDSA validate the MIMF model again. Compared to RDBR, the residual $f(t)-\mathcal{M}_{40}(f)(t)$ by RDSA in Figure \ref{fig:10_2} and \ref{fig:11_2} is smaller, which implies that RDSA is better to handle fine details of the signal than RDBR. %According to the discussion in Section \ref{sec:cov}, RDBR usually uses a larger step size in the partition-based regression for the purpose of reducing noise effect; while RDSA uses a small built-in step size in the partition-based regression. Therefore, in the case of weak noise, RDSA can handle fine features. 

\begin{figure}[ht!]
  \begin{center}
    \begin{tabular}{c}
      \includegraphics[width=6in]{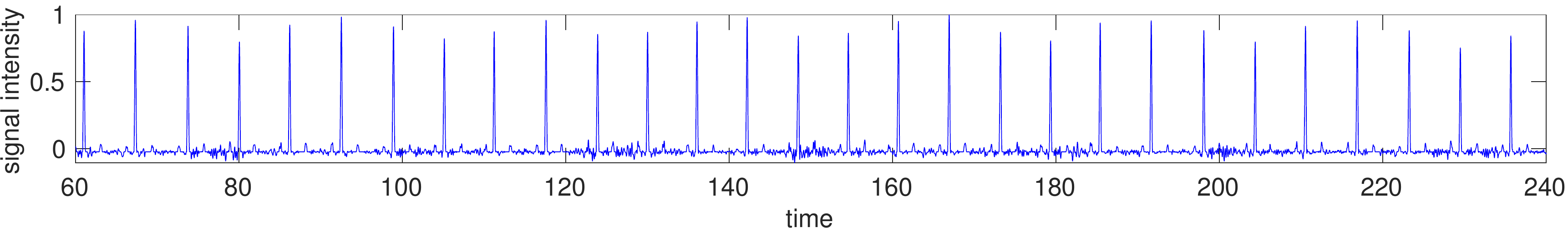}  \\
      An ECG record from a normal subject $f(t)$\\
      \includegraphics[width=6in]{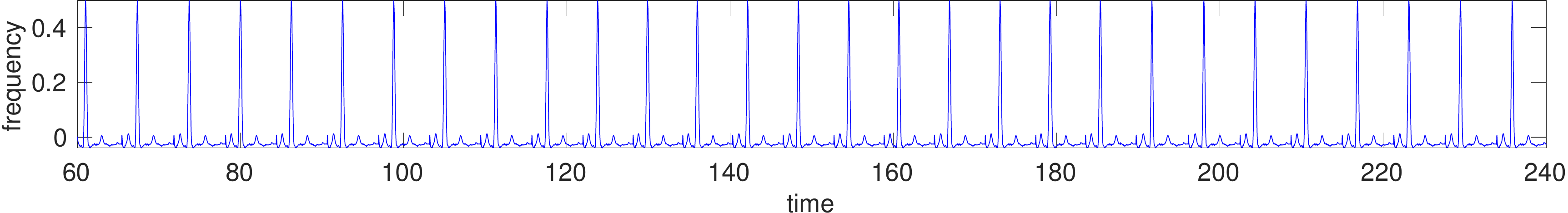}  \\
     $0$-banded multiresolution approximation $\mathcal{M}_0(f)(t)$\\
      \includegraphics[width=6in]{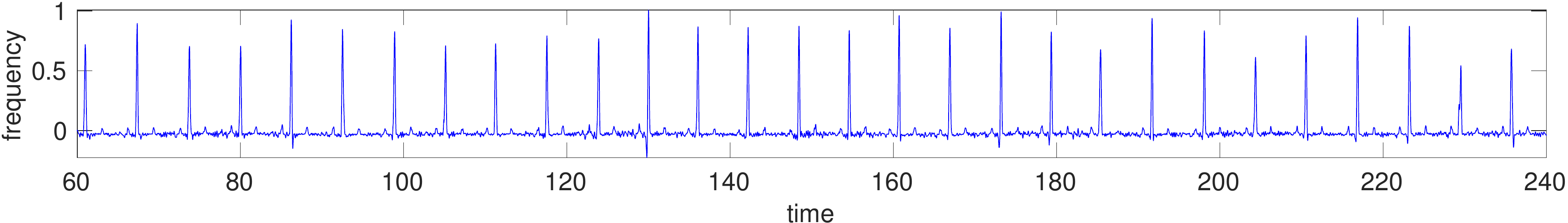}  \\
     $20$-banded multiresolution approximation $\mathcal{M}_{20}(f)(t)$\\
      \includegraphics[width=6in]{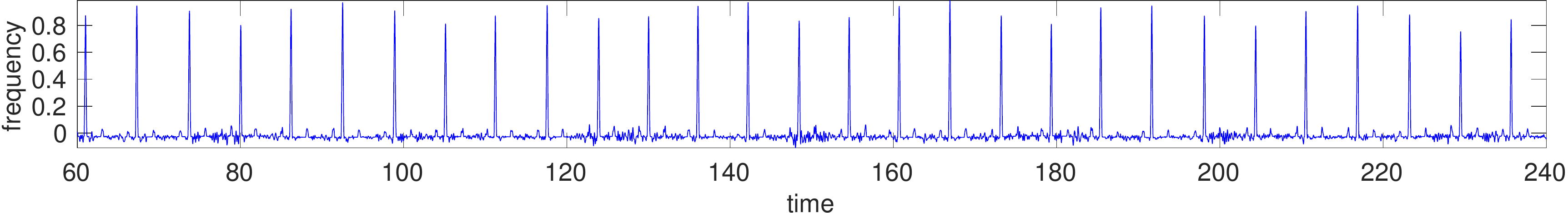}  \\
      $40$-banded multiresolution approximation $\mathcal{M}_{40}(f)(t)$
    \end{tabular}
  \end{center}
  \caption{Multiresolution approximations of an ECG record from a normal subject.}
\label{fig:10_1}
\end{figure}

\begin{figure}[ht!]
  \begin{center}
    \begin{tabular}{c}
      \includegraphics[width=6in]{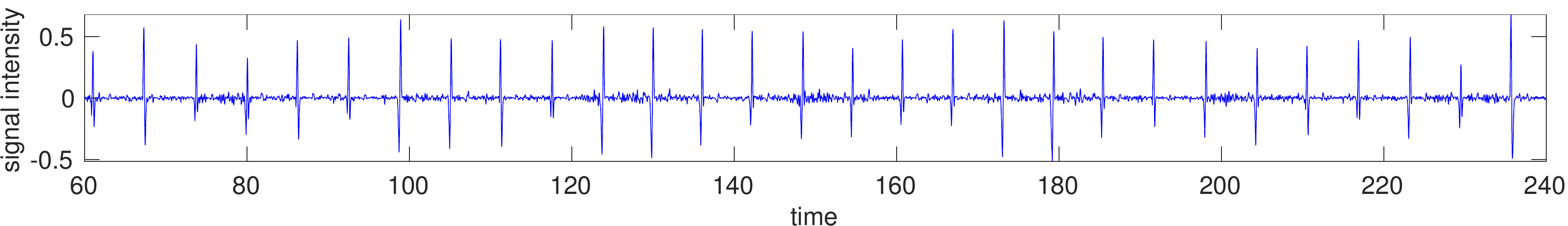}  \\
      $f(t)-\mathcal{M}_0(f)(t)$ \\
      \includegraphics[width=6in]{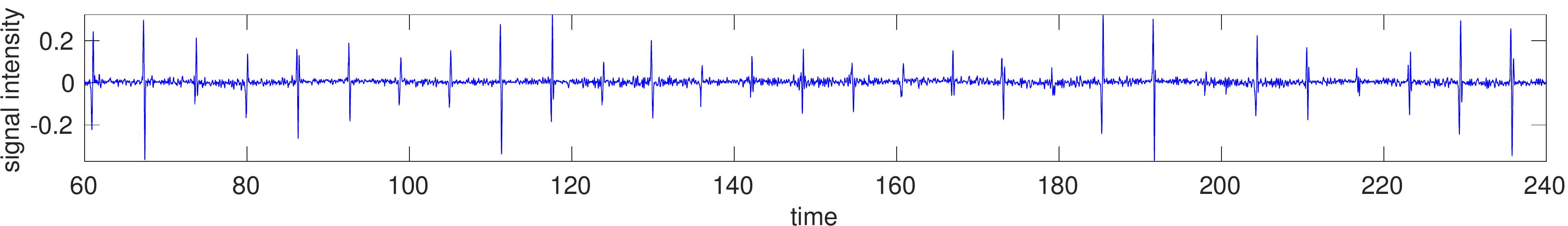}  \\
      $f(t)-\mathcal{M}_{20}(f)(t)$ \\
      \includegraphics[width=6in]{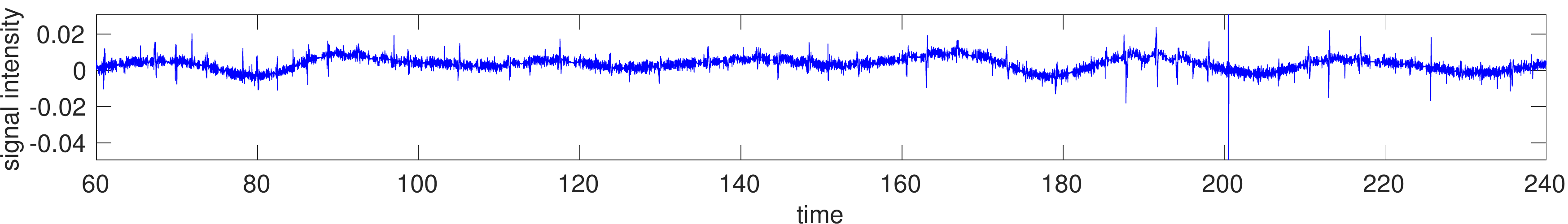}  \\
      $f(t)-\mathcal{M}_{40}(f)(t)$ 
    \end{tabular}
  \end{center}
  \caption{The residual of the multiresolution approximations of an ECG record from a normal subject in Figure \ref{fig:10_1}.}
\label{fig:10_2}
\end{figure}

\begin{figure}[ht!]
  \begin{center}
    \begin{tabular}{ccccc}
      \includegraphics[width=1.1in]{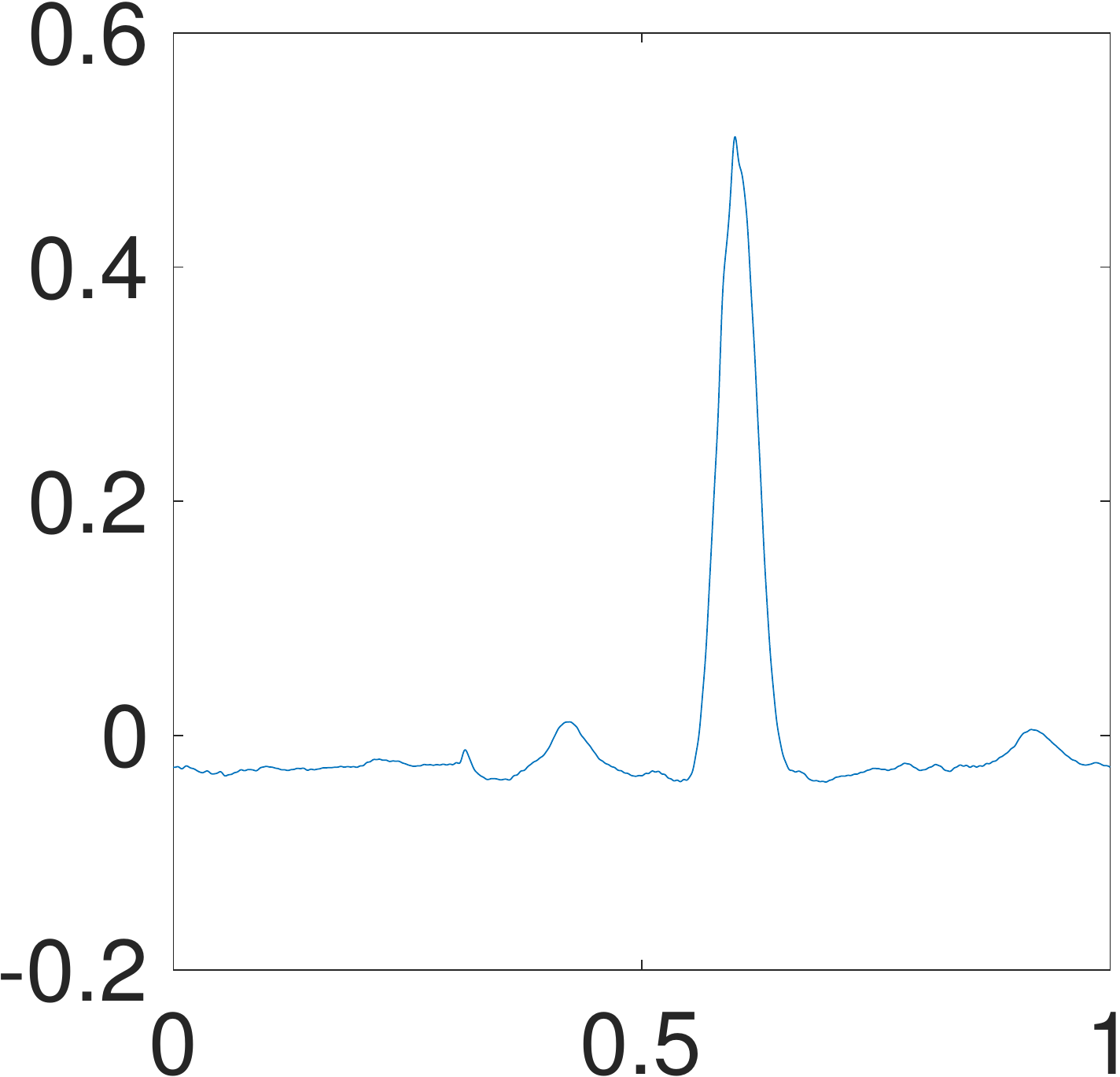}   &
      \includegraphics[width=1.1in]{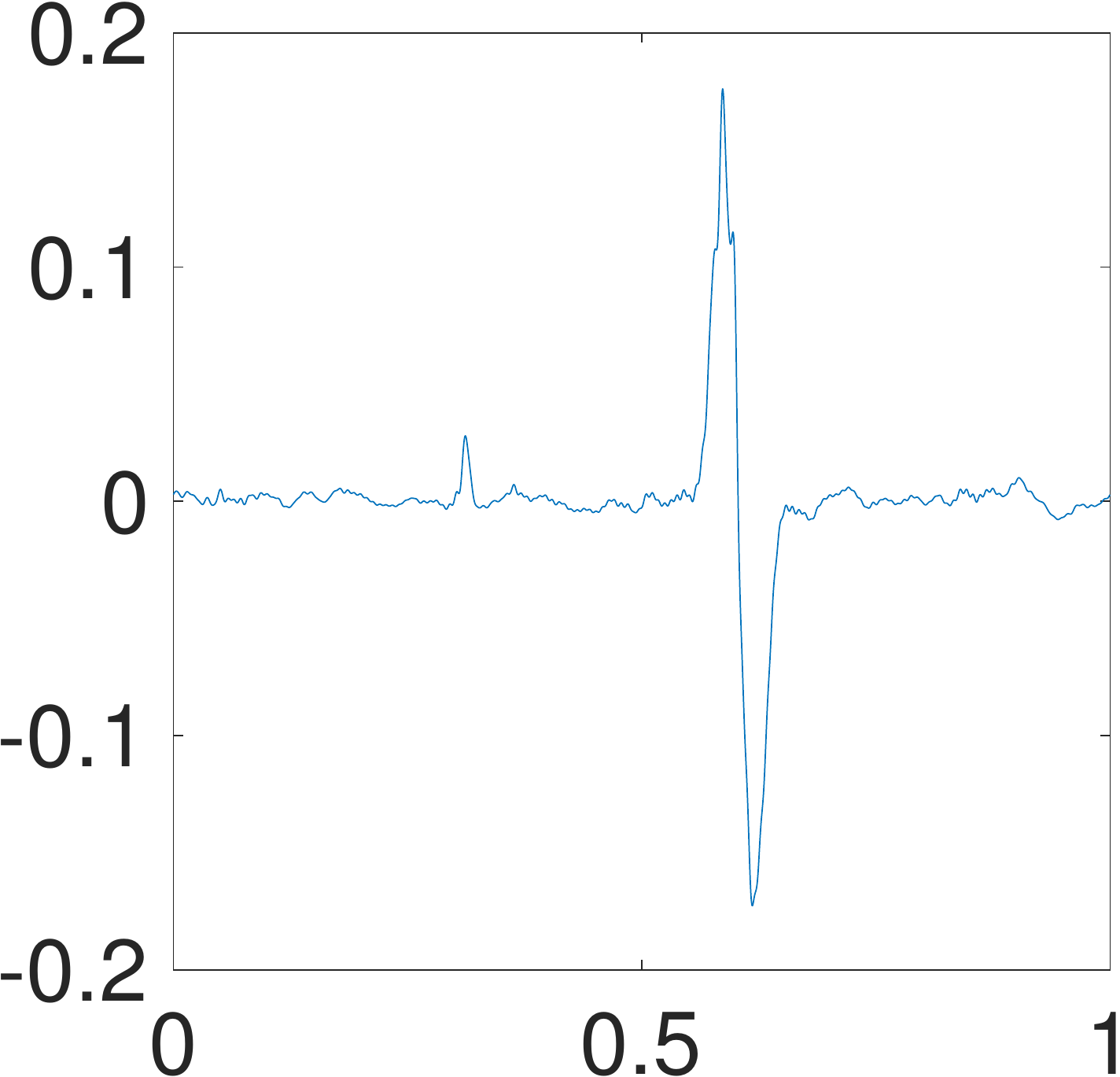}   &
      \includegraphics[width=1.05in]{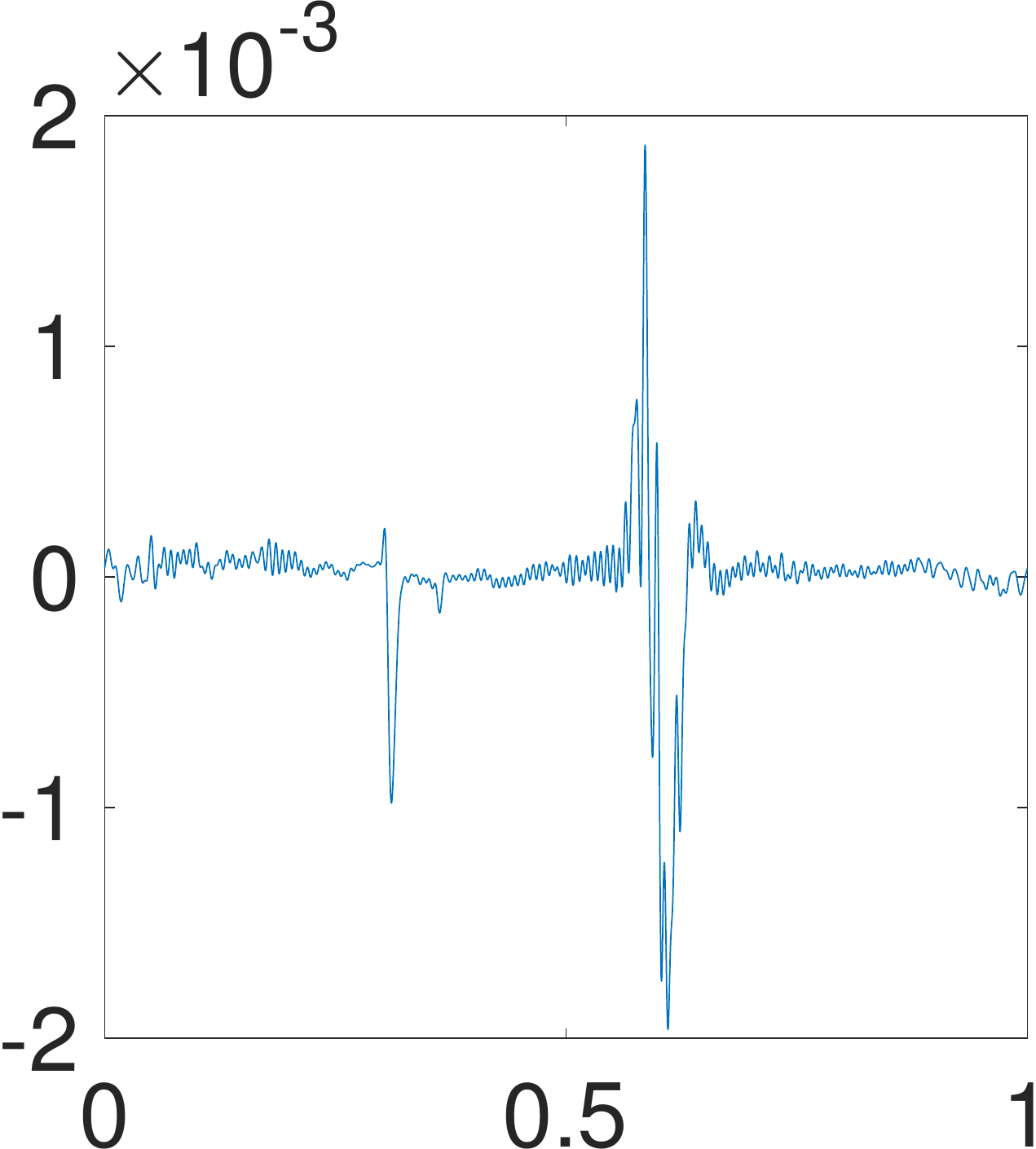}   &
      \includegraphics[width=1.1in]{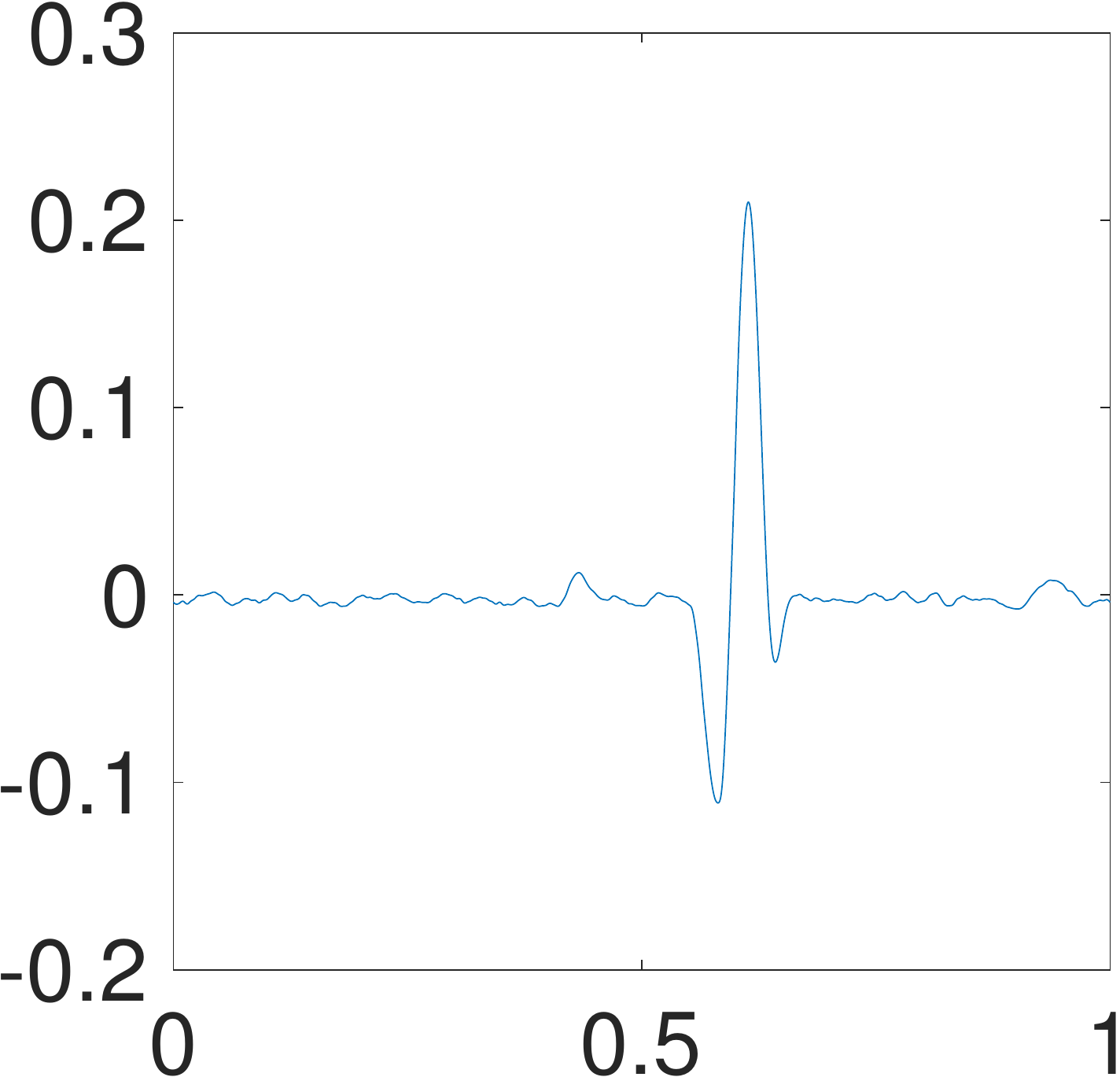}   &
      \includegraphics[width=1.1in]{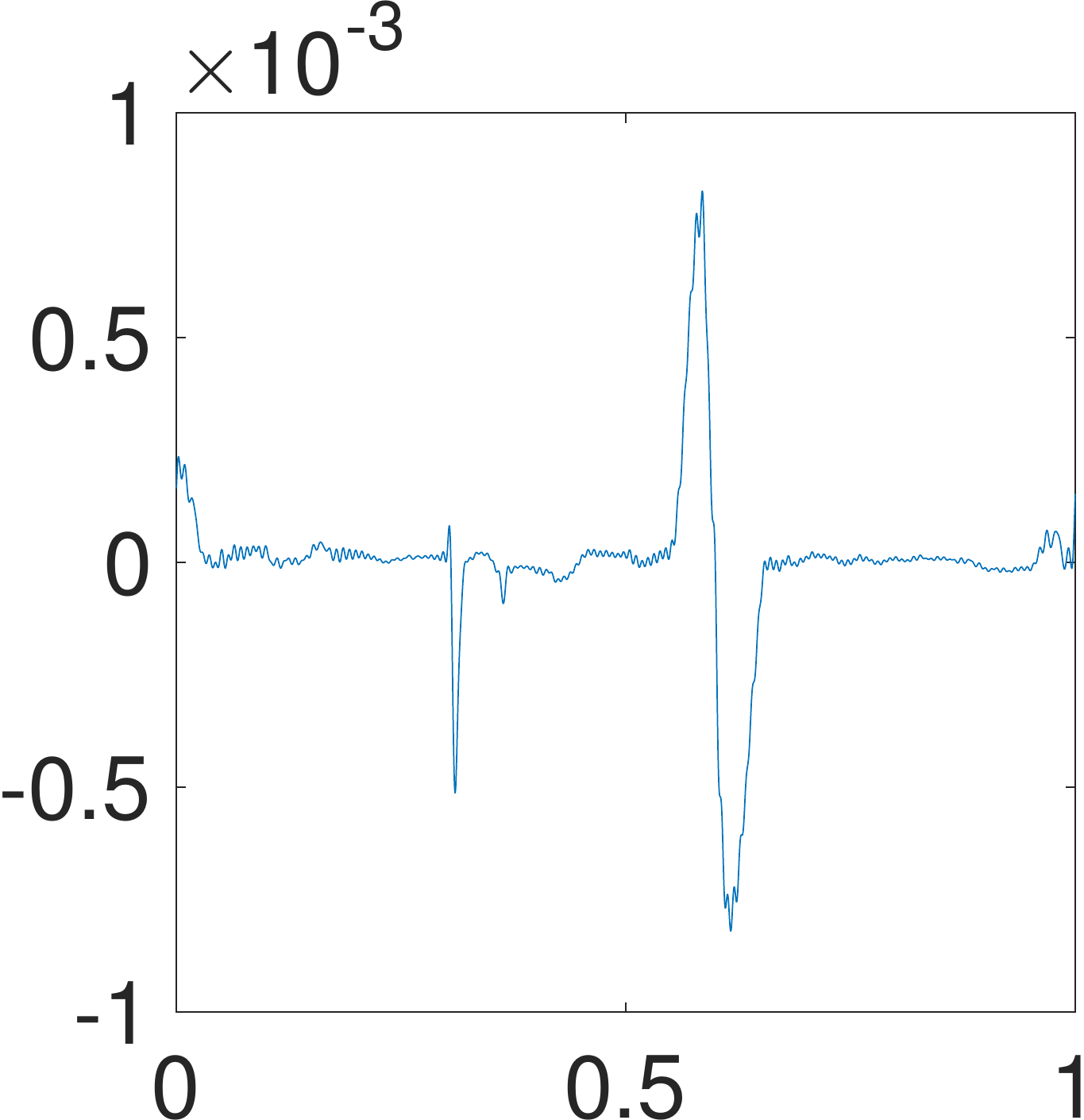}    
    \end{tabular}
  \end{center}
  \caption{Estimated shape functions $a_0 s_{c0}(t)$, $a_1 s_{c1}(t)$, $a_{-1} s_{c-1}(t)$, $b_1 s_{s1}(t)$, and $b_{-1} s_{s-1}(t)$ for the ECG signal in Figure \ref{fig:10_1}.}
\label{fig:10_3}
\end{figure}

\begin{figure}[ht!]
  \begin{center}
    \begin{tabular}{c}
      \includegraphics[width=6in]{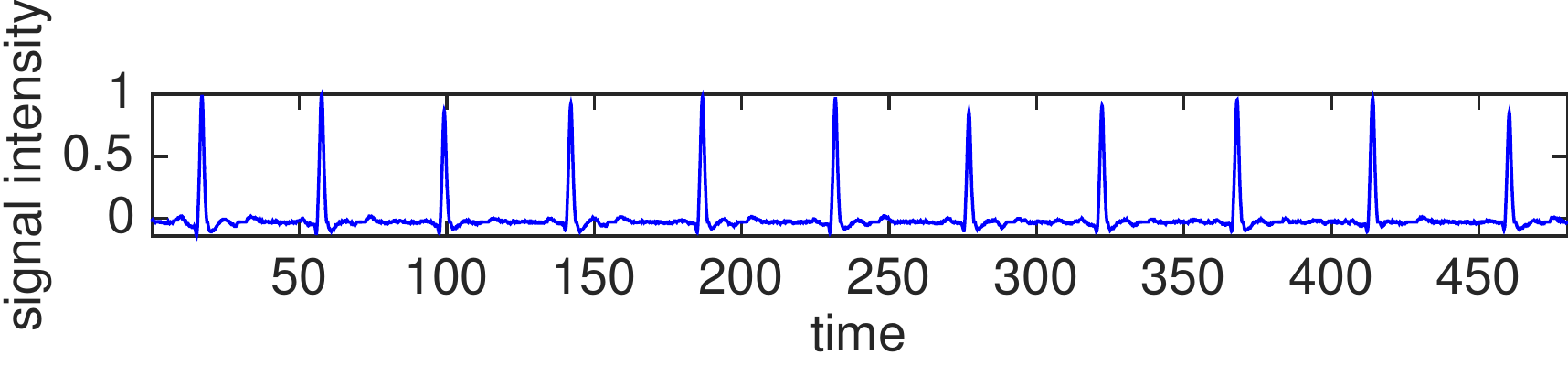}  \\
      A motion-contaminated ECG record $f(t)$\\
     \includegraphics[width=6in]{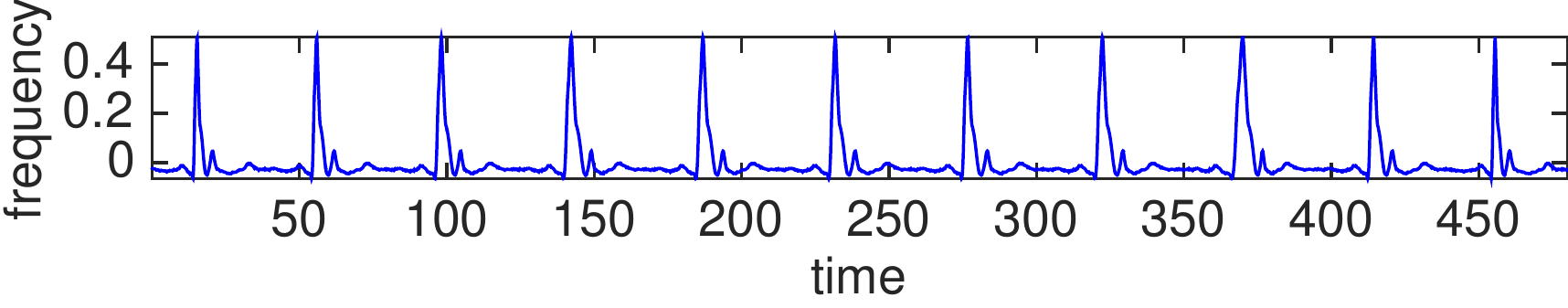}  \\
     $0$-banded multiresolution approximation $\mathcal{M}_0(f)(t)$\\
%      \includegraphics[width=6in]{Pictures/RDSA_fig8_comp_bw_5-eps-converted-to.pdf}  \\
%     $5$-banded multiresolution approximation $\mathcal{M}_0(f)(t)$\\
      \includegraphics[width=6in]{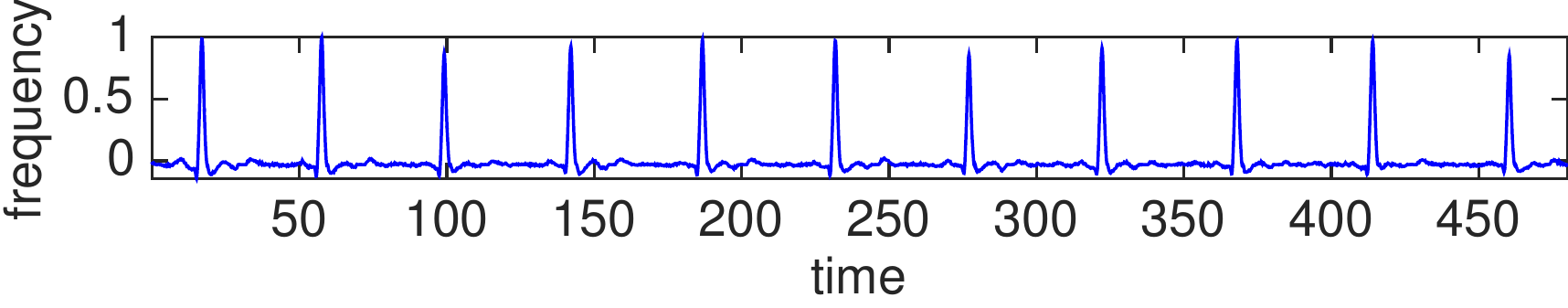}  \\
     $20$-banded multiresolution approximation $\mathcal{M}_{20}(f)(t)$\\
      \includegraphics[width=6in]{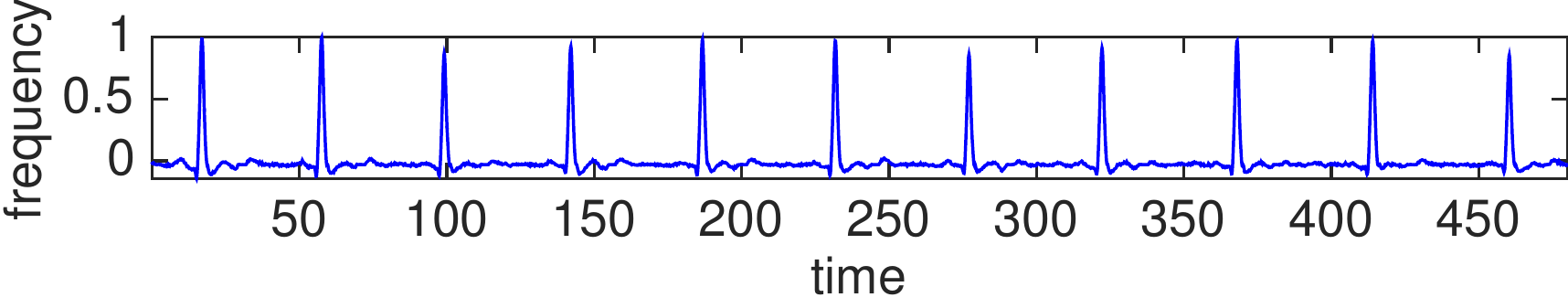}  \\
     $40$-banded multiresolution approximation $\mathcal{M}_{40}(f)(t)$
    \end{tabular}
  \end{center}
  \caption{Multiresolution approximations of a motion-contaminated ECG record.}
\label{fig:11_1}
\end{figure}

\begin{figure}[ht!]
  \begin{center}
    \begin{tabular}{c}
      \includegraphics[width=6in]{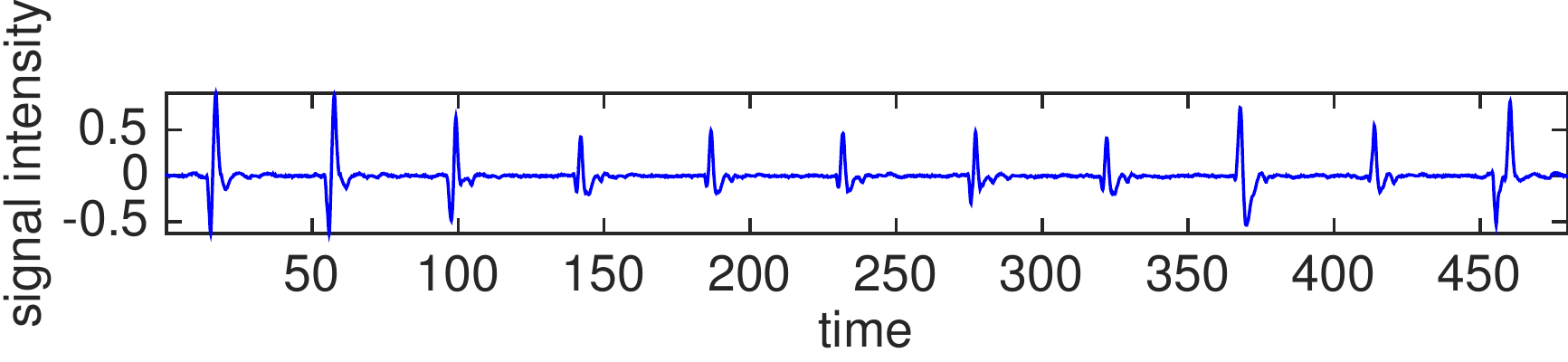}  \\
      $f(t)-\mathcal{M}_{0}(f)(t)$ \\
      \includegraphics[width=6in]{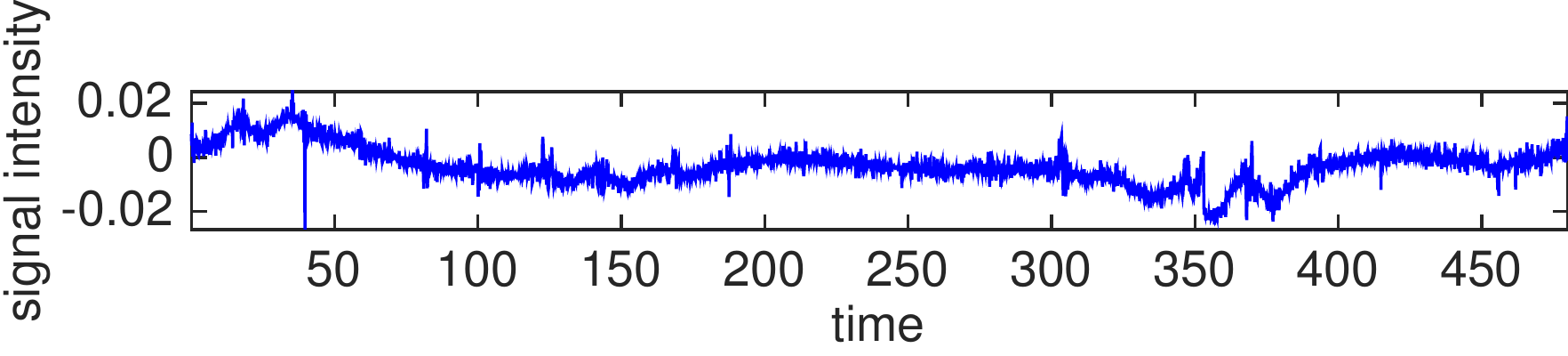}  \\
      $f(t)-\mathcal{M}_{5}(f)(t)$ \\
      \includegraphics[width=6in]{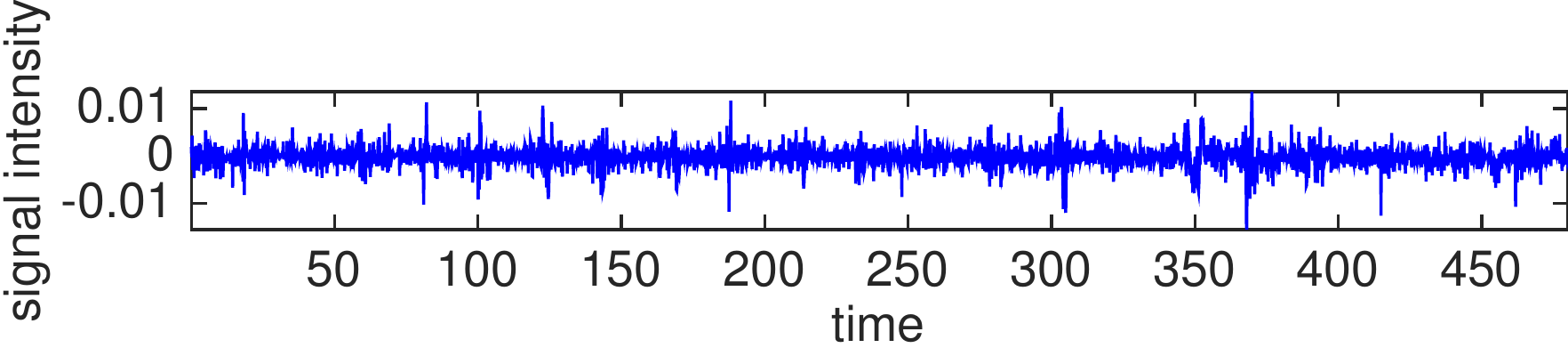}  \\
      $f(t)-\mathcal{M}_{20}(f)(t)$ \\
      \includegraphics[width=6in]{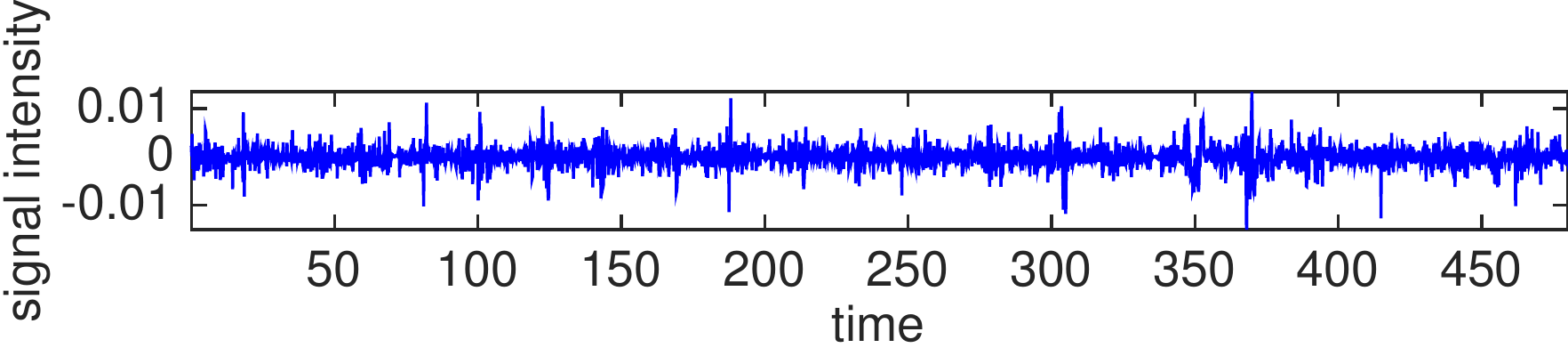}  \\
      $f(t)-\mathcal{M}_{40}(f)(t)$ 
    \end{tabular}
  \end{center}
  \caption{The residual of the multiresolution approximations of a motion-contaminated ECG record in Figure \ref{fig:11_1}.}
\label{fig:11_2}
\end{figure}

\begin{figure}[ht!]
  \begin{center}
    \begin{tabular}{ccccc}
      \includegraphics[width=1.1in]{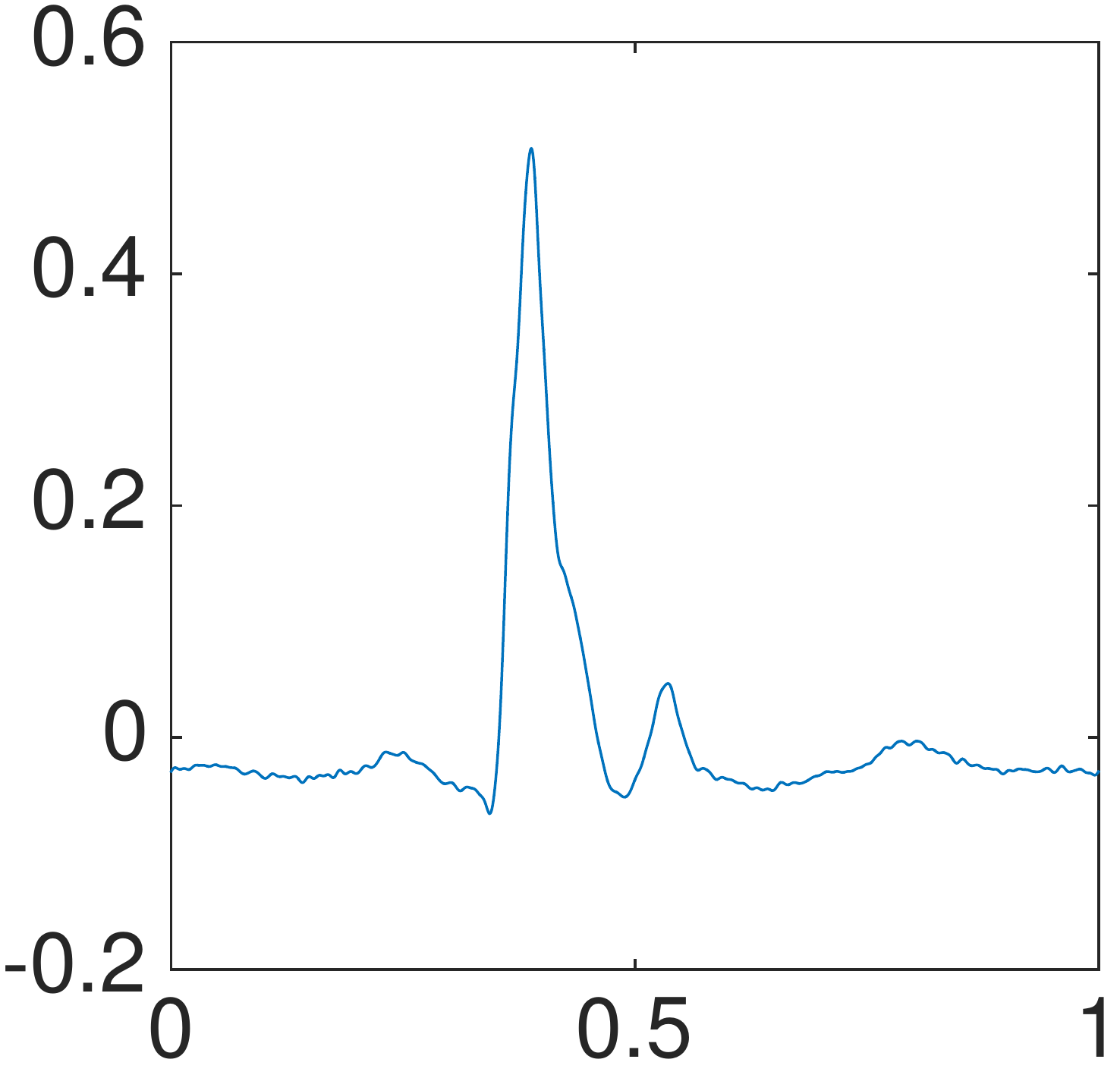}   &
      \includegraphics[width=1.1in]{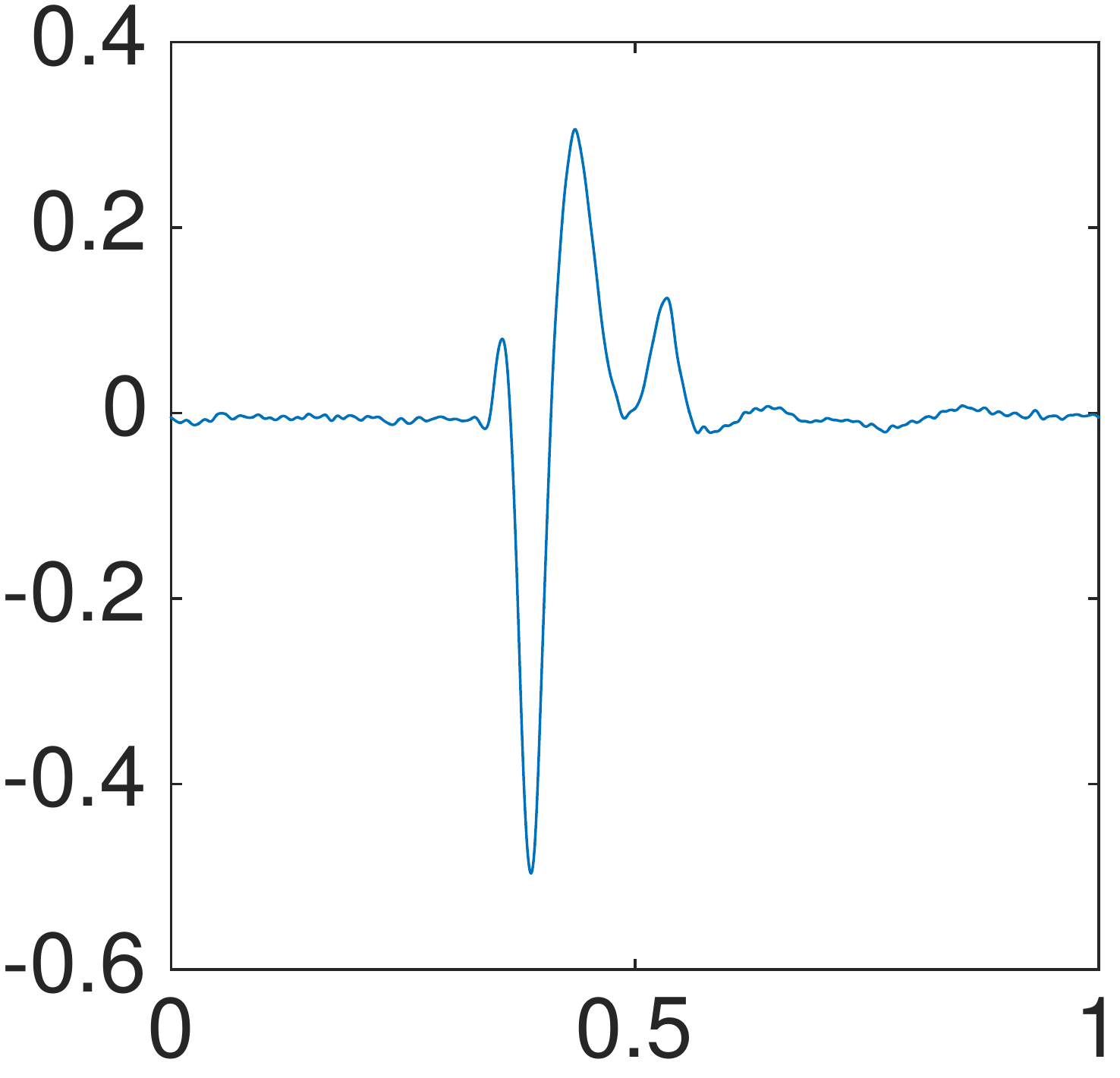}   &
      \includegraphics[width=1in]{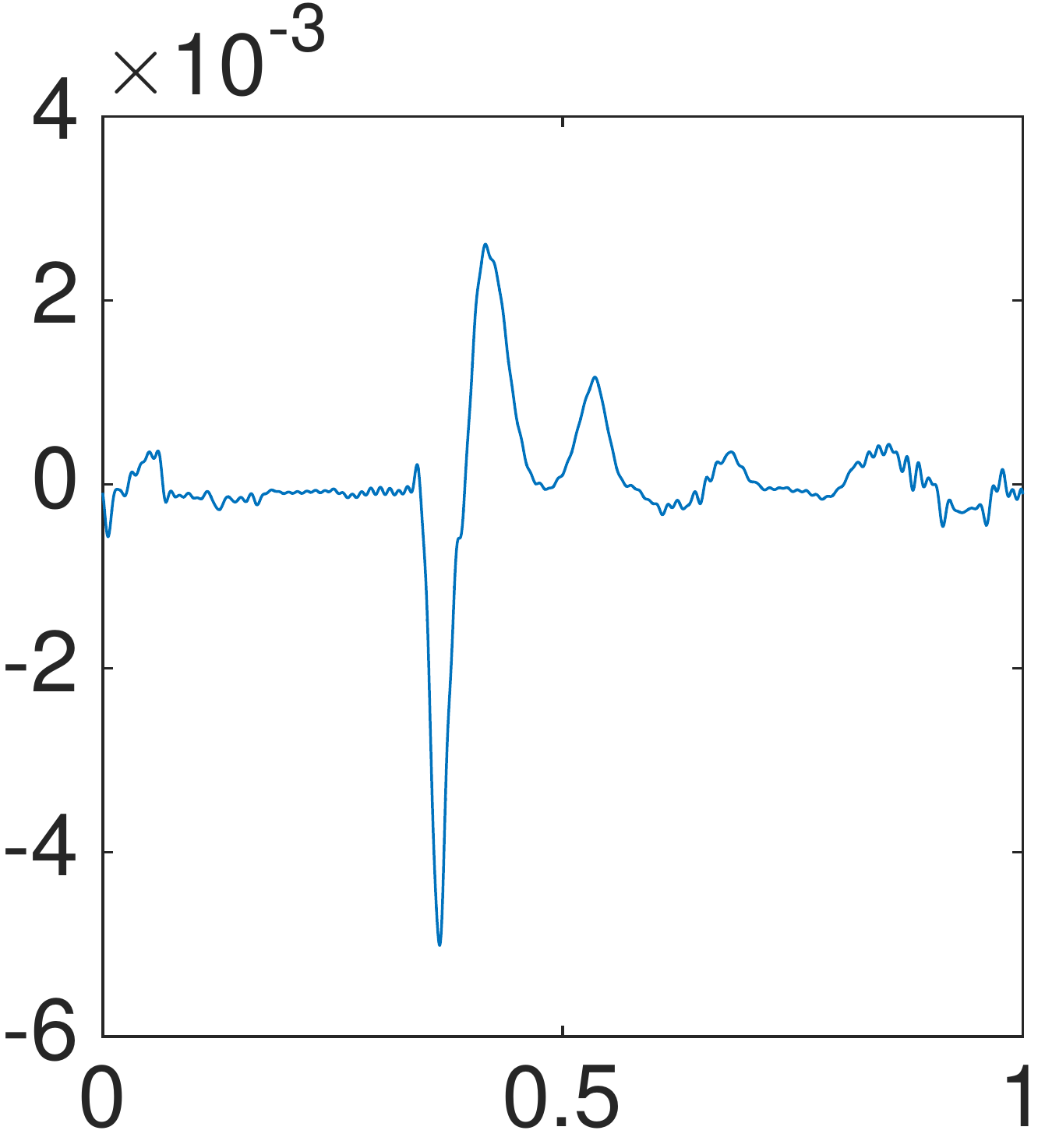}   &
      \includegraphics[width=1.1in]{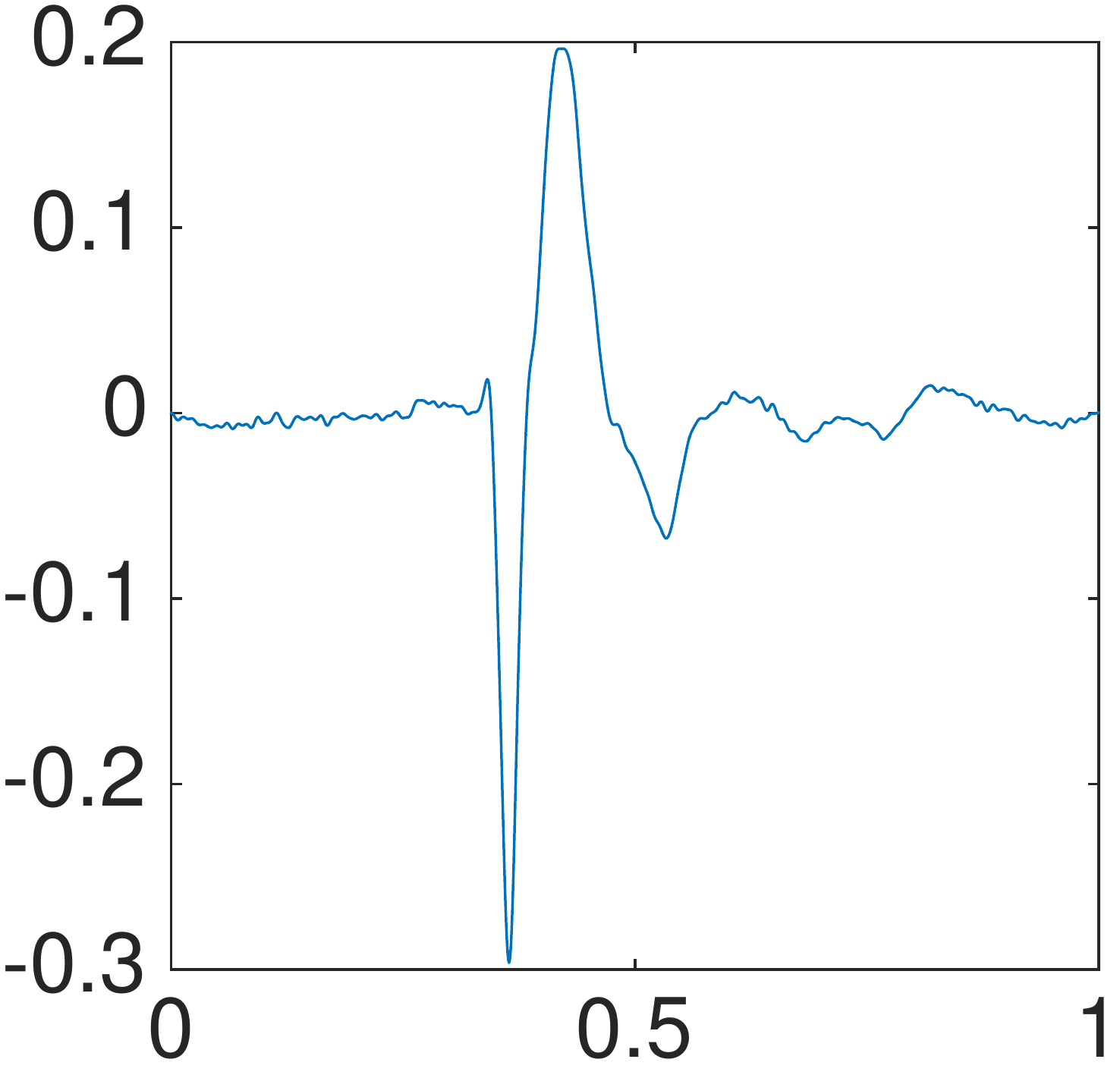}   &
      \includegraphics[width=1in]{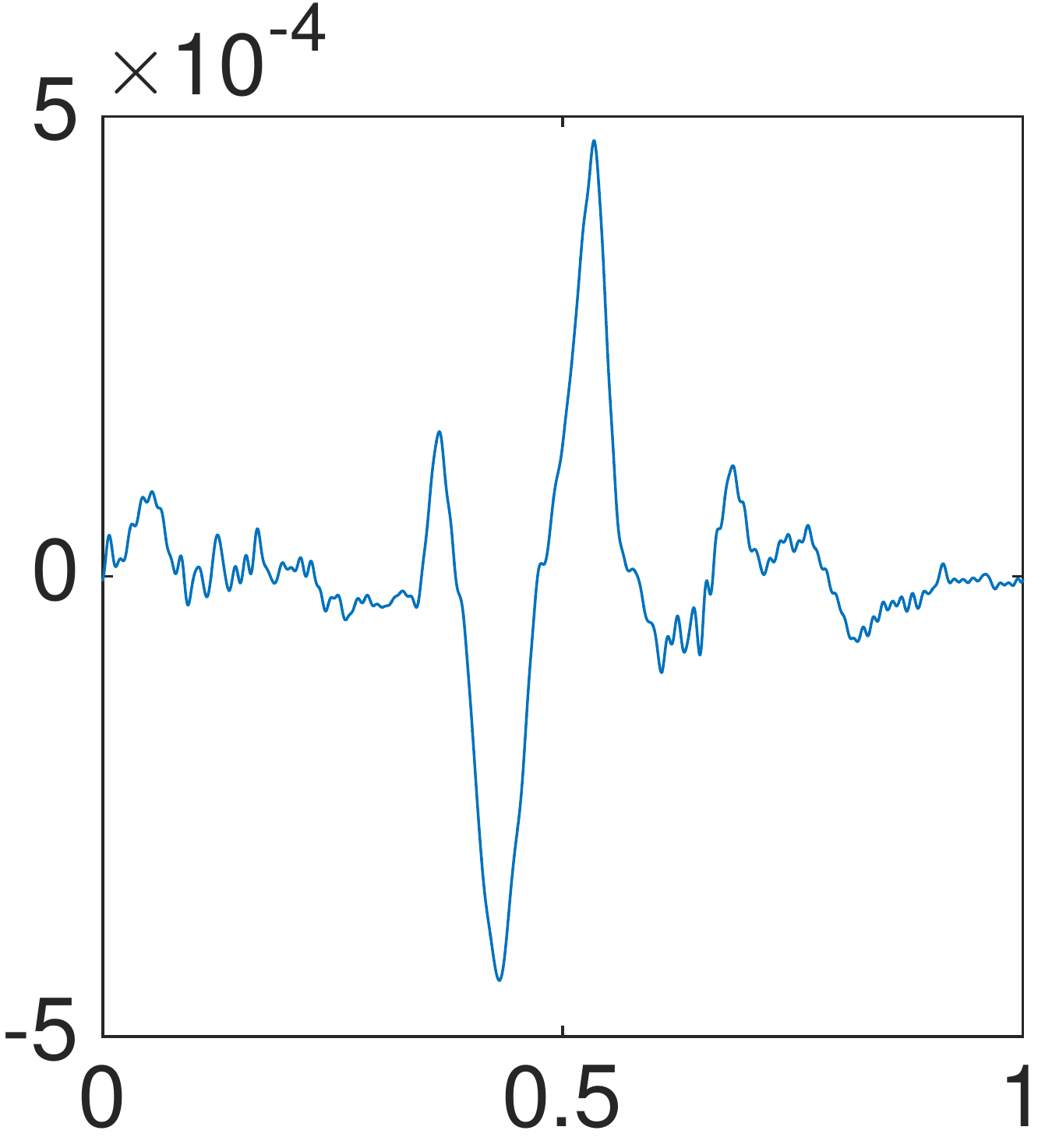}    
    \end{tabular}
  \end{center}
  \caption{Estimated shape functions $a_0 s_{c0}(t)$, $a_1 s_{c1}(t)$, $a_{-1} s_{c-1}(t)$, $b_1 s_{s1}(t)$, and $b_{-1} s_{s-1}(t)$ for the ECG signal in Figure \ref{fig:11_1}.}
\label{fig:11_3}
\end{figure}

\subsection{MMD for synthetic data}

In this section, a synthetic example of MMD is provided to demonstrate the effectiveness of RDSA. We consider a simple case when the signal has two MIMF's with ECG shape functions. In particular, we let the shape function series of each MIMF contain the same ECG shape function. This makes it easier to verify Algorithm \ref{alg:MMD2}. For example, we consider a signal of the form
\begin{equation}
\label{eqn:ex3}
f(t) = f_1(t)+f_2(t),
\end{equation}
where 
\begin{equation}\label{eq:f1_4}
f_1(t) = \alpha_1(\phi_1(t))s_1(300\pi \phi_1(t)),
\end{equation}
\begin{equation}\label{eq:f2_4}
f_2(t) = \alpha_2(\phi_2(t))s_2(440\pi\phi_2(t)),
\end{equation}
\[
\alpha_1(t) =  1+0.2\cos(2\pi t)+0.1\sin(2\pi t), 
\]
\[
\alpha_2(t) =  1+0.1\cos(2\pi t)+0.2\sin(2\pi t),
\]
\[
\phi_1(t) = x+0.006\sin(2\pi t),
\]
and 
\[
\phi_2(t) = x+0.006\cos(2\pi t).
\]
$s_1(2\pi t)$ and $s_2(2\pi t)$ are generalized shape functions defined on $[0,1]$ as shown in Figure \ref{fig:4_0}. We apply Algorithm \ref{alg:MMD2} with the known instantaneous phases just above to estimate the multiresolution expansion coefficients and the shape functions series. The product of the multiresolution expansion coefficient and its corresponding shape function is shown in Figure \ref{fig:4}. The estimation errors are very small; the estimated results and the ground truth are almost indistinguishable.

\begin{figure}[ht!]
  \begin{center}
    \begin{tabular}{cc}
      \includegraphics[height=1in]{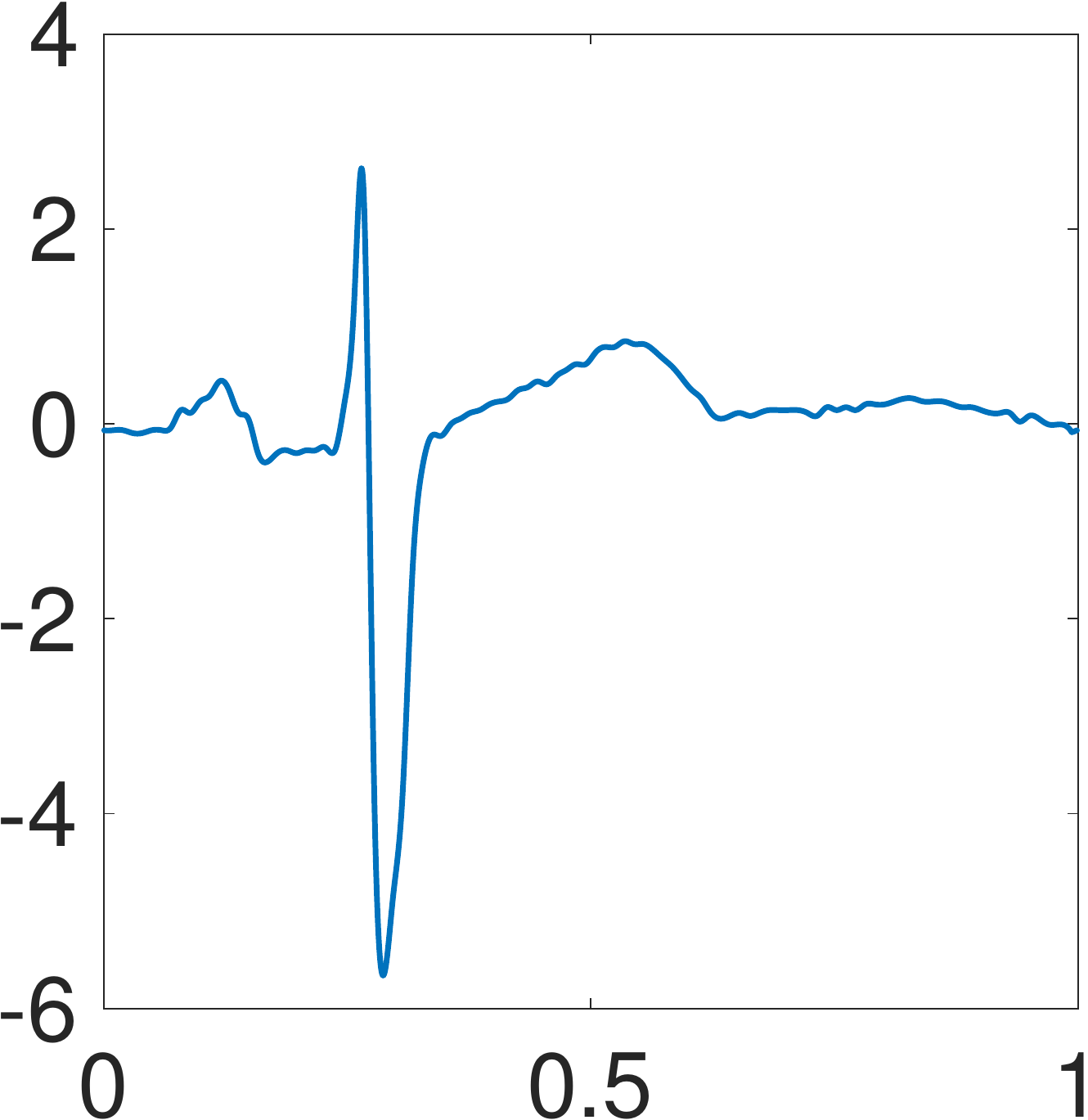}  &
   \includegraphics[height=1in]{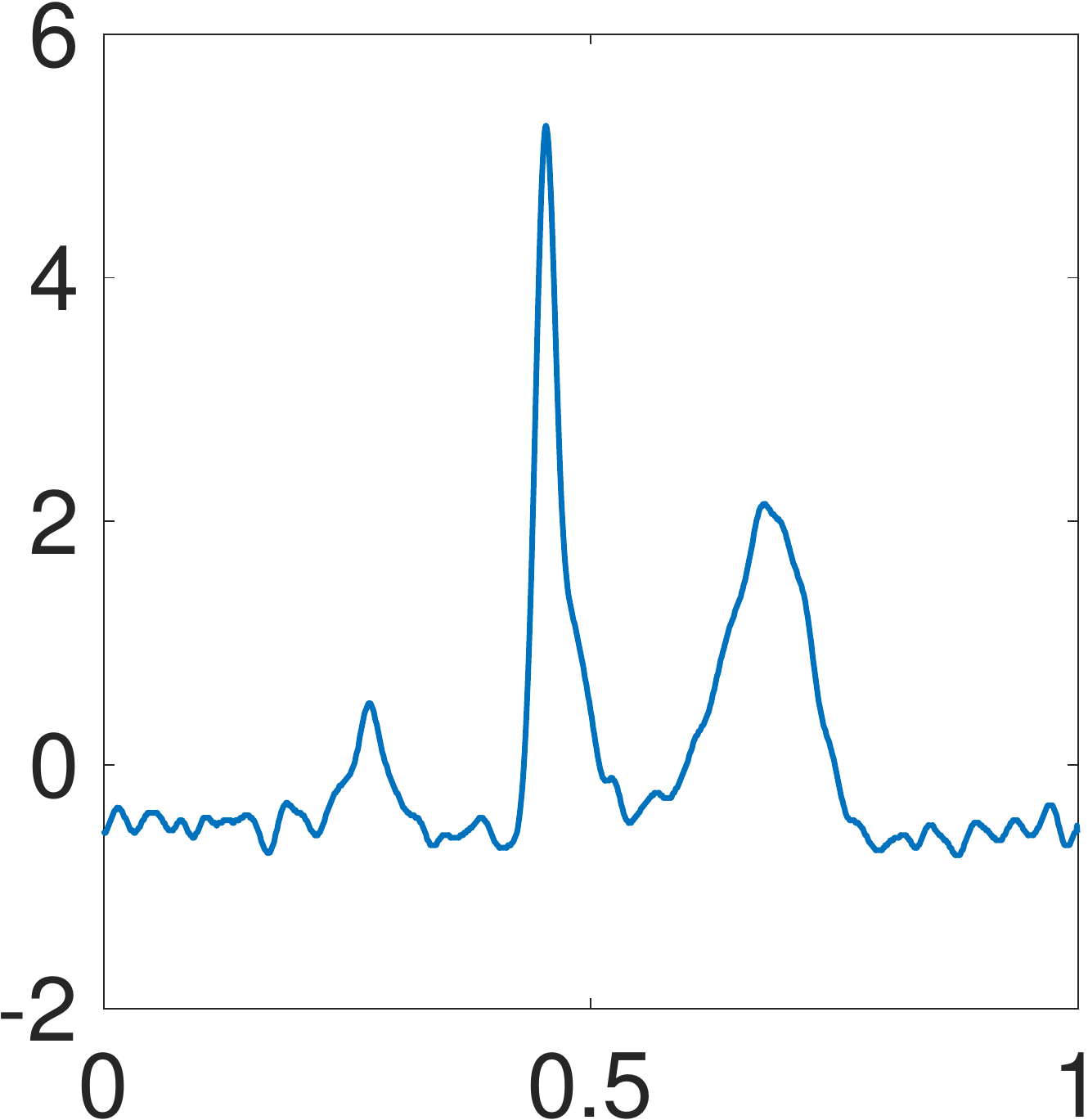}  
    \end{tabular}
  \end{center}
  \caption{Shape function $s_1(2\pi t)$ in \eqref{eq:f1_4} and $s_2(2\pi t)$ in \eqref{eq:f2_4}.}
\label{fig:4_0}
\end{figure}

\begin{figure}[ht!]
  \begin{center}
    \begin{tabular}{ccccc}
      \includegraphics[width=1in]{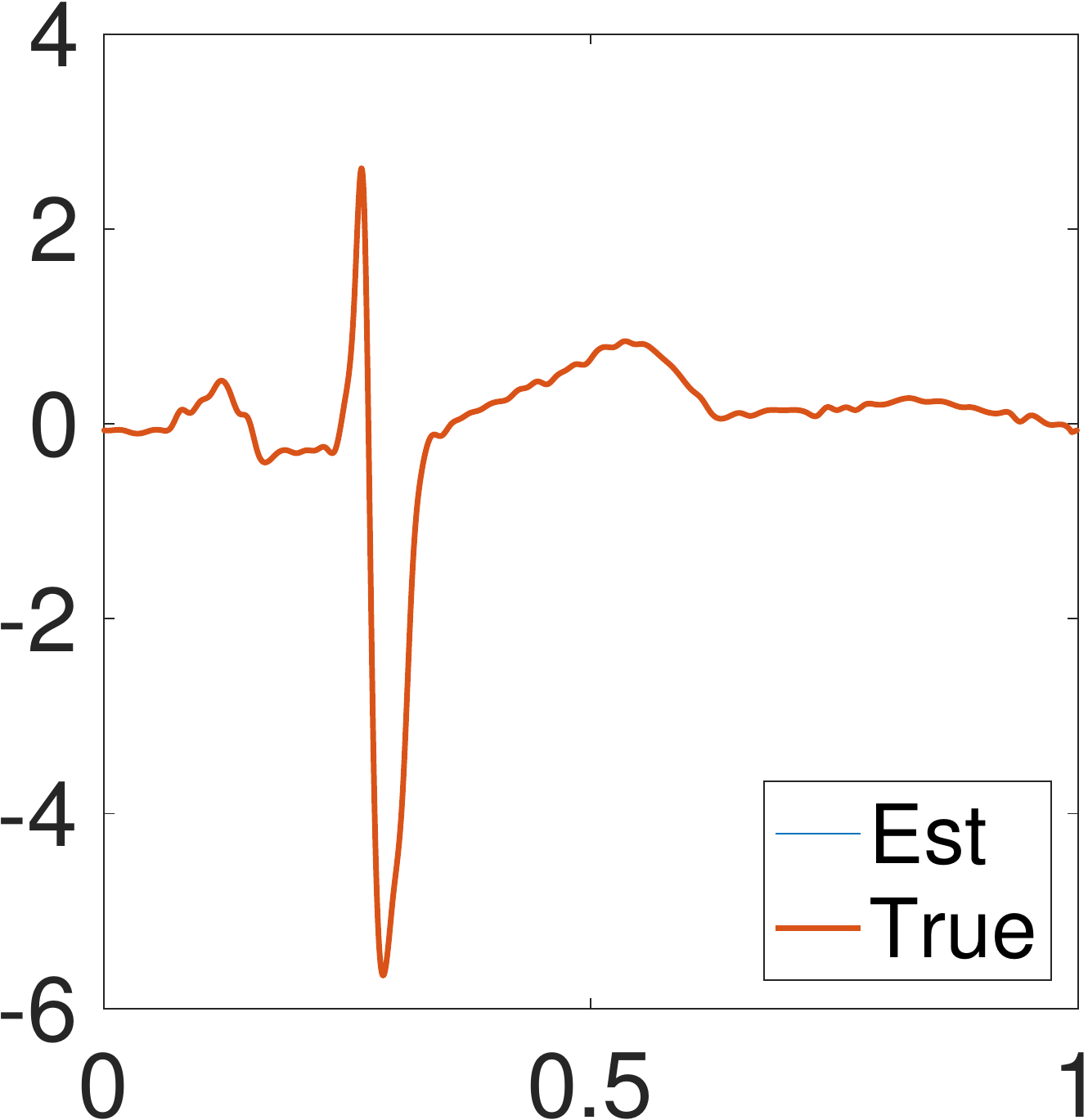}   &
      \includegraphics[width=1.075in]{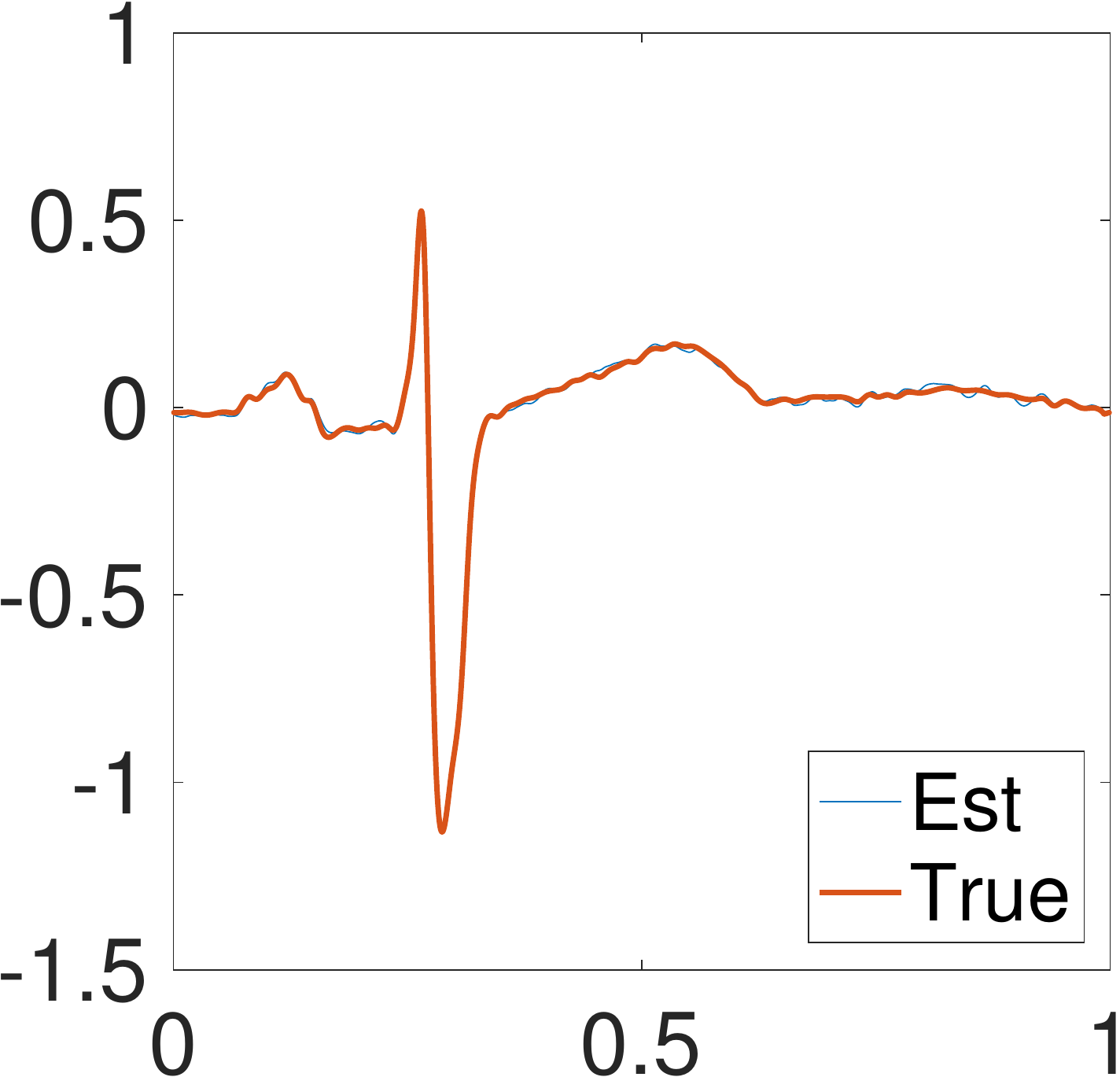}   &
      \includegraphics[width=1.125in]{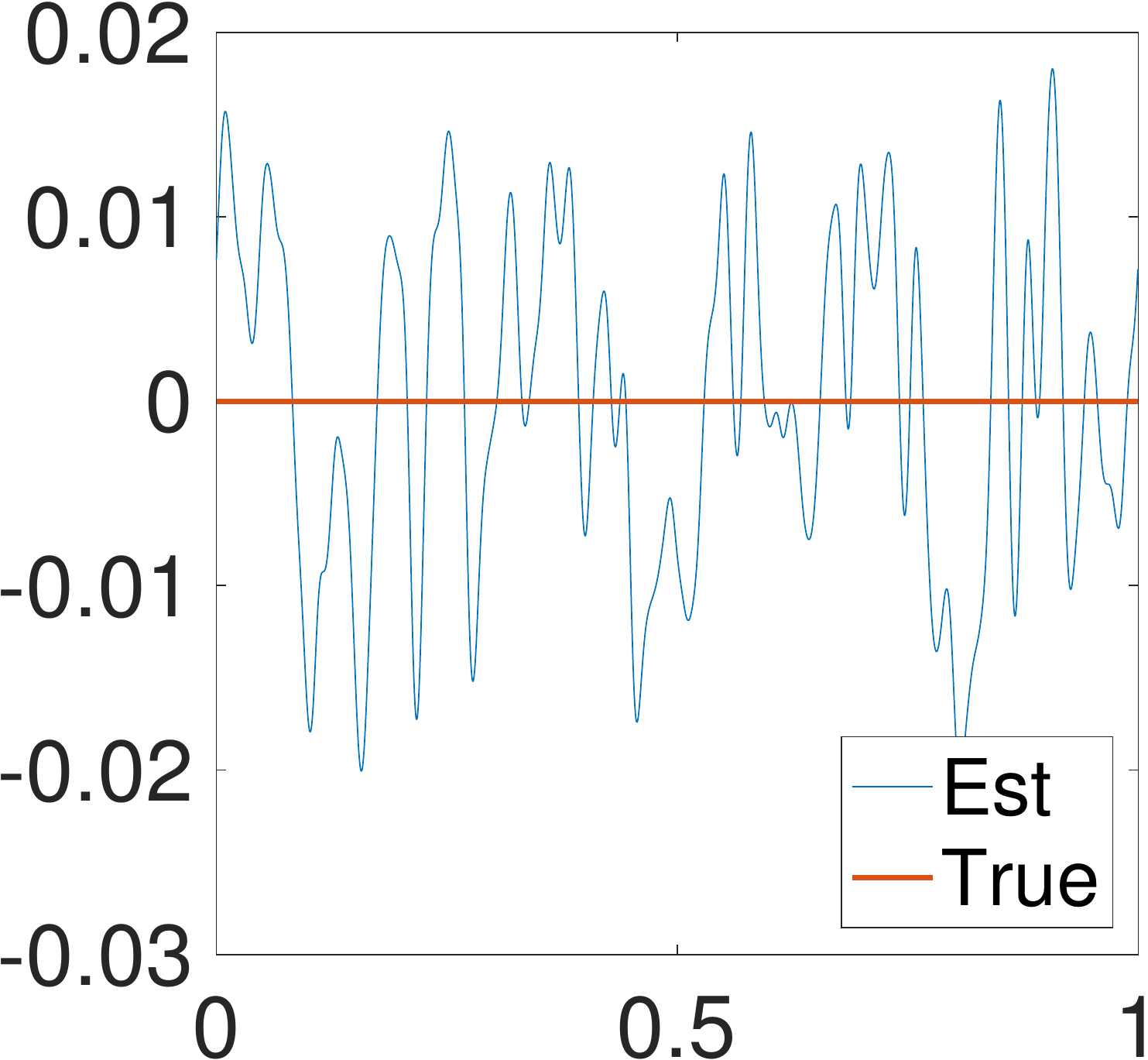}   &
      \includegraphics[width=1.075in]{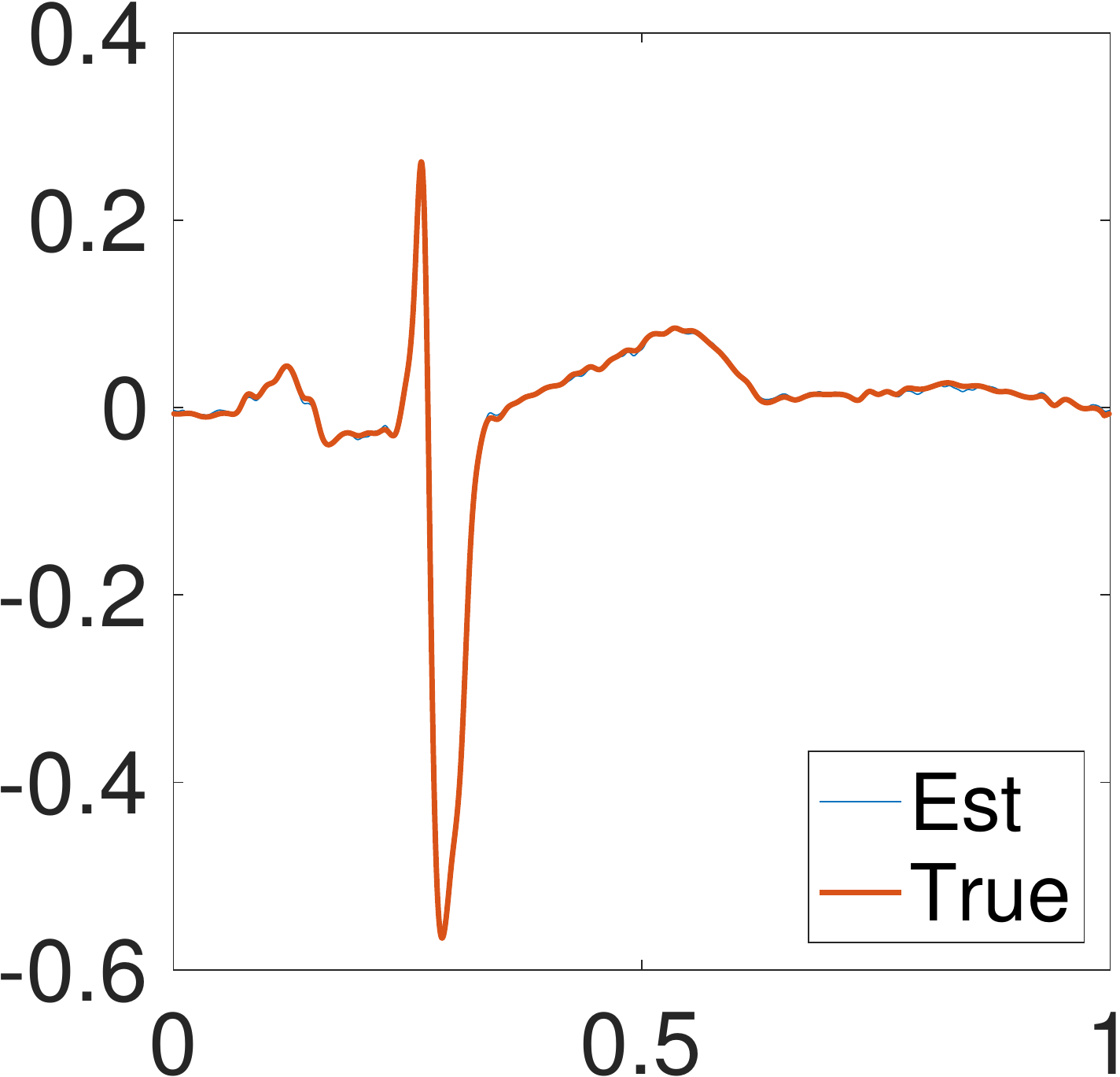}   &
      \includegraphics[width=1in]{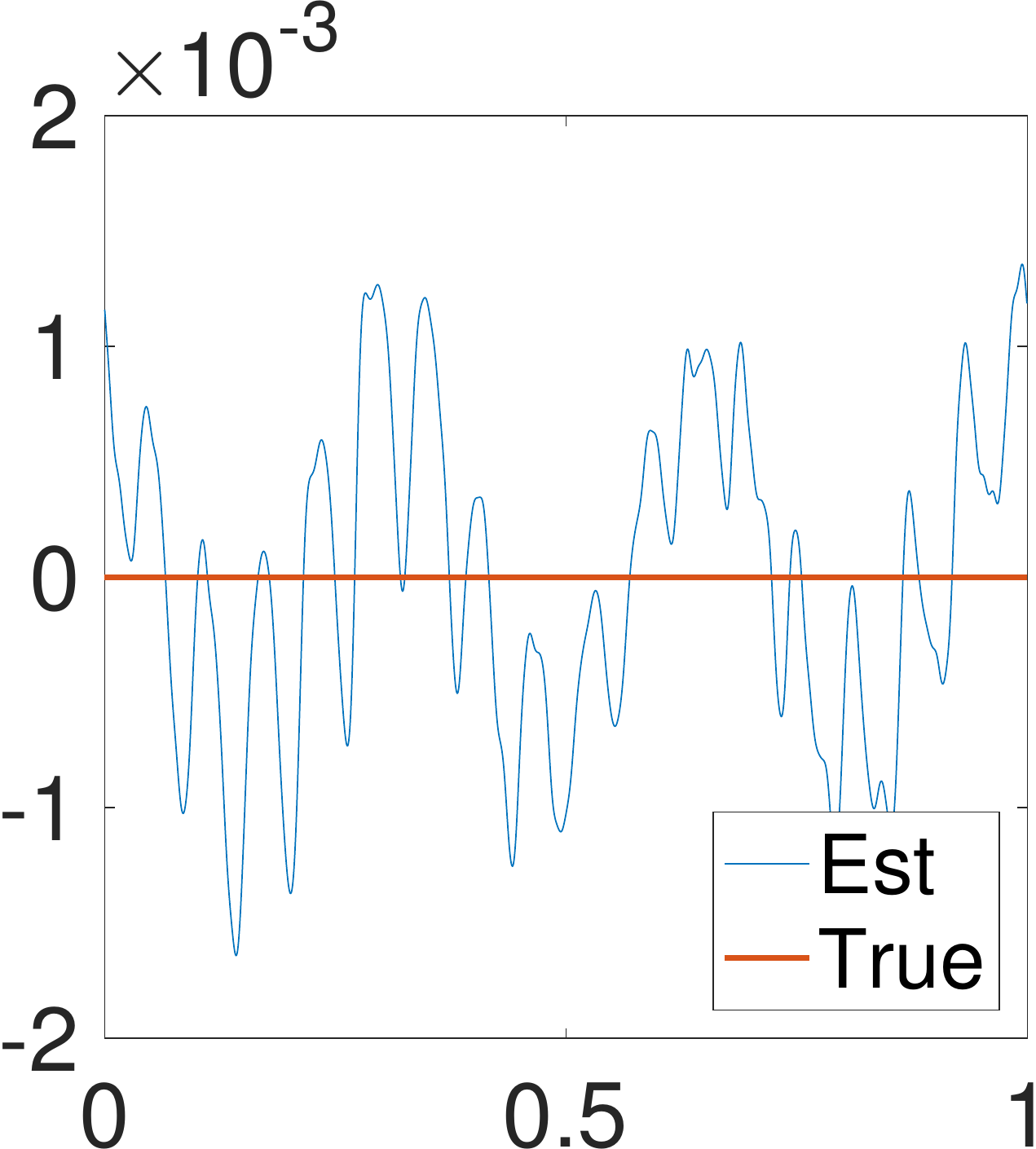}   \\
      \includegraphics[width=1in]{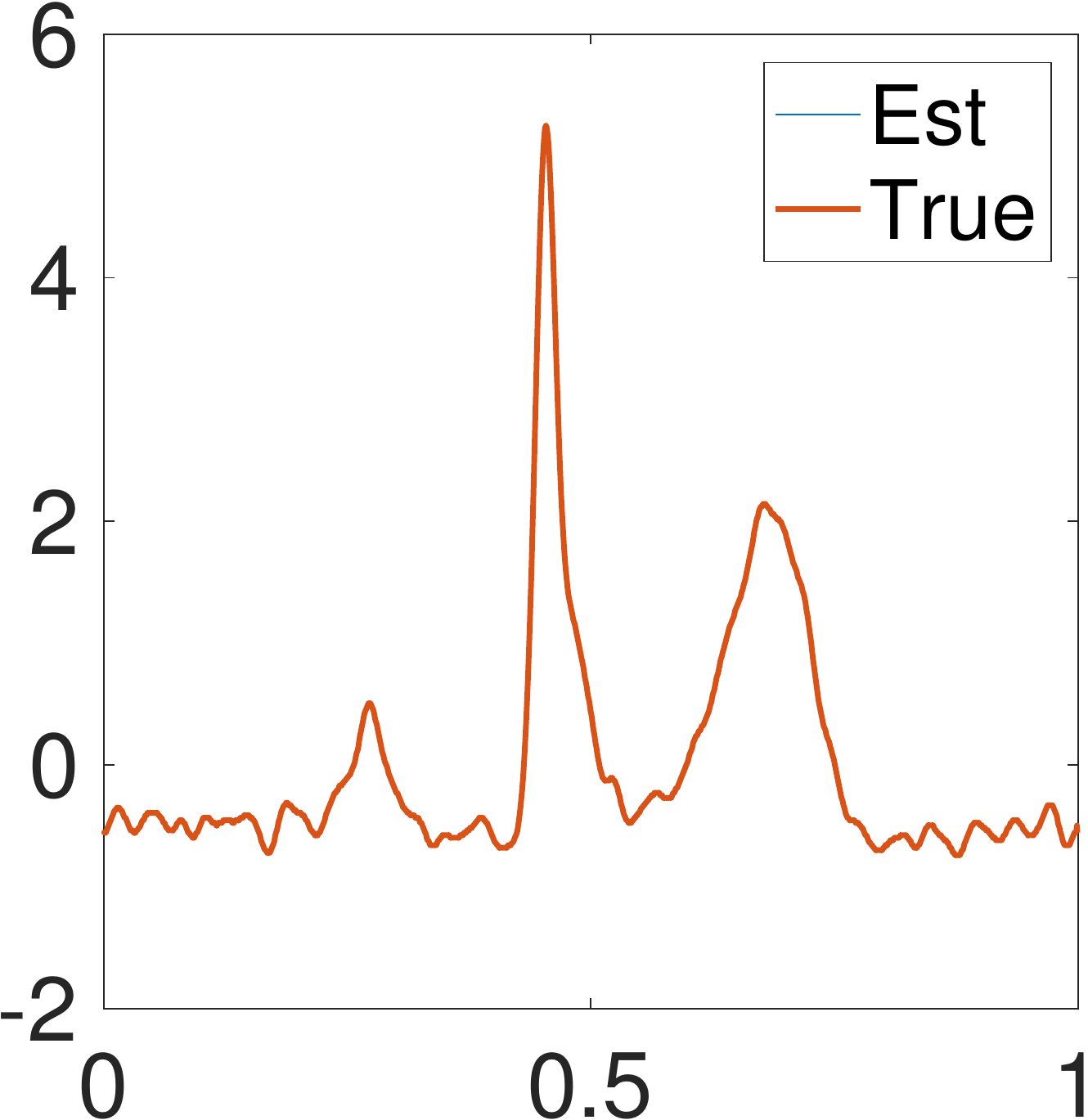}   &
      \includegraphics[width=1.1in]{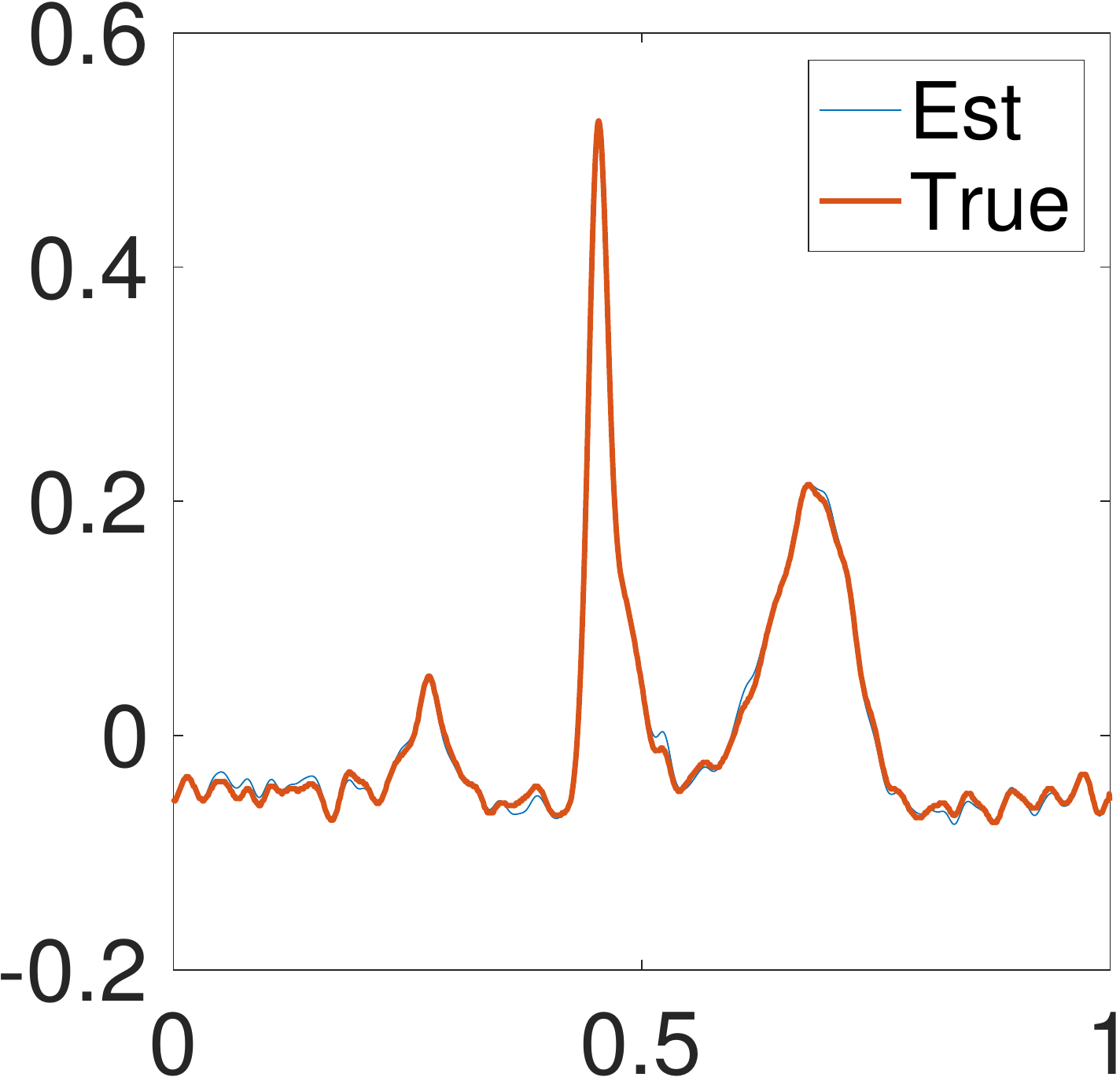}   &
      \includegraphics[width=1.15in]{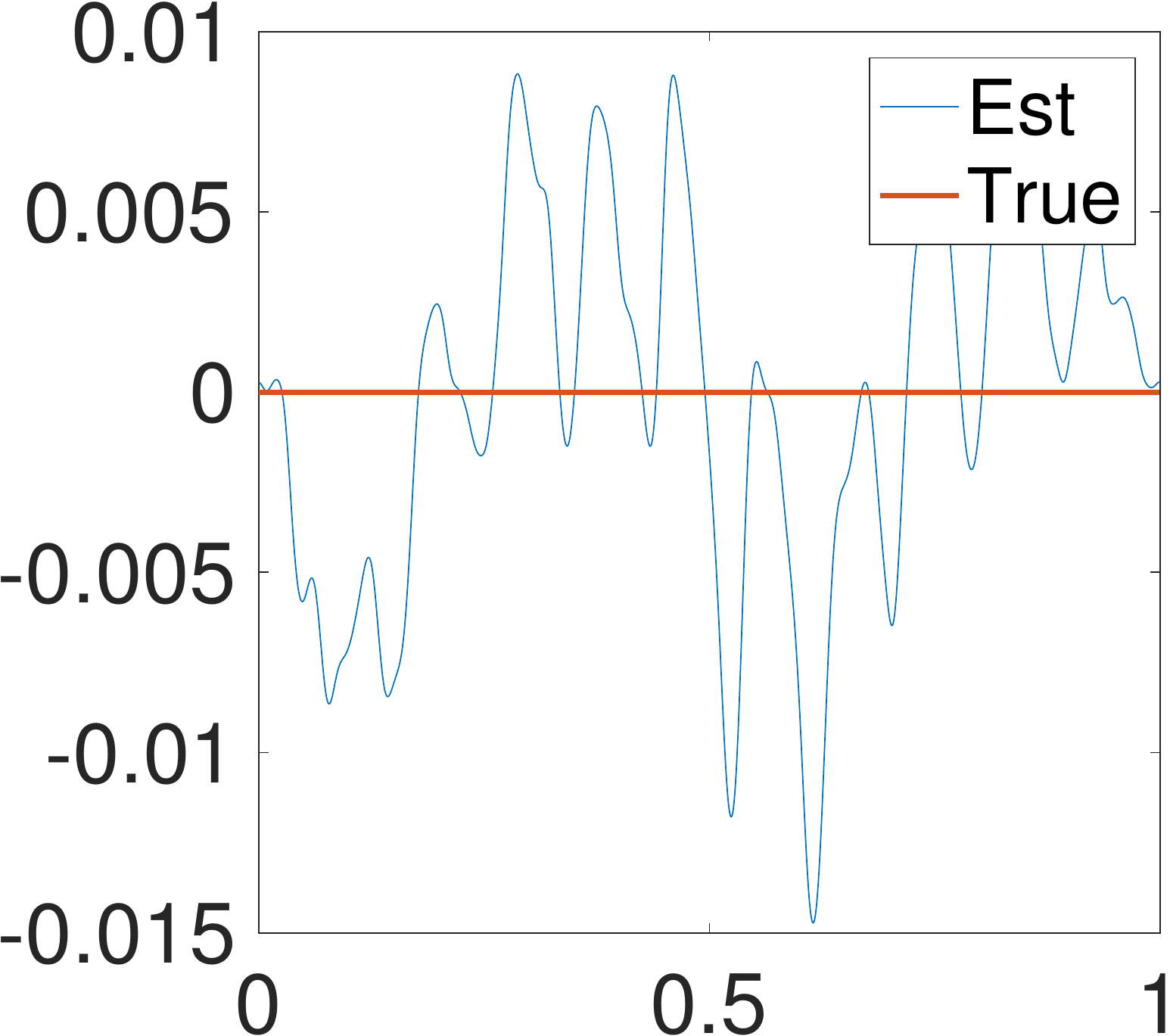}   &
      \includegraphics[width=1.075in]{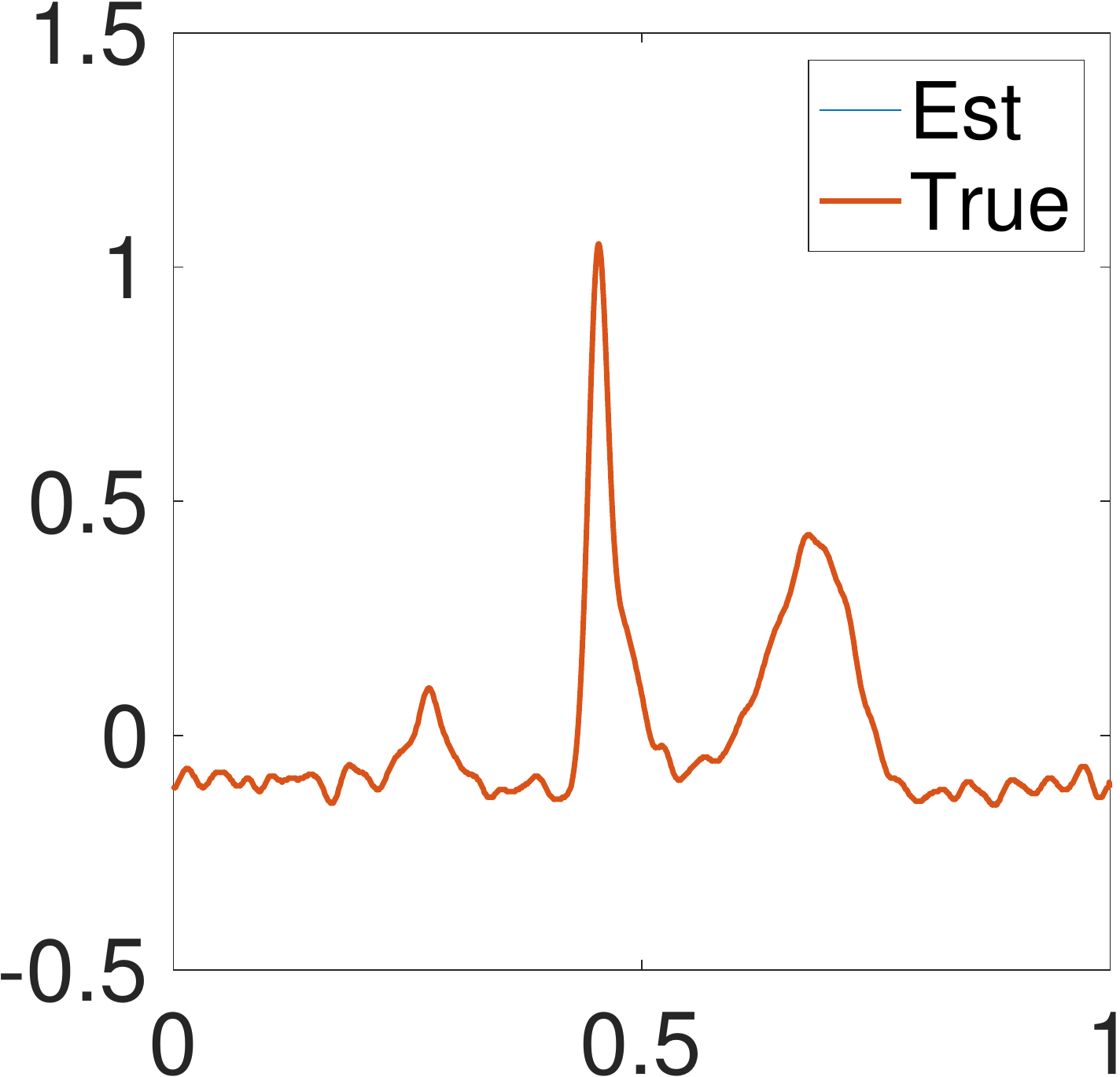}   &
      \includegraphics[width=1in]{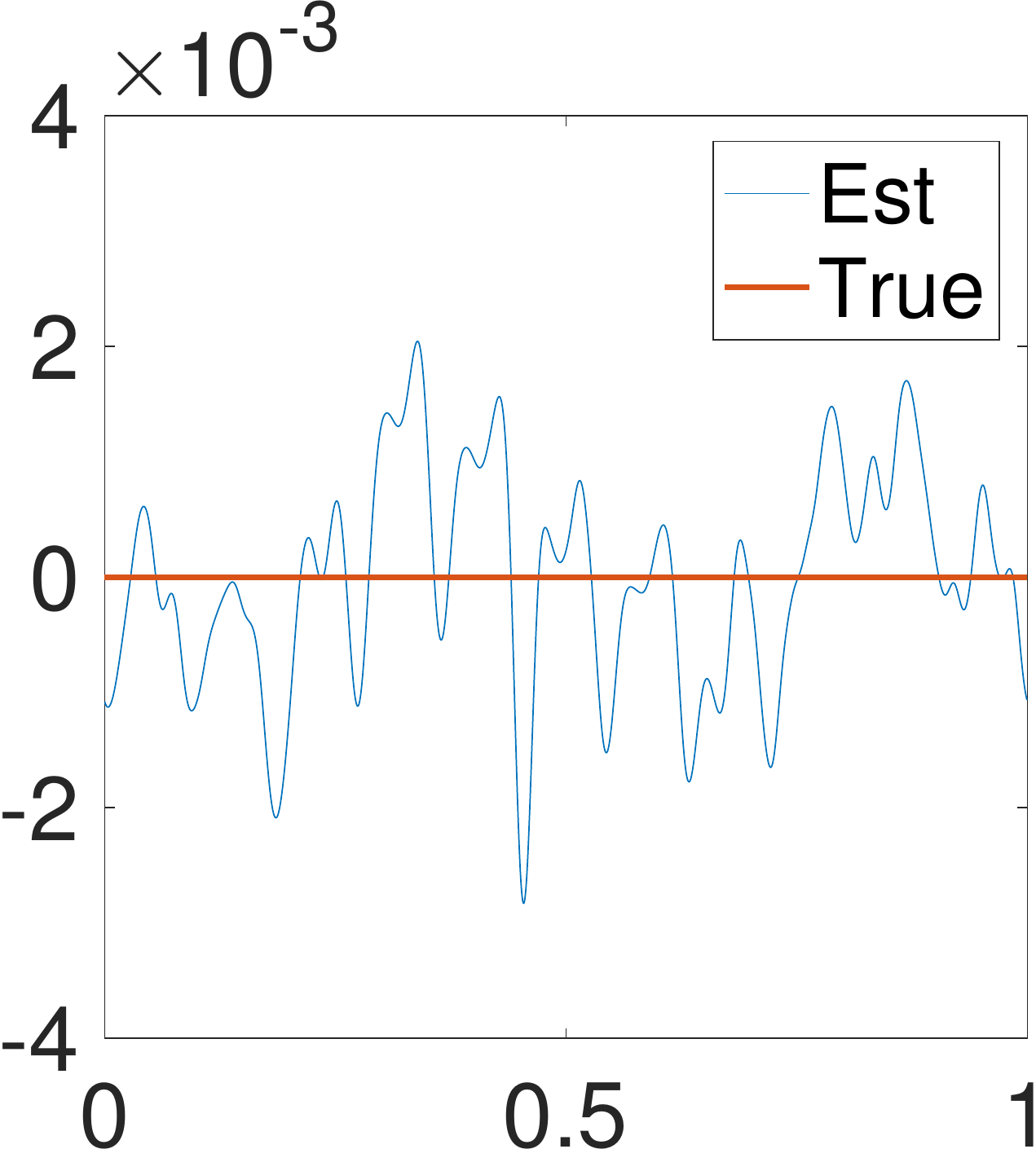}   
    \end{tabular}
  \end{center}
  \caption{Top: estimated shape functions $a_{0,1}s_{c0,1}(2\pi t)$, $a_{1,1}s_{c1,1}(2\pi t)$, $a_{-1,1}s_{c-1,1}(2\pi t)$, $b_{1,1}s_{s1,1}(2\pi t)$, and $b_{-1,1}s_{s-1,1}(2\pi t)$ of $f_1(t)$ in \eqref{eq:f1_4}. Bottom: estimated shape functions $a_{0,2}s_{c0,2}(2\pi t)$, $a_{1,2}s_{c1,2}(2\pi t)$, $a_{-1,2}s_{c-1,2}(2\pi t)$, $b_{1,2}s_{s1,2}(2\pi t)$, and $b_{-1,2}s_{s-1,2}(2\pi t)$ of $f_2(t)$ in \eqref{eq:f2_4}.}
\label{fig:4}
\end{figure}

\subsection{MMD for real data}

This is an example of photoplethysmography (PPG)\footnote{From \url{http://www.capnobase.org}.} that contains a hemodynamical MIMF and a respiration MIMF. The instantaneous frequencies and phases are not known and they are estimated via the synchrosqueezed transform in \cite{1DSSWPT}. Figure \ref{fig:15_1} shows the estimated instantaneous frequencies of the respiratory and cardiac cycles. Inputing their corresponding instantaneous phases into RDSA, the PPG signal is separated into a respiratory MIMF and a cardiac MIMF as shown in Figure \ref{fig:15_2}; their leading multiresolution shape functions are shown in Figure \ref{fig:15_3}. 

The last two panels of Figure \ref{fig:15_2} shows that the PPG signal has been completely separated into two MIMF's; the residual signal only contains noise, a smooth trend, and some sharp changes that are not correlated to the oscillation in MIMF's. The second panel shows that the MIMF model can characterize time-varying shape functions, while the third panel shows that the MIMF model can capture the time-varying amplitude functions. 

\begin{figure}[ht!]
  \begin{center}
    \begin{tabular}{c}
  \hspace{-0.5cm}  \includegraphics[width=6.1in]{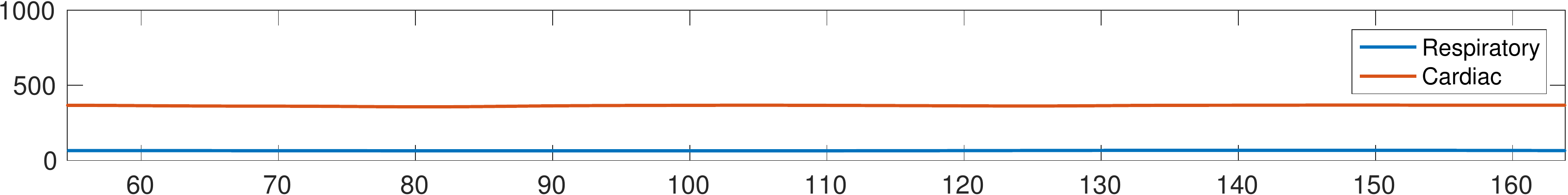} 
    \end{tabular}
  \end{center}
  \caption{Estimated fundamental instantaneous frequencies of the real PPG signal in the first panel of Figure \ref{fig:15_2} by the synchrosqueezed transform.}
\label{fig:15_1}
\end{figure}

\begin{figure}[ht!]
  \begin{center}
    \begin{tabular}{c}
      \includegraphics[width=6in]{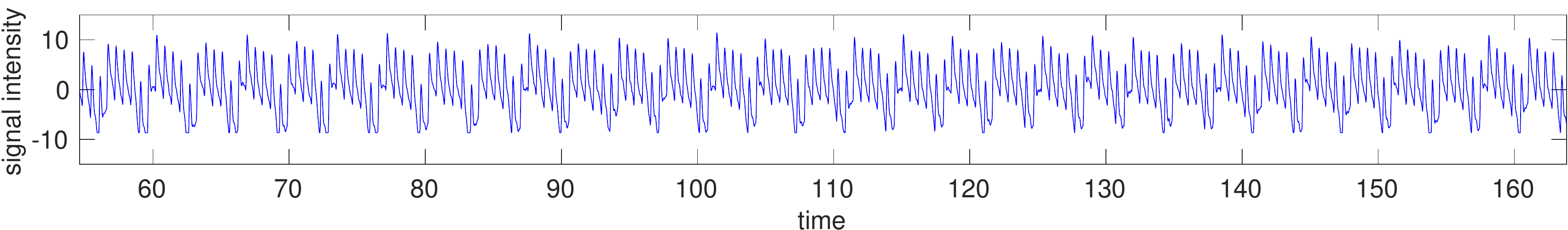}  \\
   \includegraphics[width=6in]{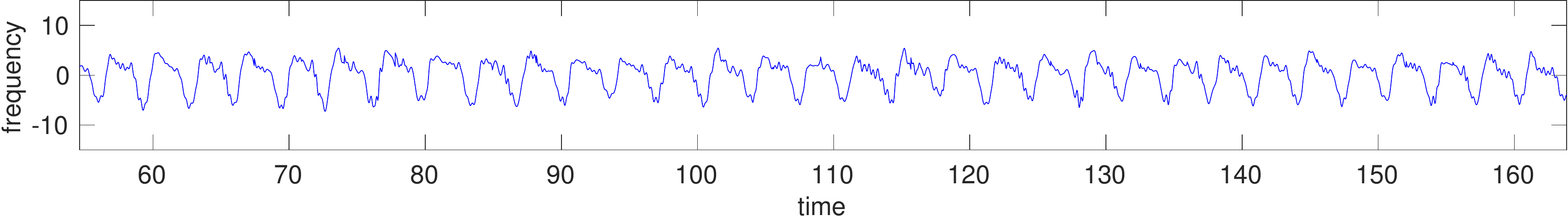}   \\
   \includegraphics[width=6in]{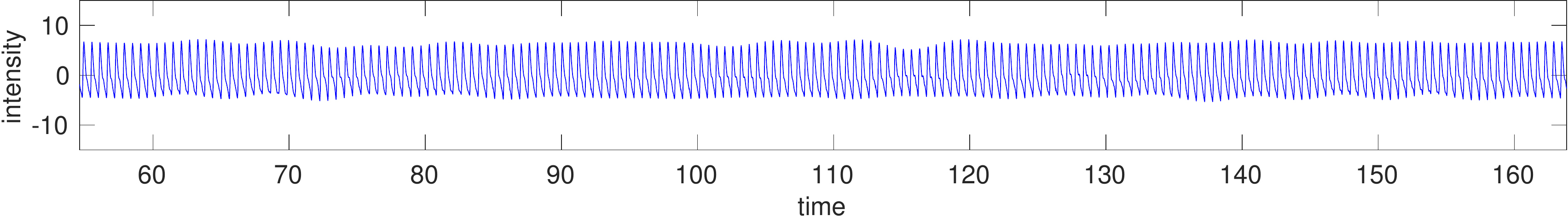}  \\
      \includegraphics[width=6in]{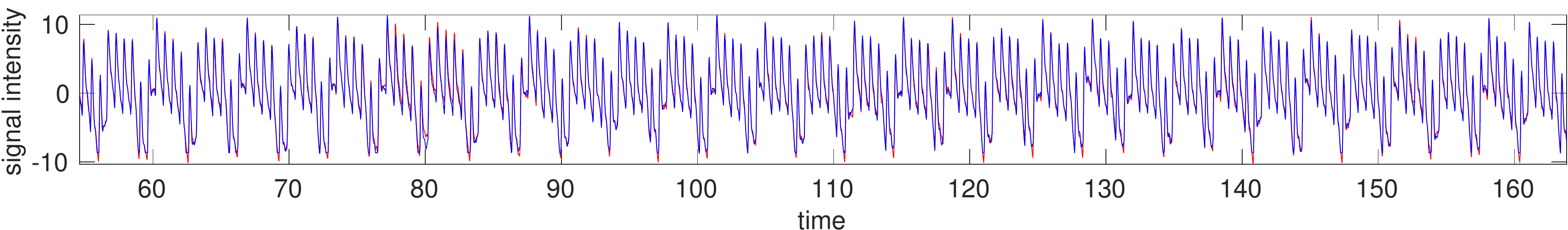}  \\
      \includegraphics[width=6in]{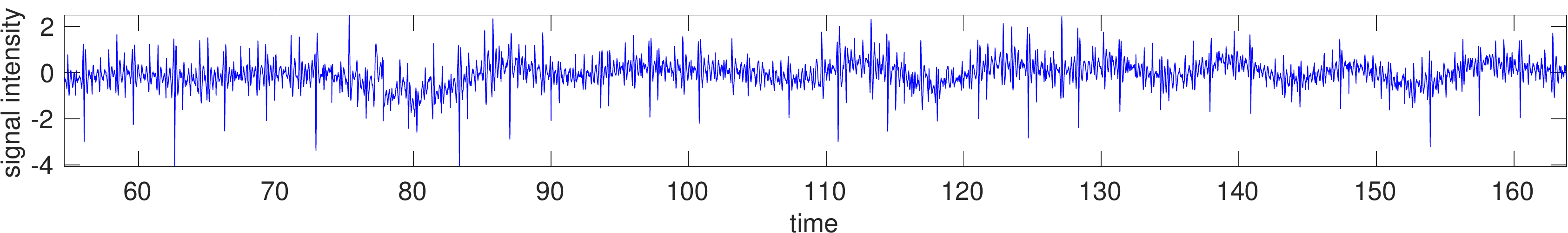}  
    \end{tabular}
  \end{center}
  \caption{First panel: the raw PPG signal $f(t)$. Second panel: the respiratory MIMF $f_1(t)$. Third panel: the cardiac MIMF $f_2(t)$. Fourth panel: the summation of the respiratory and cardiac MIMF's $f_1(t)+f_2(t)$ (red) compared to the raw PPG signal $f(t)$ (blue). The fifth panel: the residual signal $f(t)-f_1(t)-f_2(t)$.}
\label{fig:15_2}
\end{figure}

\begin{figure}[ht!]
  \begin{center}
    \begin{tabular}{ccccc}
      \includegraphics[width=1in]{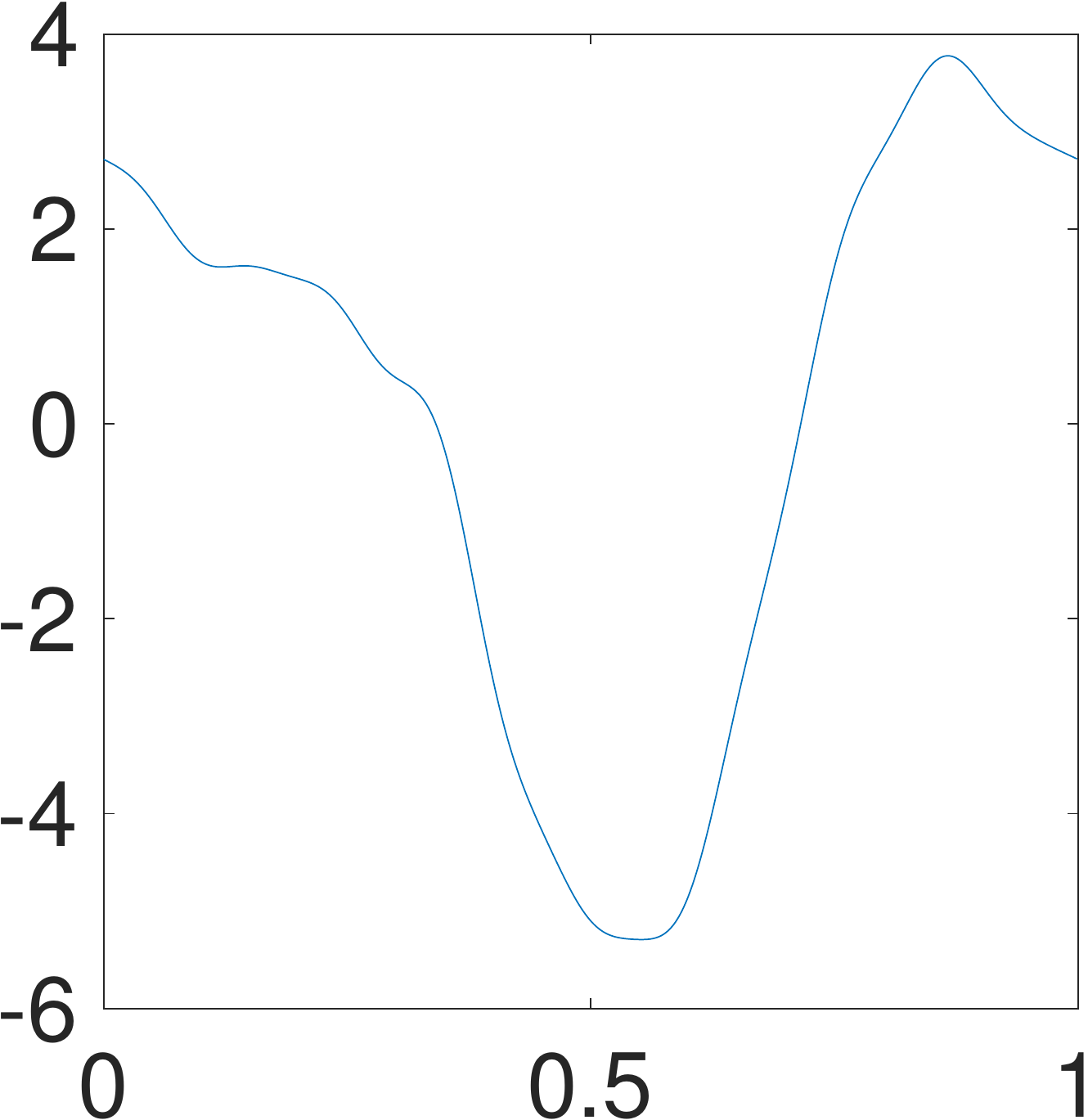}   &
      \includegraphics[width=1.05in]{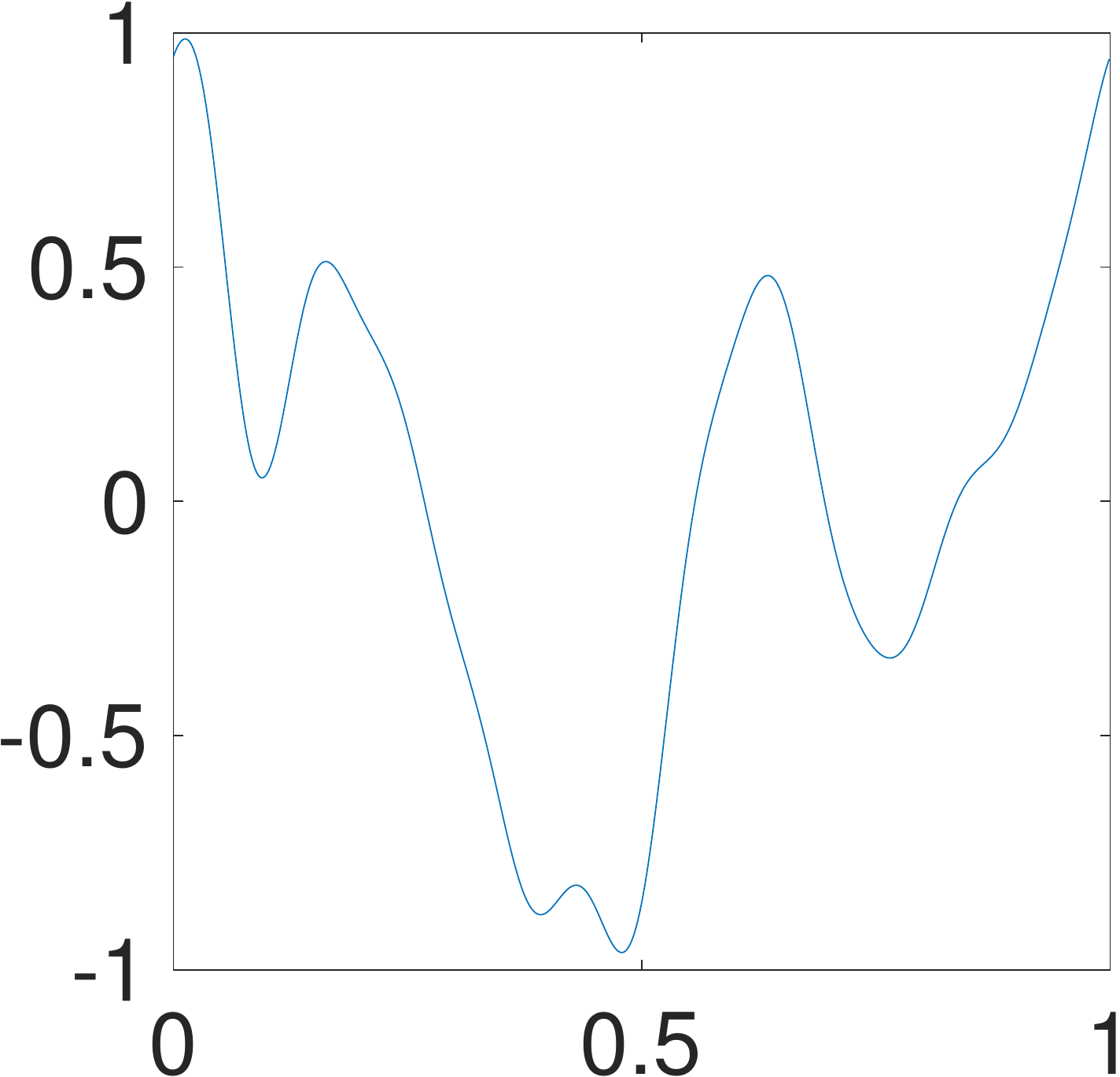}   &
      \includegraphics[width=1.1in]{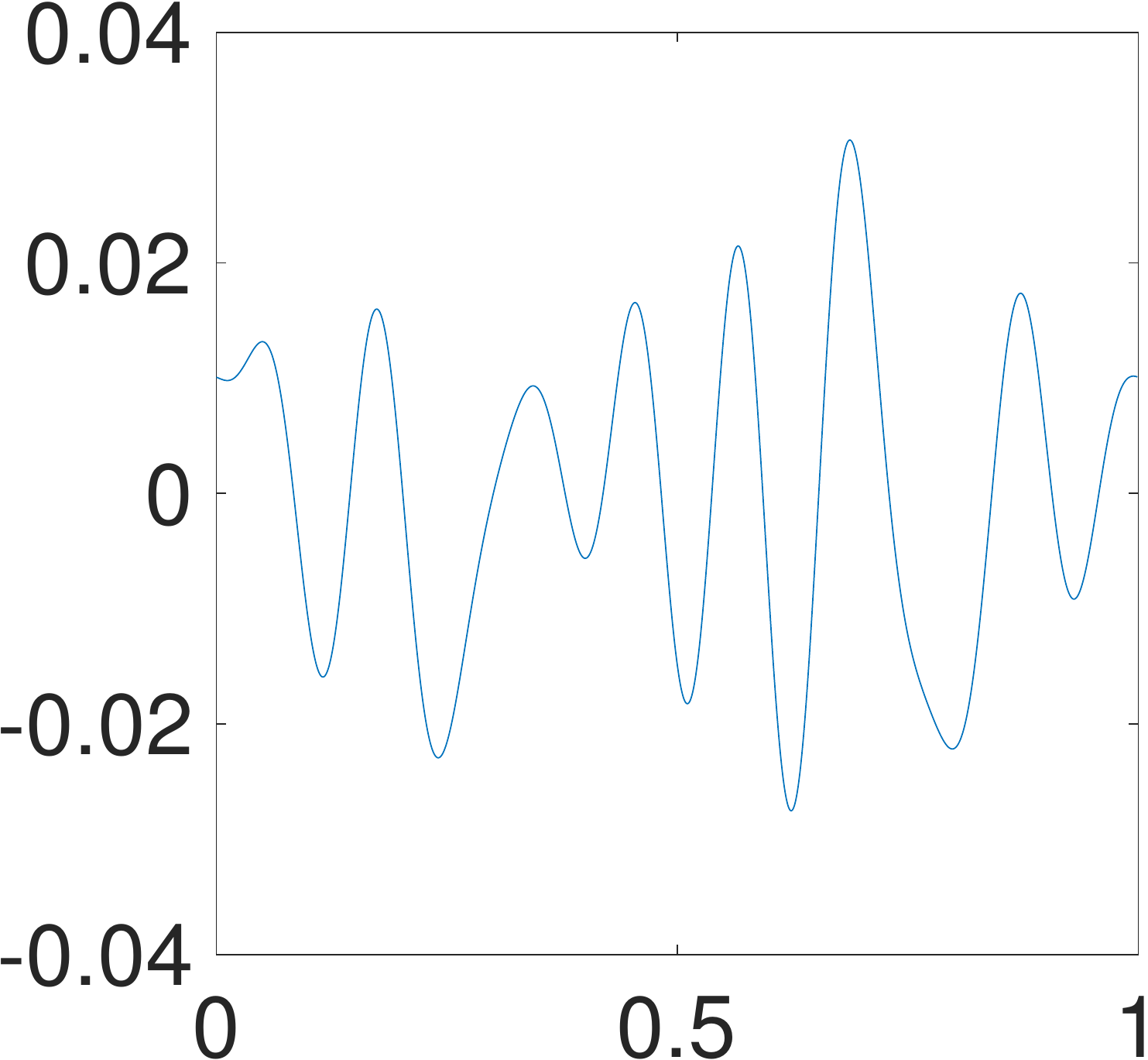}   &
      \includegraphics[width=1.05in]{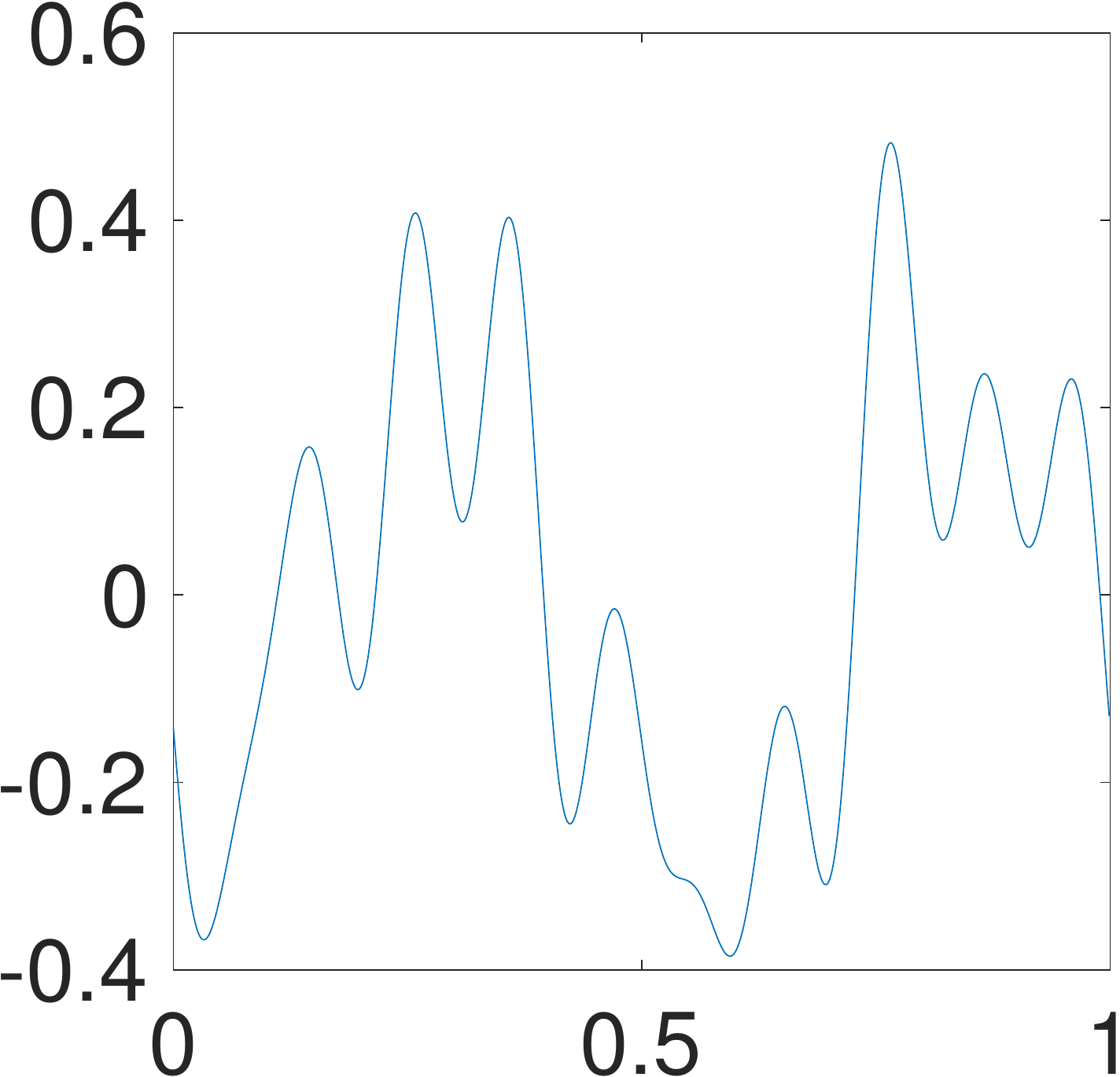}   &
      \includegraphics[width=1in]{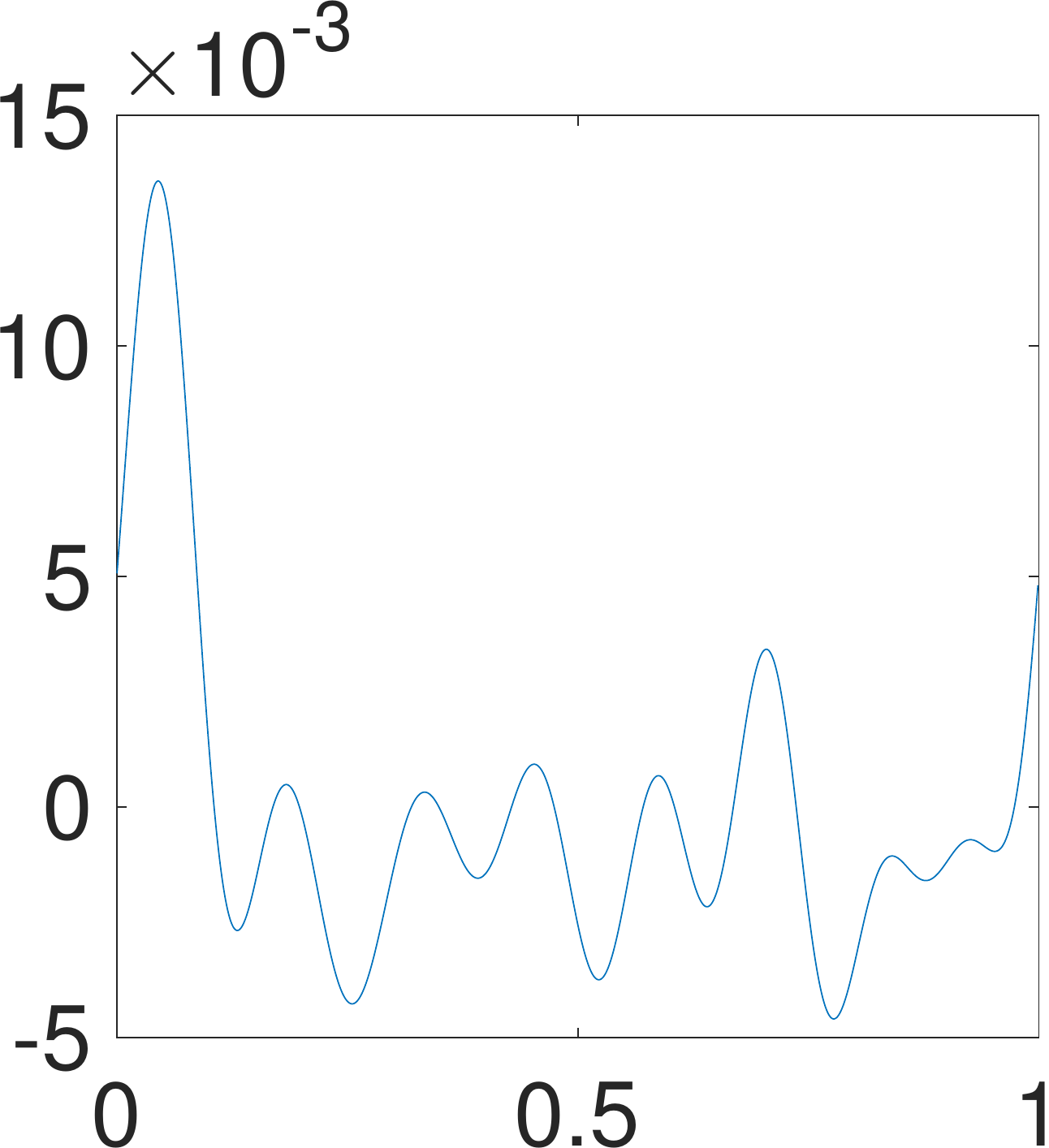}   \\
      \includegraphics[width=1in]{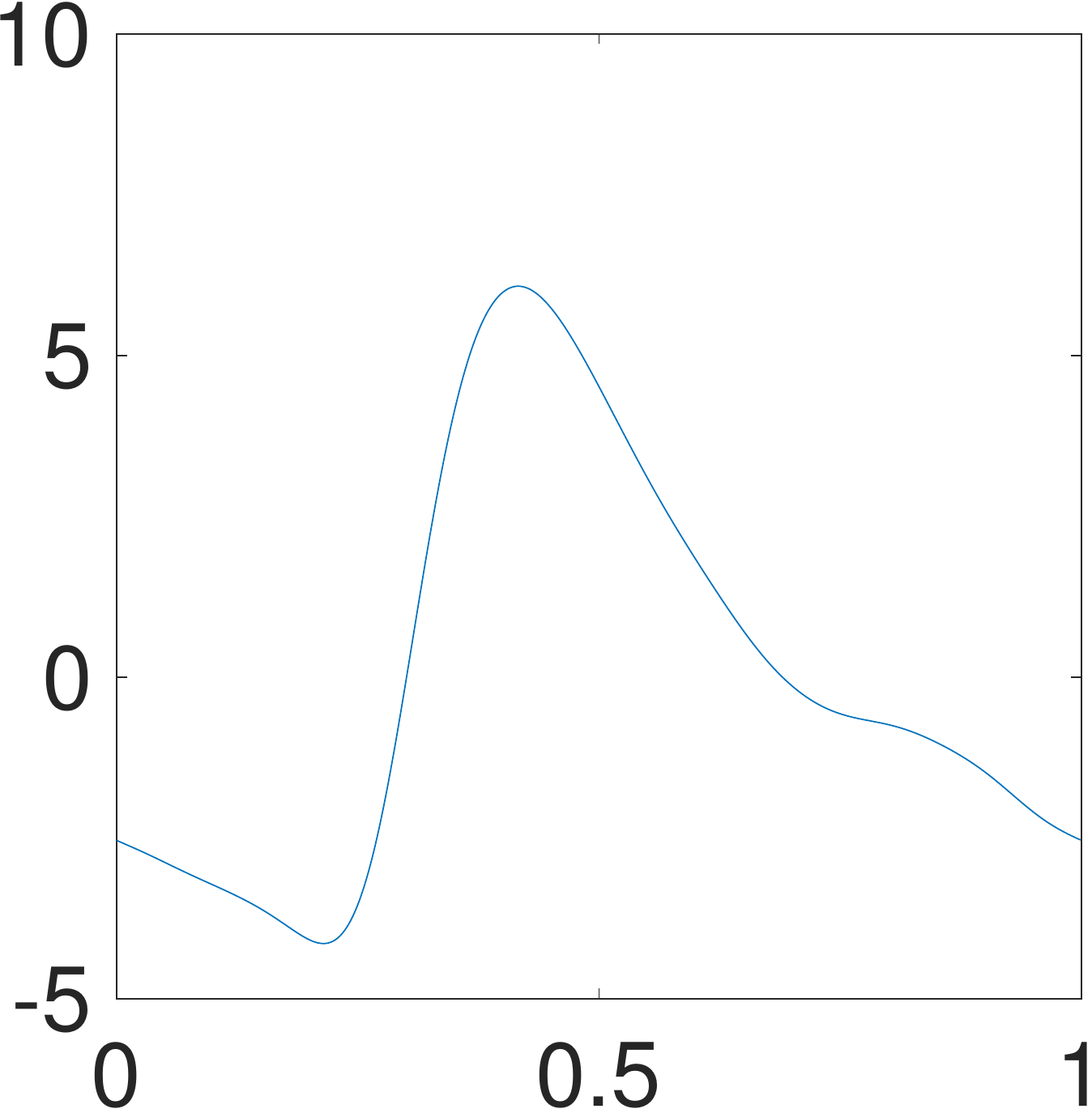}   &
      \includegraphics[width=1.05in]{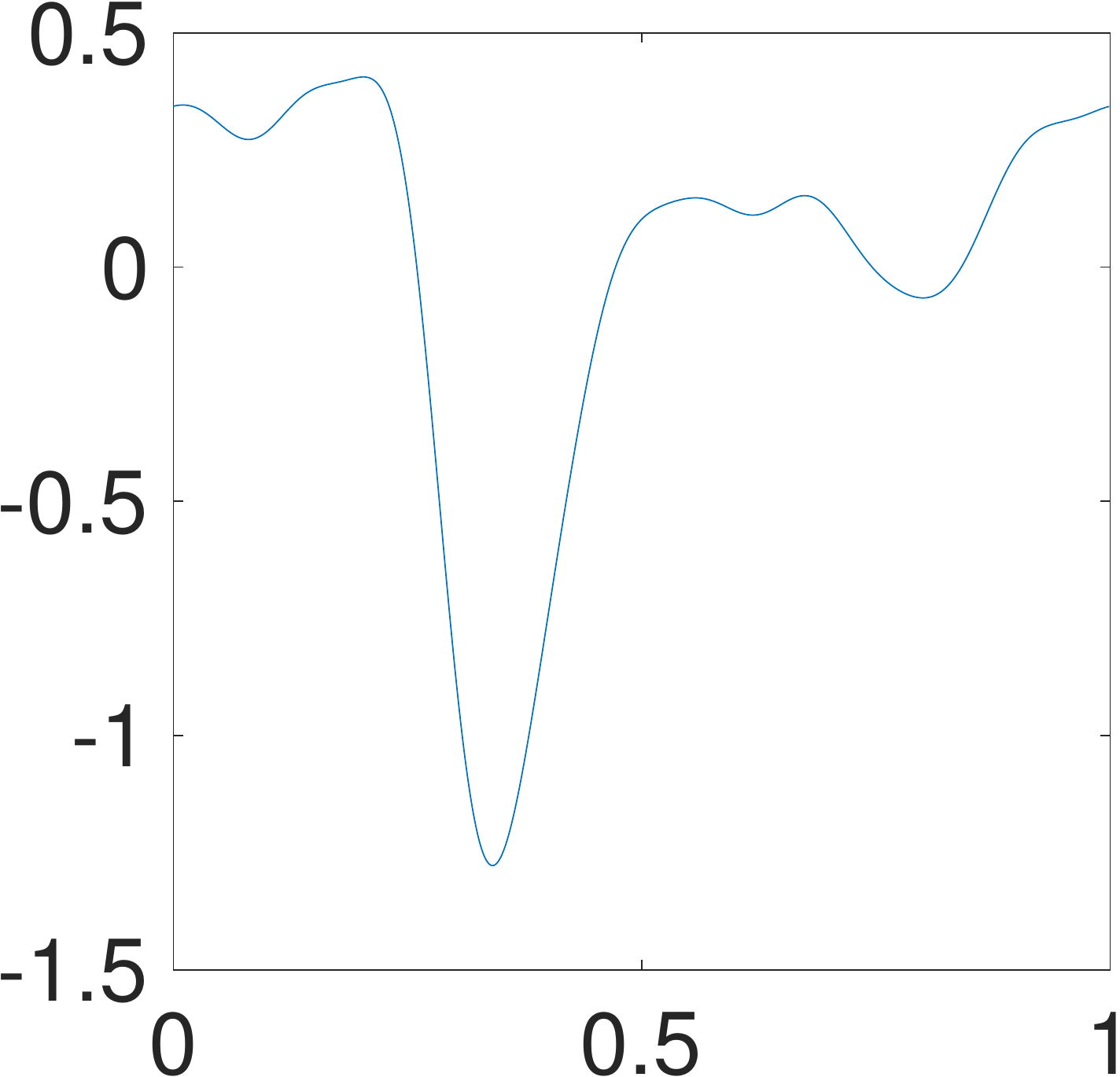}   &
      \includegraphics[width=1.15in]{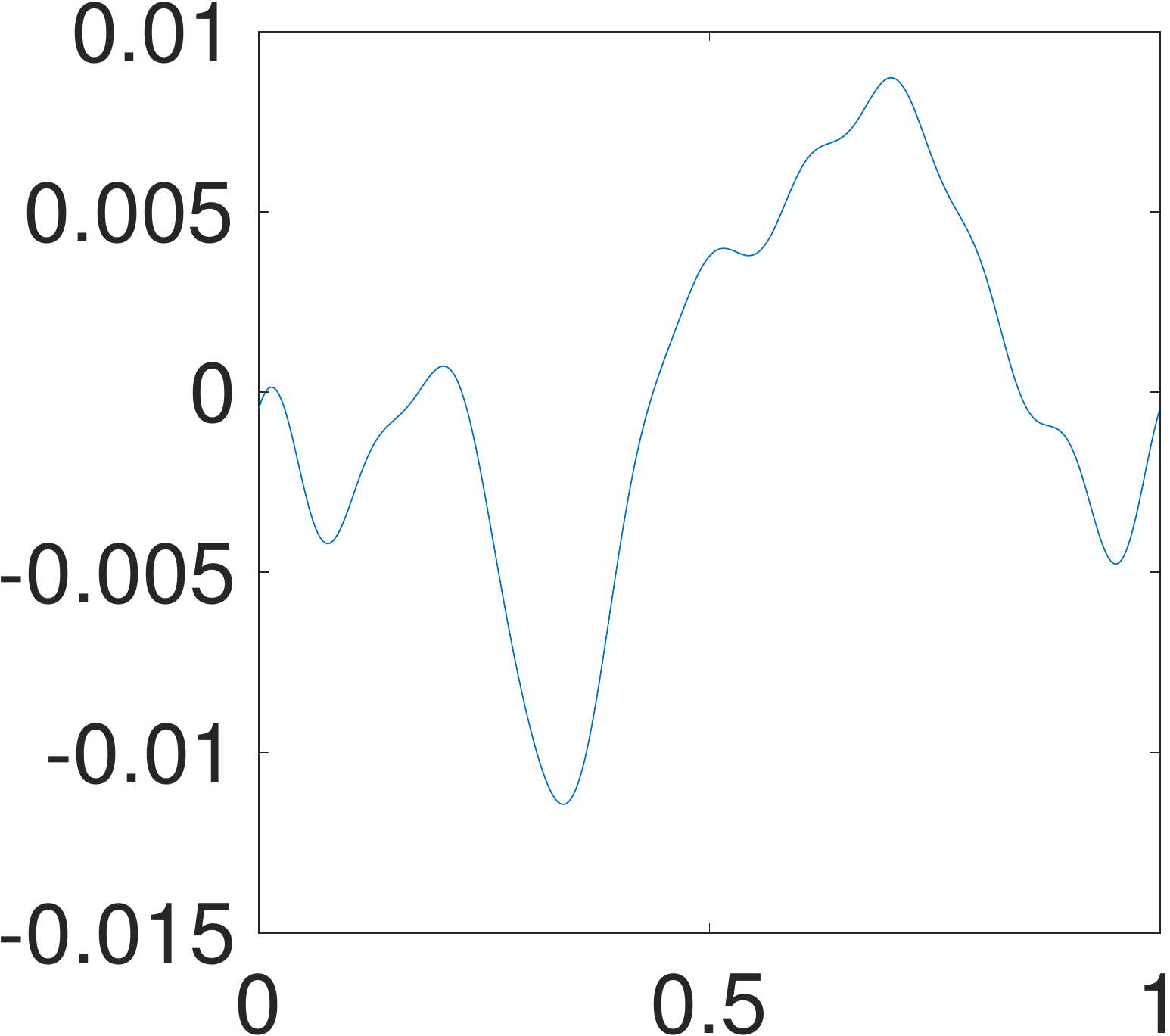}   &
      \includegraphics[width=1.05in]{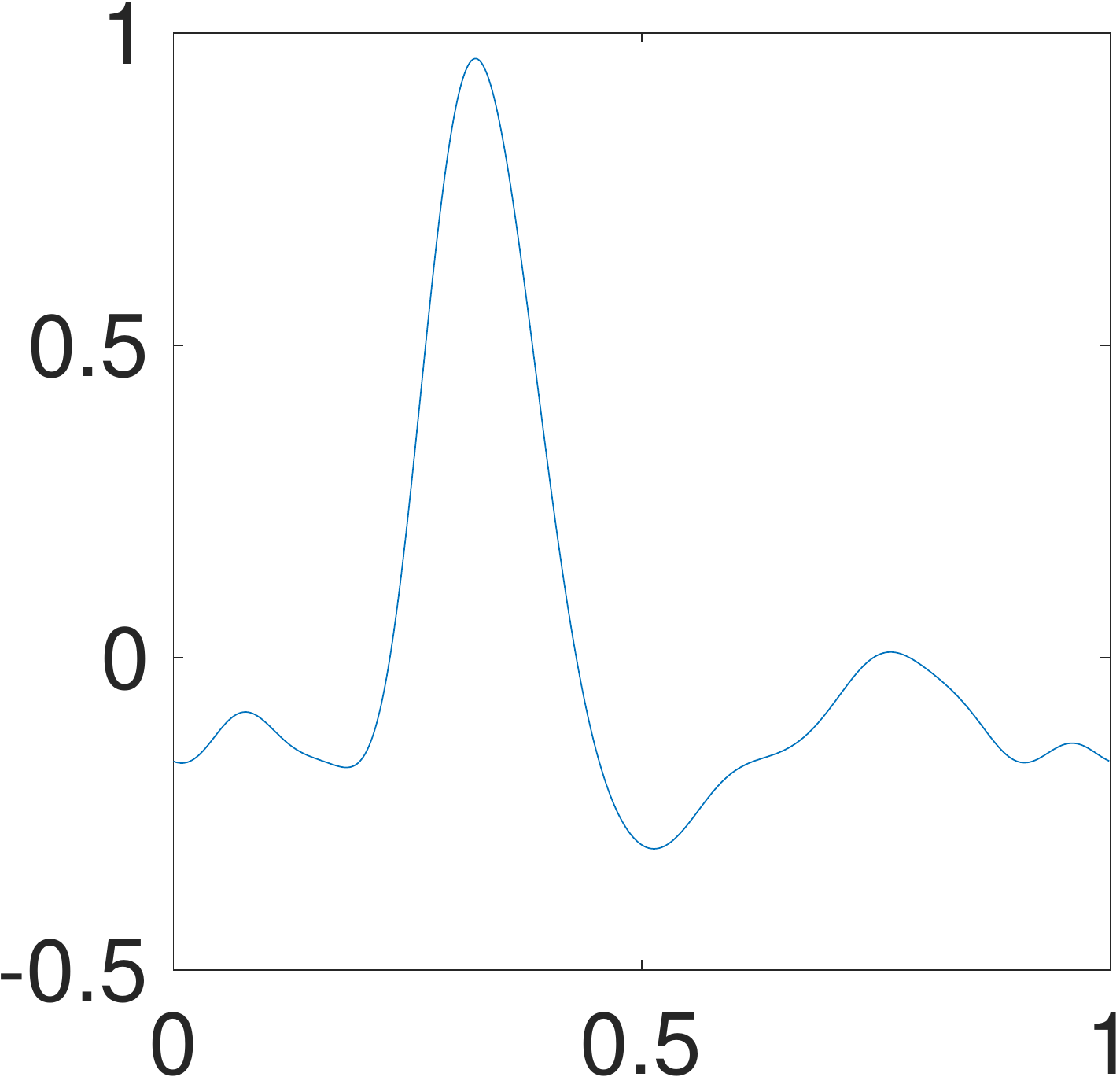}   &
      \includegraphics[width=1in]{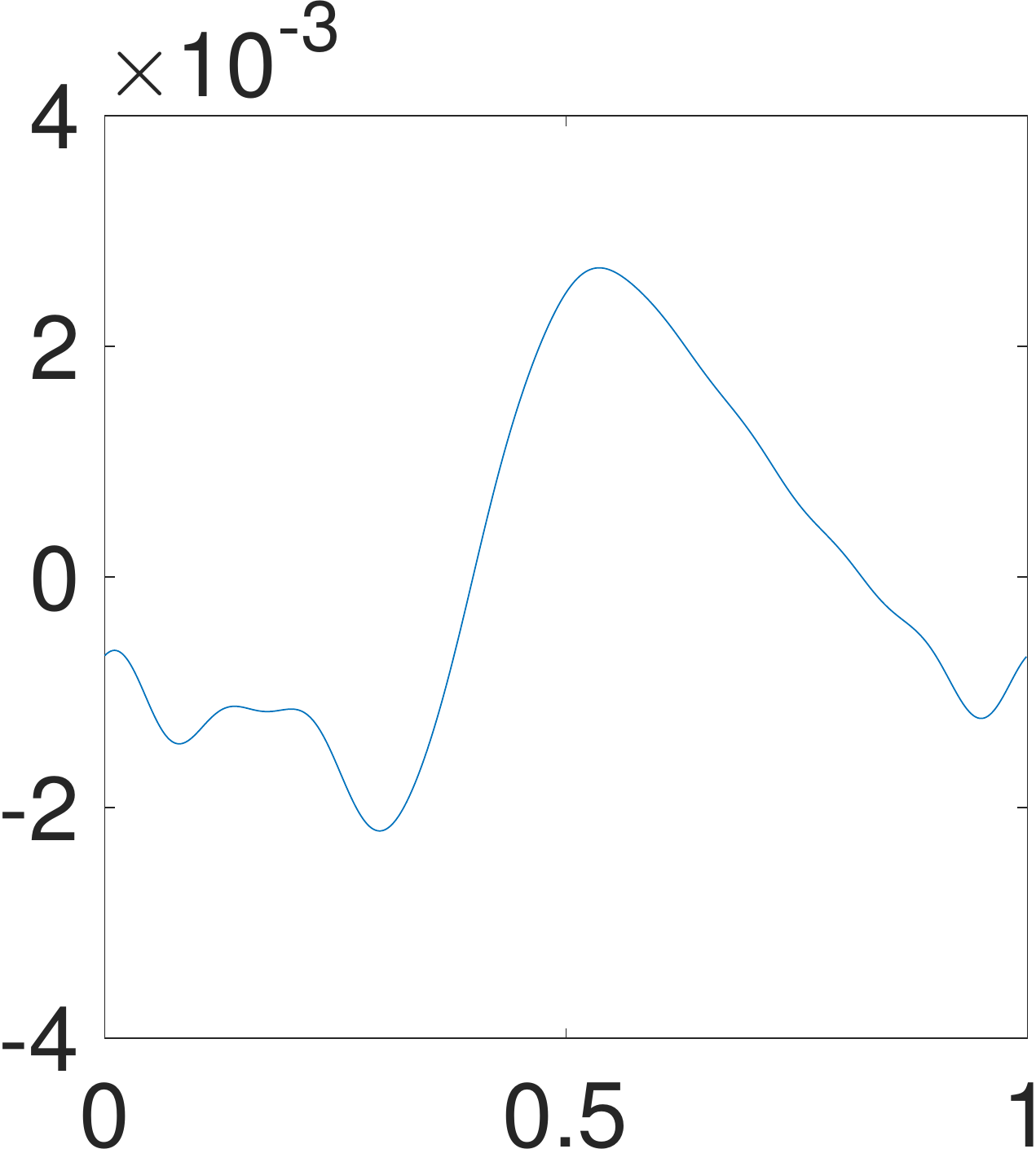}   
    \end{tabular}
  \end{center}
  \caption{Top: estimated shape functions $a_{0,1}s_{c0,1}(2\pi t)$, $a_{1,1}s_{c1,1}(2\pi t)$, $a_{-1,1}s_{c-1,1}(2\pi t)$, $b_{1,1}s_{s1,1}(2\pi t)$, and $b_{-1,1}s_{s-1,1}(2\pi t)$ for the respiratory MIMF. Bottom: estimated shape functions $a_{0,2}s_{c0,2}(2\pi t)$, $a_{1,2}s_{c1,2}(2\pi t)$, $a_{-1,2}s_{c-1,2}(2\pi t)$, $b_{1,2}s_{s1,2}(2\pi t)$, and $b_{-1,2}s_{s-1,2}(2\pi t)$ for the cardiac MIMF.}
\label{fig:15_3}
\end{figure}

\section{Conclusion}
\label{sec:con}

This paper proposed the recursive diffeomorphism-based spectral analysis (RDSA) for the multiresolution mode decomposition. The convergence of RDSA has been theoretically and numerically proved. The computational efficiency is significantly better than the recursive diffeomorphism-based regression in \cite{MMD}. RDSA analyzes oscillatory time series by providing its multiresolution expansion coefficients and shape function series; these features would be more meaningful than those by traditional Fourier analysis and wavelet analysis. As we have seen in numerical examples, these features visualize important variation of signals, which are important for abnormality detection in oscillatory time series. The computational efficiency of RDSA makes the multiresolution mode decomposition a practical model for large-scale time series analysis and online data analysis, e.g, real-time monitoring systems for heart condition.

The fast algorithms proposed in this paper can be naturally extended to higher dimensional spaces for the applications like atomic crystal images in physics \cite{Crystal,LuWirthYang:2016}, art investigation \cite{Canvas,Canvas2}, geology \cite{GeoReview,SSCT,977903}, imaging \cite{4685952}, etc. In higher dimensional spaces, the computational efficiency is a crucial issue. Hence, the extension of RDSA to higher dimensional spaces would be very important.

{\bf Acknowledgments.} 
H.Y. thanks Ingrid Daubechies for her inspiration and discussion. This research is supported by the start-up grant from the Department of Mathematics at the National University of Singapore.

\bibliographystyle{unsrt} 
\bibliography{ref}

\end{document}